\documentclass[a4paper]{article}

\usepackage[T1]{fontenc}
\usepackage[latin9]{inputenc}
\usepackage{geometry}
\usepackage{amsmath}
\usepackage{amsthm}
\usepackage{amssymb}
\usepackage{tikz}
\usepackage[margin=1cm]{caption}
\usepackage{paralist}
\usepackage[unicode=true,pdfusetitle,
 bookmarks=true,bookmarksnumbered=false,bookmarksopen=false,
 breaklinks=false,pdfborder={0 0 0},pdfborderstyle={},backref=false,colorlinks=false]
 {hyperref}
 
\numberwithin{equation}{section}

\makeatletter

\usepackage[nameinlink,capitalise,noabbrev]{cleveref}

\hypersetup{%
    bookmarksnumbered, bookmarksopen=true, bookmarksopenlevel=1,%
}

\theoremstyle{plain}
\newtheorem{thm}{Theorem}[section]
\crefname{thm}{Theorem}{Theorems}
\theoremstyle{plain}
\newtheorem{lem}[thm]{Lemma}
\crefname{lem}{Lemma}{Lemmas}
\theoremstyle{plain}
\newtheorem{cor}[thm]{Corollary}
\theoremstyle{plain}
\newtheorem*{claim*}{Claim}
\crefname{claim}{Claim}{Claims}
\theoremstyle{definition}
\newtheorem{defn}[thm]{Definition}
\crefname{defn}{Definition}{Definitions}
\theoremstyle{plain}
\newtheorem{conjecture}[thm]{Conjecture}
\crefname{conjecture}{Conjecture}{Conjectures}
\theoremstyle{plain}
\newtheorem{prop}[thm]{Proposition}
\crefname{prop}{Proposition}{Propositions}
\theoremstyle{definition}

\theoremstyle{definition}

\theoremstyle{plain}
\newtheorem{claim}[thm]{Claim}
\newtheorem{fact}[thm]{Fact}
\crefformat{equation}{#2(#1)#3}

\author{Jacob Fox\thanks{Department of Mathematics, Stanford University, Stanford, CA 94305.
Email: \href{mailto:jacobfox@stanford.edu} {\nolinkurl{jacobfox@stanford.edu}}.
Research supported by a Packard Fellowship and by NSF Career Award
DMS-1352121.}\and Matthew Kwan\thanks{Department of Mathematics, Stanford University, Stanford, CA 94305.
Email: \href{mailto:mattkwan@stanford.edu} {\nolinkurl{mattkwan@stanford.edu}}.
Research supported in part by SNSF project 178493.}\and Lisa Sauermann\thanks{School of Mathematics, Institute for Advanced Study, Princeton, NJ 08540.
Email: \href{mailto:lsauerma@stanford.edu} {\nolinkurl{lsauerma@stanford.edu}}.}}
\date{}

\crefformat{equation}{#2(#1)#3}

\let\originalleft\left
\let\originalright\right
\renewcommand{\left}{\mathopen{}\mathclose\bgroup\originalleft}
\renewcommand{\right}{\aftergroup\egroup\originalright}
\usepackage{verbatim}

\makeatletter
\renewcommand*{\UrlTildeSpecial}{%
  \do\~{%
    \mbox{%
      \fontfamily{ptm}\selectfont
      \textasciitilde
    }%
  }%
}%
\let\Url@force@Tilde\UrlTildeSpecial
\makeatother

\makeatother

\begin{document}
\title{Anticoncentration for subgraph counts in random graphs}

\maketitle
\global\long\def\RR{\mathbb{R}}%
\global\long\def\QQ{\mathbb{Q}}%
\global\long\def\E{\mathbb{E}}%
\global\long\def\Var{\operatorname{Var}}%
\global\long\def\CC{\mathbb{C}}%
\global\long\def\NN{\mathbb{N}}%
\global\long\def\ZZ{\mathbb{Z}}%
\global\long\def\GG{\mathbb{G}}%
\global\long\def\tallphantom{\vphantom{\sum}}%
\global\long\def\tallerphantom{\vphantom{\int}}%
\global\long\def\supp{\operatorname{supp}}%
\global\long\def\one{\boldsymbol{1}}%
\global\long\def\d{\operatorname{d}}%
\global\long\def\Unif{\operatorname{Unif}}%
\global\long\def\Po{\operatorname{Po}}%
\global\long\def\Bin{\operatorname{Bin}}%
\global\long\def\Ber{\operatorname{Ber}}%
\global\long\def\Geom{\operatorname{Geom}}%
\global\long\def\Rad{\operatorname{Rad}}%
\global\long\def\floor#1{\left\lfloor #1\right\rfloor }%
\global\long\def\ceil#1{\left\lceil #1\right\rceil }%
\global\long\def\cond{\,\middle|\,}%
\global\long\def\g{\boldsymbol{\gamma}}%
\global\long\def\x{\boldsymbol{\xi}}%
\global\long\def\su{\subseteq}%
\global\long\def\sm{\!\setminus\!}

\global\long\def\ls#1{\textcolor{blue}{\textbf{[LS comments:} #1\textbf{]}}}
\global\long\def\mk#1{\textcolor{orange}{\textbf{[MK comments:} #1\textbf{]}}}

\begin{abstract}\noindent
Fix a graph $H$ and some $p\in (0,1)$, and let $X_H$ be the number of copies of $H$ in a random graph $\GG(n,p)$. Random variables of this form have been intensively studied since the foundational work of Erd\H os and R\'enyi. There has been a great deal of progress over the years on the large-scale behaviour of $X_H$, but the more challenging problem of understanding the small-ball probabilities has remained poorly understood until now. More precisely, how likely can it be that $X_H$ falls in some small interval or is equal to some particular value? In this paper we prove the almost-optimal result that if $H$ is connected then for any $x\in \NN$ we have $\Pr(X_H=x)\le n^{1-v(H)+o(1)}$. Our proof proceeds by iteratively breaking $X_H$ into different components which fluctuate at ``different scales'', and relies on a new anticoncentration inequality for random vectors that behave ``almost linearly''.


\end{abstract}

\section{Introduction}

Let $\GG\left(n,p\right)$ be the binomial random graph model, where
we fix a set of $n$ vertices and include each of the $\binom{n}{2}$
possible edges independently with probability $p$. For graphs $H$
and $G$, let $X_{H}\left(G\right)$ be the number of subgraphs of
$G$ isomorphic to $H$, so that if $G\sim\GG\left(n,p\right)$ then
we can interpret $X_{H}$ as the random variable that counts the number
of copies of $H$ in a random graph.

These subgraph-counting random variables and their distributions are
central objects of study in the theory of random graphs, going back
to the foundational work of Erd\H os and R\'enyi~\cite{ER60}.
Early work~\cite{Bol81,ER60,RV85} concerned\emph{ }existence of
subgraphs: fixing $H$, for which $n$ and $p$ is it likely that
$X_{H}>0$, and for which $n$ and $p$ is it likely that $X_{H}=0$?
It turns out that there is a \emph{threshold} value of $p$ (as a
decaying function of $n$) that cleanly separates these two behaviours,
and further work~\cite{Bol81,BW89,ER60,KR83,RV85,Sch79} focused
on investigating the (Poisson-type) distribution of $X_{H}$ near
this threshold.

In this paper, we are more interested in the behaviour far above this existence threshold (when $p$ is a constant
independent of $n$). When appropriately normalised, $X_{H}$ has an asymptotically\footnote{All asymptotics, here and for the rest of the paper, are as $n\to\infty$, while $p$ is fixed.}
normal distribution (this was proved by Nowicki and Wierman~\cite{NW88}
and Ruci\'nski~\cite{Ruc88}, following several results~\cite{Bar82,Kar84,KR83}
pushing increasingly further past the existence threshold). Further work by Barbour, Karo\'nski and
Ruci\'nski~\cite{BKR89} provided quantitative bounds on the rate of
convergence to the normal distribution. However,
this asymptotic normality only characterises the ``large-scale''
behaviour of the distribution of $X_{H}$, and is basically due to
the fact that $X_{H}$ closely correlates with the number of edges
in $\GG\left(n,p\right)$.

A more challenging direction of research is to understand ``local'' aspects of
the distributions of these subgraph-counting random variables\footnote{We remark that another completely different challenging direction of research is the study of \emph{large deviations} of subgraph counts in random graphs. Recently, there has been a lot of progress in this area, see for example the monograph \cite{Cha17}, the recent papers \cite{Aug20,BB19,BGLZ17,CD20,HMS19}, and the references therein.}. In the past decade, there have been a number of advances in this direction. Following work by Loebl, Matou\v sek and Pangr\'ac~\cite{PML04} for the case where $H$ is a triangle, it
was proved by Kolaitis and Kopparty~\cite{KK13} (see also \cite{DKR15})
that if we fix some $p\in\left(0,1\right)$, some prime $q\in\NN$
and some connected graph $H$ with at least one edge, then $X_{H}$
mod $q$ has an asymptotically uniform distribution on $\left\{ 0,\dots,q-1\right\} $.
More recently, \emph{local central
limit theorems} have begun to emerge, giving asymptotic formulas for
the point probabilities $\Pr\left(X_{H}=x\right)$ in terms of a normal
density function. Such a theorem was first proved for the case where
$H$ is a triangle by Gilmer and Kopparty~\cite{GK16} (see also
\cite{Ber16}), and this was extended by Berkowitz~\cite{Ber18}
to the case where $H$ is any clique.

In this paper we are concerned with a somewhat looser question: what can be said about the \emph{anticoncentration}
behaviour of $X_{H}$? Roughly speaking, this is asking for uniform
upper bounds on the point probabilities $\Pr\left(X_{H}=x\right)$,
or more generally on the small ball probabilities $\Pr\left(X_{H}\in I\right)$,
where $I$ is an interval of prescribed length. Meka, Nguyen and Vu~\cite{MNV16} developed some general polynomial anticoncentration inequalities, and used the polynomial
structure of $X_{H}$ to prove the bound $\Pr\left(X_{H}=x\right)\le n^{-1+o\left(1\right)}$
for constant $p\in\left(0,1\right)$ and any $H$ that contains at
least one edge. In \cite{FKS1} we proposed the following conjecture.
\begin{conjecture}
\label{conj:general}Fix $p\in\left(0,1\right)$ and fix a graph $H$
with no isolated vertices. Then
\[
\max_{x\in\NN}\Pr\left(X_{H}=x\right)=O\left(n^{1-v\left(H\right)}\right),
\]
where $v(H)$ is the number of vertices of $H$.
\end{conjecture}

The motivation for \cref{conj:general} is that if $p$ is fixed then
$\Var X_{H}=\Theta\left(n^{2v\left(H\right)-2}\right)$ (see for example \cite[Lemma~3.5]{JLR00}), and the aforementioned
asymptotic normality therefore implies that $X_{H}$ is concentrated
on an interval of length $\Theta\left(n^{v\left(H\right)-1}\right)$.
Provided that the distribution of $X_{H}$ is sufficiently ``smooth'',
we should expect each value in this interval to have comparable probability.
Note that this line of reasoning implies that \cref{conj:general},
if true, is best possible: any stronger bound would contradict Chebyshev's inequality. Also, observe that
the assumption that $H$ has no isolated vertices is necessary: if
$H'$ is obtained from $H$ by removing isolated vertices then $X_{H}$
is a deterministic multiple of $X_{H'}$, so inherits its point probabilities.

We also remark that while Meka, Nguyen and Vu were the first to explicitly
consider anticoncentration of subgraph counts in general, actually
Pangr\'ac, Matou\v sek and Loebl~\cite{PML04} considered anticoncentration
of the triangle-count $X_{K_{3}}$ more than ten years earlier: a
primary motivation for their aforementioned work on triangle-counts
mod $q$ was to show that the point probabilities $\Pr\left(X_{K_{3}}=x\right)$
are small (they gave a bound of $O\left(1/\log n\right)$), which
in turn implies that two independent copies of $\GG\left(n,p\right)$
are unlikely to have the same Tutte polynomial. In addition, many
of the other aforementioned results concerning the distribution of
$X_{H}$ imply anticoncentration bounds: any central limit theorem
already implies that $\max_x\Pr\left(X_{H}=x\right)=o\left(1\right)$, and one can\footnote{The central limit theorem of Barbour, Karo\'nski and Ruci\'nski
was not stated in a way that allows one to directly read off
estimates for probabilities regarding $X_{H}$. But, it is
possible to deduce such an estimate with the method of \cite[Proposition 1.2.2]{Ros11}.} deduce from the quantitative central limit theorem of 
Barbour, Karo\'nski and Ruci\'nski~\cite{BKR89} that $\Pr\left(X_{H}=x\right)\le\Pr\left(\left|X-x\right|\le n^{v\left(H\right)-3/2}\right)=O\left(1/\sqrt{n}\right)$. The local central limit theorem of Berkowitz~\cite{Ber18} definitively
settles the matter in the case where $H=K_{h}$ is an $h$-vertex clique, in which
case it actually gives the asymptotically optimal bound $\Pr\left(X_{H}=x\right)\le\left(2\pi\Var X_{H}\right)^{-1/2}+o\left(n^{1-h}\right)=O(n^{1-h})$.

In \cite{FKS1} we used ideas related to Erd\H os' combinatorial
proof of the Erd\H os--Littlewood--Offord theorem (see \cite{Erd45}) to give a simple
proof of the general bound $\Pr\left(X_{H}=x\right)\le\Pr\left(\left|X_{H}-x\right|\le n^{v\left(H\right)-2}\right)=O\left(1/n\right)$,
and we showed how to extend these methods to prove the sharper bound
$\Pr\left(X_{K_{h}}=x\right)=n^{1-h+o\left(1\right)}$ in the case
where $H=K_{h}$ is a clique. In this paper we develop these methods
much further, proving an approximate version of \cref{conj:general}
for all connected $H$.
\begin{thm}
\label{thm:subgraph-counts}Fix $p\in\left(0,1\right)$ and fix a connected
graph $H$. Then
\[
\max_{x\in\ZZ}\Pr\left(X_{H}=x\right)=n^{1-v\left(H\right)+o\left(1\right)}.
\]
\end{thm}

Concerning the $o(1)$-term in \cref{thm:subgraph-counts}, the arguments in our proof yield a bound for the $o(1)$-term that decays extremely slowly as $n$ goes to infinity. In order to simplify the presentation of the proof and to avoid additional technical details, we decided to write the proof in a way that does not give any explicit bounds for the $o(1)$-term.

The general idea for the proof of \cref{thm:subgraph-counts} is to break up $X_H$ into different components that fluctuate at ``different scales'' and handle each component separately. In the case where $H$ is a clique, this plan is relatively simple to execute, but in the more general setting of \cref{thm:subgraph-counts} there are a number of additional challenges that must be overcome. We discuss these in \cref{sec:outline}. Our proof has a number of new ingredients; one that is perhaps worth highlighting is a combinatorial anticoncentration
inequality for vector-valued random variables that behave ``almost
linearly'', in the spirit of some anticoncentration theorems due
to Hal\'asz. The details are in \cref{sec:halasz}.

\subsection{Basic definitions and notation}

We use standard graph-theoretic notation throughout. In particular, the vertex and edge sets of a graph $G$ are denoted by $V(G)$ and $E(G)$, and the sizes of these sets are denoted by $e(G)$ and $v(G)$. For disjoint vertex sets $A,B$ in a graph $G$, we write $e_G(A)$ for the number of edges $e(G[A])$ inside $A$, and we write $e_G(A,B)$ for the number of edges between $A$ and $B$. Abusing notation, for a vertex $v$ we write $e_G(v,A)$ to mean $e_G(\{v\},A)$, which is the size of the $A$-neighbourhood of $v$ in $G$. We write $N(v)$ for the neighbourhood of $v$ (that is, the set of vertices adjacent to $v$). A homomorphism $\phi$ from a multigraph $H$ to a multigraph $G$ consists of a map from the vertices of $H$ to the vertices of $G$, and a map from the edges of $H$ to the edges of $G$, such that whenever $e$ is an edge of $H$ between vertices $x$ and $y$, the image $\phi(e)$ is an edge between the vertices $\phi(x)$ and $\phi(y)$.

We initially introduced $X_H$ as the random variable that counts unlabelled copies of $H$ (as is standard in this area), but for the proof it will be slightly more convenient to redefine $X_H$ to count the number of \emph{labelled} copies of $H$ (injective homomorphisms from $H$ into $G$). The labelled/unlabelled distinction is irrelevant for \cref{thm:subgraph-counts}, because these two counts differ by a fixed multiplicative factor (the number of automorphisms of $H$).

We use standard asymptotic notation throughout. For functions $f=f\left(n\right)$
and $g=g\left(n\right)$ we write $f=O\left(g\right)$ to mean that there
is a constant $C$ such that $\left|f\right|\le C\left|g\right|$,
we write $f=\Omega\left(g\right)$ to mean there is a constant $c>0$
such that $f\ge c\left|g\right|$ for sufficiently large $n$, we
write $f=\Theta\left(g\right)$ to mean that $f=O\left(g\right)$
and $f=\Omega\left(g\right)$, and we write $f=o\left(g\right)$ or
$g=\omega\left(f\right)$ to mean that $f/g\to0$ as $n\to\infty$. Unless stated otherwise, all asymptotics are as $n\to\infty$ (all other variables should be viewed as constant).

We will use notation of the form $\E_{G}$ to indicate an expected value with respect to a random choice of $G$ (if there are other sources of randomness, then formally this is a conditional expected value). We write $\Ber(p)$ for the $p$-Bernoulli distribution, meaning that if $\xi\sim\Ber(p)$ then $\Pr(\xi=1)=p$ and $\Pr(\xi=0)=1-p$. Finally, we write $\NN$ for the set of non-negative integers, we write $[n]$ for the set $\{1,\dots,n\}$, and all logarithms are to base $e$.

\section{Discussion and main ideas of the proof}\label{sec:outline}

Before discussing the new ideas in the proof of \cref{thm:subgraph-counts},
it is worth reviewing the proofs in \cite{FKS1} giving weaker anticoncentration
bounds. We will build on these ideas to prove \cref{thm:subgraph-counts}.

First, to prove the bound $\Pr\left(\left|X_{H}-x\right|\le n^{v\left(H\right)-2}\right)=O\left(1/n\right)$,
the key observation was that $X_{H}$ is an ``almost linear'' function
of the edges of the underlying random graph $G\sim\GG\left(n,p\right)$.
Specifically, for a pair of vertices $e=\left\{ x,y\right\} $, the
difference $\Delta X_{H}:=X_{H}\left(G+e\right)-X_{H}\left(G-e\right)$
is tightly concentrated around its expectation $\E\Delta X_{H}=\Theta\left(n^{v\left(H\right)-2}\right)$,
meaning that adding or removing an edge typically causes $X_{H}$
to increase or decrease by about this amount\footnote{This observation is closely related to the fact that the \emph{number}
of edges in $G$ is closely correlated with $X_{H}$. This fact can
be used to prove a central limit theorem for $X_{H}$ (see for example
\cite[Example~6.4]{JLR00}).}. Using this observation, and some ideas from Erd\H os' proof of
the Erd\H os--Littlewood--Offord theorem~\cite{Erd45} and Lubell's
proof of the LYM inequality~\cite{Lub66}, it is possible to prove
that the anticoncentration behaviour of $X_{H}/\E\Delta X_{H}$ is
about the same as the anticoncentration behaviour of the number of edges of $G$, which is a binomial random
variable with parameters $\binom{n}{2}$ and $p$. This gives the
desired bound, which in some sense gives ``coarse scale'' anticoncentration for $X_H$.

Second, to prove the bound $\Pr\left(X_{K_{h}}=x\right)\le n^{1-h+o\left(1\right)}$
for cliques, the key idea was to fix a vertex $v$ and write $X_{K_{h}}\left(G\right)=X_{K_{h}}\left(G-v\right)+X_{K_{h-1}}\left(G\left[N\left(v\right)\right]\right)$.
That is, the number of $h$-cliques in $G$ is the same as the number
of $h$-cliques in $G-v$, plus the number of $\left(h-1\right)$-cliques
in the neighbourhood of $v$ (which yield an $h$-clique when combined
with $v$). Now, it is possible to use similar ideas as in the preceding paragraph to show that $X_{K_{h}}\left(G-v\right)$ is anticoncentrated at a coarse scale. On the other hand, $X_{K_{h-1}}\left(G\left[N\left(v\right)\right]\right)$ has a much smaller order of magnitude, and it is actually concentrated on a relatively small interval around its expectation. Furthermore, we can bound the point probabilities for $X_{K_{h-1}}\left(G\left[N\left(v\right)\right]\right)$ inductively.

Then, roughly speaking, the idea was as follows: We want to bound the probability of having $X_{K_{h}}\left(G\right)=X_{K_{h}}\left(G-v\right)+X_{K_{h-1}}\left(G\left[N\left(v\right)\right]\right)=x$. Since $X_{K_{h-1}}\left(G\left[N\left(v\right)\right]\right)$ is concentrated on a small interval around its expectation $\E X_{K_{h-1}}\left(G\left[N\left(v\right)\right]\right)$, in order to have $X_{K_{h}}\left(G\right)=x$ the value of $X_{K_{h}}\left(G-v\right)$ must (typically) be reasonably close to $x-\E X_{K_{h-1}}\left(G\left[N\left(v\right)\right]\right)$. Using the coarse scale anticoncentration bounds for $X_{K_{h}}\left(G-v\right)$, we can bound the probability that this happens. After knowing the value $X_{K_{h}}\left(G-v\right)$, we know what value $X_{K_{h-1}}\left(G\left[N\left(v\right)\right]\right)$ needs to take to have $X_{K_{h}}\left(G\right)=x$. By induction, we can bound the probability that $X_{K_{h-1}}\left(G\left[N\left(v\right)\right]\right)$
takes this particular value.

If $X_{K_{h}}\left(G-v\right)$ and $X_{K_{h-1}}\left(G\left[N\left(v\right)\right]\right)$
were independent, it would be easy to conclude the desired anticoncentration bound $\Pr\left(X_{K_{h}}=x\right)\le n^{1-h+o\left(1\right)}$;
we would be able to simply multiply the above two probability estimates.
Unfortunately, these random variables are not independent, so we need
to rule out the possibility that the fluctuations in $X_{K_{h}}\left(G-v\right)$
``cancel out'' the fluctuations in $X_{K_{h-1}}\left(G\left[N\left(v\right)\right]\right)$
in a way that causes $X_{K_{h}}$ to concentrate on a particular value.
The approach we took was to show that actually $X_{K_{h-1}}\left(G\left[N\left(v\right)\right]\right)$
is anticoncentrated even after conditioning on a typical outcome
of $G-v$ (after which the only remaining
randomness comes from the set of neighbours $N\left(v\right)$).

We proved a conditional anticoncentration bound of this type by induction, using a moment argument. To be more specific, we viewed the conditional probabilities $\Pr\left(X_{K_{h-1}}\left(G\left[N\left(v\right)\right]\right)=z\,\middle|\,G-v\right)$ as random variables depending on $G-v$. To study these random variables we studied their high moments, which essentially comes down to considering collections of different candidates for the neighborhood $N(v)$, and bounding the probability that for each of these candidates we simultaneously have $X_{K_{h-1}}\left(G\left[N\left(v\right)\right]\right)=z$. We accomplished this with a multiple exposure argument: we iteratively went through our candidate sets for $N(v)$ and exposed the
edges of $G-v$ inside each set which had not yet been exposed before. Using a suitable induction hypothesis, and the ideas sketched earlier, at each step we can bound the probability that the corresponding candidate set for $N(v)$ gives rise to the desired value of $X_{K_{h-1}}\left(G\left[N\left(v\right)\right]\right)$. This way we obtained a suitable bound for the moment argument.

Now, there are several obstacles that need to be overcome to generalise
the above argument beyond the case where $H$ is a clique. First,
the decomposition $X_{K_{h}}=X_{K_{h}}\left(G-v\right)+X_{K_{h-1}}\left(G\left[N\left(v\right)\right]\right)$
was very convenient for us: it allowed us to consider two separate
random variables, one of which can be studied on a ``coarse'' scale
using our Littlewood--Offord type techniques, and the other of which
can be studied inductively. Actually, it is not a huge problem to
generalise this decomposition. In general, we have $X_{H}=X_{H}\left(G-v\right)+X_{H}^{v}$,
where $X_{H}^{v}$ is the number of copies of $H$ in $G$ which contain
$v$. One can check that $X_{H}^{v}$ is a certain sum of weighted
subgraph counts in $G-v$, where each of the subgraphs has $v\left(H\right)-1$
vertices, and the weight of a subgraph depends on its intersection
with the neighbourhood $N\left(v\right)$ of $v$. So, if we generalise
the induction hypothesis to certain weighted sums of subgraph counts,
it is still possible to control the anticoncentration of $X_{H}^{v}$
inductively.

The second main obstacle, which is more serious, concerns the multiple-exposure
argument we used to show that $X_{K_{h-1}}\left(G\left[N\left(v\right)\right]\right)$
is anticoncentrated even after conditioning on a typical outcome of
$G-v$. This crucially depended on the fact that in order to know
the value of $X_{K_{h-1}}\left(G\left[N\left(v\right)\right]\right)$
given a particular candidate for $N\left(v\right)$, the only edges
of $G-v$ that we need to expose are the edges inside $N\left(v\right)$
(leaving the remaining edges for future rounds of exposure). Unfortunately,
in general one may need to examine all the edges of $G-v$ to determine
the value of $X_{H}^{v}$, even if we fix a candidate for $N\left(v\right)$.
Specifically, this is the case whenever $H$ is not a complete multipartite
graph (if $H$ is complete multipartite, then we do not need to expose
the edges which lie completely outside $N\left(v\right)$, and actually
in this special case it is not too hard to extend the proof in \cite{FKS1}
to prove \cref{conj:general}).

In the general case where $H$ is not a complete multipartite graph, in order to use a moment argument as above,
we need some other way to estimate the joint probability that many
candidates for $N\left(v\right)$ each result in a specific outcome
of $X_{H}^{v}$. Write $X_{H}^{v}\left(N\left(v\right),G-v\right)$
to indicate the dependence of $X_{H}^{v}$ on both $N\left(v\right)$
and $G-v$. For a collection of sets $A_{1},\dots,A_{t}$ as candidates for $N\left(v\right)$, we want to control the joint probability that all of $X_{H}^{v}\left(A_{1},G-v\right),\dots,X_{H}^{v}\left(A_{t},G-v\right)$ are equal to a given value. Since we cannot consider these random variables separately anymore (as we did before with the multiple exposure argument), we need to somehow modify the induction hypothesis to handle this joint probability.

Specifically, we can generalise to a statement about joint anticoncentration probabilities of random variables of the following type (thus strengthening the induction hypothesis). Take a sequence of distinct vertices $v_1, \dots,v_g$ and for each $v_i$, consider some collection of $t_i$ different candidates for the neighbourhood of $v_i$. Then, consider the $T=t_1\dotsm t_g$ different random variables obtained by making different choices for the neighbourhoods for each of $v_1, \dots,v_g$, and considering the number of copies of $H$ which contain all of
$v_1,\dots,v_g$, conditioned on $v_1,\dots,v_g$ having these neighbourhoods. The idea is that at each step of the induction we introduce a new vertex $v_j$, and we consider many candidates for the neighbourhood of $v_j$ for a moment argument. Given that we are considering joint probabilities of $T=t_1\dotsm t_g$ random variables, our induction hypothesis needs to give a bound of the form $n^{(g-h+1)T+o(1)}$ on the joint probability that all our random variables are equal to particular values. This can be viewed as an anticoncentration bound for a random \emph{vector} $\boldsymbol X$.

To make the above ideas work, we need multivariate generalisations of some of the ideas described so far. For example, we need an anticoncentration inequality for random \emph{vectors} that are almost-linear in the sense sketched earlier, having the property that adding or removing an edge typically causes
a predictable change in their values. There are some classical anticoncentration inequalities by Hal\'asz~\cite{Hal77} that give the kind of bounds we need,
for random vectors that depend linearly (and ``non-degenerately'') on a sequence of independent random choices. (Some kind of non-degeneracy assumption is necessary, to rule out situations where the random vector is always contained in a proper subspace of smaller dimension). The standard
proofs of Hal\'asz' inequalities are Fourier-analytic, and are not
robust enough to apply to our almost-linear setting, but we were able to find some
combinatorial arguments that apply to our setting, again inspired
by proofs of Erd\H os~\cite{Erd45} and Lubell~\cite{Lub66}. More details are in \cref{sec:halasz}.

In order to use the estimates in \cref{sec:halasz}, we
need to check a non-degeneracy condition: basically, we need to consider the effects of changing the status of various edges, and we need to show that the corresponding changes to $\boldsymbol X$ are in ``many different directions'', spanning $\RR^T$. We also
need to check a similar non-degeneracy condition for the effects of adding or removing
vertices from the various candidate neighbourhoods. Unfortunately, these non-degeneracy conditions do not hold for an arbitrary connected graph $H$ (they do hold, however, if $H$ has a vertex with edges to all other vertices). Therefore, we actually need to further modify our approach.

Instead of considering a single vertex $v_j$ in each step of the induction, we will consider $a_j$ different vertices, having a diverse range of adjacencies to the vertices previously chosen. Then, our decomposition is that we split the copies of $H$ into the copies that contain none of our $a_j$ identified vertices, and the copies that contain at least one of them. With this modification, there is a much richer range of possibilities for the effect of changing the status of an edge, and this allows us to prove the desired non-degeneracy condition. Unfortunately, while this modification is conceptually
rather simple, it complicates notation enormously. We now need to maintain a collection of sets of vertices,
and a collection of possibilities for the neighbourhoods of these
vertices. To encode all of these data we introduce the notion of
a \emph{colour system}: each step of the induction is associated with
a different colour, and at each step there are multiple ``shades''
of each colour indicating the different possibilities for the neighbourhoods
of the various vertices introduced at that step. We can then state
our induction hypothesis for a class of random variables defined in
terms of colour systems, and prove it using the ideas we discussed in this outline.

\section{Anticoncentration for \texorpdfstring{``almost-linear''}{"almost-linear"} random vectors}
\label{sec:halasz}

The Erd\H os--Littlewood--Offord theorem states that if $\xi_1,\dots,\xi_n$ are independent Bernoulli random variables satisfying $\Pr(\xi_i=0)=\Pr(\xi_i=1)=1/2$, and $X=a_1\xi_1+\dots+a_n\xi_n$ is a linear combination of these random variables (where each coefficient $a_i$ has absolute value at least one) then for any $x\in \RR$ we have $\Pr(|X-x|\le 1)=O(1/\sqrt n)$. As outlined in \cref{sec:outline}, in~\cite[Theorem~1.2]{FKS1} we adapted Erd\H os' proof of this theorem to handle the case of ``almost-linear'' functions of $\xi_1,\dots,\xi_n$.

In \cite{Hal77}, among other results, Hal\'asz gave a multivariate generalisation of the Erd\H os--Littlewood--Offord theorem, for sums of random vectors satisfying a certain non-degeneracy condition. Specifically, suppose that $\boldsymbol{a}_{1},\dots,\boldsymbol{a}_{n}\in\RR^{d}$ are $d$-dimensional vectors with the property that for every unit vector $\boldsymbol{e}\in\RR^{d}$, there are $\Omega\left(n\right)$ indices $i$ with $\left|\left\langle \boldsymbol{a}_{i},\boldsymbol{e}\right\rangle \right|\ge1$. Hal\'asz proved that, with $\boldsymbol{X}=\xi_{1}\boldsymbol{a}_{1}+\dots+\xi_{n}\boldsymbol{a}_{n}$, we have $\max_{\boldsymbol{x}\in \RR^d}\Pr\left(\left\Vert \boldsymbol{X}-\boldsymbol{x}\right\Vert _{2}\le1\right)=O\left(n^{-d/2}\right)$. As outlined in \cref{sec:outline}, for the proof of \cref{thm:subgraph-counts} we will need a similar bound for almost-linear $\boldsymbol{X}$. Our non-degeneracy condition will be somewhat cruder than Hal\'asz'; we assume that there are vectors $\boldsymbol{v}_{1},\dots,\boldsymbol{v}_{m}\in\RR^{d}$, spanning $\RR^{d}$, such that each of these vectors is represented $\Omega\left(n\right)$ times as the direction of the ``typical effect'' of changing the status of some $\xi_{i}$.

\begin{thm}
\label{thm:rough-halasz}Fix real numbers $0<p<1$ and $0<\varepsilon<1$, integers $m\geq d\geq 1$, and vectors $\boldsymbol{v}_{1},\dots,\boldsymbol{v}_{m}\in\RR^{d}$ spanning $\RR^d$. Then there is a constant $c>0$, depending on $p$, $\varepsilon$, $d$ and $\boldsymbol{v}_{1},\dots,\boldsymbol{v}_{m}$, such that the following holds. For any positive integer $n$ and any function $\boldsymbol{f}:\left\{ 0,1\right\} ^{n}\to\RR^{d}$,
let $\x\sim\Ber\left(p\right)^{n}$ and for $i=1,\dots,n$ define the random variables
\[\Delta_{i}\boldsymbol{f}\left(\x\right)  =\boldsymbol{f}\left(\xi_{1},\dots,\xi_{i-1},1,\xi_{i+1}.\dots,\xi_{n}\right)-\boldsymbol{f}\left(\xi_{1},\dots,\xi_{i-1},0,\xi_{i+1},\dots,\xi_{n}\right).\]
Suppose that for some positive real numbers $r$ and $s$ with $r\sqrt{n\log n}\ge s$ there are disjoint subsets
$I_{1},\dots,I_{m}\subseteq\left\{ 1,\dots,n\right\} $ of size at least $\varepsilon n$ such that for each $i\in I_{j}$ we have $\Pr\left(\left\Vert \Delta_{i}\boldsymbol{f}\left(\x\right)-s\boldsymbol{v}_{j}\right\Vert _{\infty}\ge r\right)\le n^{-6d^2}$. Then for any $\boldsymbol{x}\in\RR^{d}$ we have
\[
\Pr\left(\left\Vert \boldsymbol{f}\left(\x\right)-\boldsymbol{x}\right\Vert _{\infty}<r\sqrt{n\log n}\right)\leq c\cdot \left(\frac{r\sqrt{\log n}}{s}\right)^{d}.
\]
\end{thm}

Roughly speaking, the assumption on the function $\boldsymbol{f}$ in \cref{thm:rough-halasz} means the following: For each of the vectors $\boldsymbol{v}_j$ (with $1\leq j\leq m$) there is a reasonably large subset $I_j\su \{1,\dots,n\}$, such that for every $i\in I_j$ the following holds. When changing the $i$-th coordinate of $\x$ from $0$ to $1$ the corresponding vectors $\boldsymbol{f}(\x)$ typically differ by roughly $s\boldsymbol{v}_j$ (more precisely, with high probability the difference of the corresponding vectors $\boldsymbol{f}(\x)$ is close to the vector $s\boldsymbol{v}_j$). This condition can be seen as some sort of ``almost-linearity'' (at least with respect to certain coordinates of $\x$). Intuitively, this condition suggests that for a random vector $\x\sim\Ber\left(p\right)^{n}$, the vector $\boldsymbol{f}(\x)$ must be reasonably spread out and not too concentrated close to any given $\boldsymbol{x}\in\RR^{d}$. \cref{thm:rough-halasz} makes this intuition precise.
The precise bound of $n^{-6d^2}$ for the probability $\Pr\left(\left\Vert \Delta_{i}\boldsymbol{f}\left(\x\right)-s\boldsymbol{v}_{j}\right\Vert _{\infty}\ge r\right)$ in the assumptions of \cref{thm:rough-halasz} was chosen for convenience in the proof of the theorem. The exponent $6d^2$ is certainly not sharp, but the precise value of this exponent is not relevant for the remainder of this paper, so we made no effort to optimise the exponent at the cost of complicating the proof of \cref{thm:rough-halasz}.

As mentioned above, \cref{thm:rough-halasz} can be considered to be an analogue of Hal\'asz' classical anticoncentration inequality for linear vector-valued functions \cite{Hal77} (satisfying a certain non-degeneracy condition), in the weaker setting of ``almost-linear'' functions. However, our non-degeneracy condition is somewhat more restrictive than the one in Hal\'asz' original result. We also remark that an inequality very similar to Halasz' was also proved by Tao and Vu~\cite[Theorem~1.4]{TV12}, and that there is a large body of work proving similar results without a non-degeneracy condition (in which case the bounds are much weaker; see for example the survey in \cite[Section~2]{NV13}).

Before starting the proof of \cref{thm:rough-halasz}, we record the following basic fact about lattices.

\begin{lem}\label{claim:lattice}
Fix a basis $\boldsymbol{v}_{1},\dots,\boldsymbol{v}_{d}\in\RR^{d}$ of $\RR^d$. There exists a constant $c'$, only depending on $\boldsymbol{v}_{1},\dots,\boldsymbol{v}_{d}$, such that for any $\boldsymbol{x'}\in \RR^{d}$ and any real number $z\geq 1$ there are at most $c'\cdot z^d$ different $d$-tuples of integers $(t_1,\dots,t_d)\in\ZZ^d$ with $\left\Vert (t_1\boldsymbol{v}_{1}+\dots+t_d\boldsymbol{v}_{d})-\boldsymbol{x'}\right\Vert _{\infty}<z$.
\end{lem}

We next prove \cref{thm:rough-halasz}, further developing the ideas in the proof of \cite[Theorem~1.2]{FKS1}. As mentioned above, the assumption on $\boldsymbol{f}$ in \cref{thm:rough-halasz} means that when changing the $i$-th coordinate of $\x$ from zero to one for some $i\in I_j$, the vector $\boldsymbol{f}(\x)$ typically changes by roughly $s\boldsymbol{v}_j$. This means that, if we successively change appropriately chosen zeros to ones in $\x$, we can (with sufficiently high probability) control the changes of the vector $\boldsymbol{f}(\x)$, and show that $\boldsymbol{f}(\x)$ cannot be too often close to any given vector $\boldsymbol{x}\in \RR^d$. Indeed, if we change the coordinates of $\x$ with indices in $I_1\cup \dots\cup I_m$, then the vector $\boldsymbol{f}(\x)$ will typically move roughly along a lattice spanned by the vectors $\boldsymbol{v}_{1},\dots,\boldsymbol{v}_{m}$, and we can use \cref{claim:lattice} to show that there are not too many choices for $\xi$ where $\left\Vert\boldsymbol f(\boldsymbol{\chi})-\boldsymbol x\right\Vert_{\infty}<r\sqrt{n\log n}$.

\begin{proof}[Proof of \cref{thm:rough-halasz}]
First, by relabelling the given vectors $\boldsymbol{v}_{1},\dots,\boldsymbol{v}_{m}\in\RR^{d}$, we may assume that $\boldsymbol{v}_{1},\dots,\boldsymbol{v}_{d}$ form a basis of $\RR^{d}$. We ignore the other vectors $\boldsymbol{v}_{d+1},\dots,\boldsymbol{v}_{m}$ (in other words, we may assume that $m=d$).

Now, for $j=1,\dots, d$, let $n_j=|I_{j}|$, so $\varepsilon n\leq n_j\leq n$. Let $\x^{j}=\left(\xi_{i}\right)_{i\in I_{j}}$ be the restriction of the random vector $\x\sim \Ber\left(p\right)^{n}$ to the coordinates in $I_j$. Observe that the number $\vert \x^{j}\vert$ of ones in $\x^{j}$ is binomially distributed with parameters $n_j$ and $p$. Therefore, for any $0\leq t\leq n_j$ we have that
\begin{equation}\label{eq:halasz1}
\Pr\left(\left|\x^{j}\right|=t\right)\leq c_p\cdot n_j^{-1/2}
\le (c_p/\sqrt{\varepsilon})\cdot n^{-1/2}
\end{equation}
for a constant $c_p>0$ only depending on $p$.

Now, observe that $\Pr\left(\left|\x^{j}\right|=t\right)$ is an increasing function in $t$ for $t\leq pn_j$ and a decreasing function for $t\geq pn_j$. Thus, for each $j=1,\dots, d$, there are integers $0\leq a_j\leq b_j\leq n_j$ such that for any integer $t$ we have
\begin{equation}\label{eq:halasz0}
\Pr\left(\left|\x^{j}\right|=t\right)> n^{-2d}\text{ if and only if }a_j\le t\leq b_j.
\end{equation}
That is to say, $a_j$ and $b_j$ are defined as the boundaries of the range of values  that have probability at least $n^{-2d}$ of occurring as $\left|\x^{j}\right|$. We next bound the difference $b_j-a_j$.

The Chernoff bound (see for example \cite[Theorem A.1.4]{alon-spencer}) yields $\Pr\left(\left|\x^{j}\right|< pn_j-d\sqrt{n\log n}\right)\le n^{-2d^2}\le n^{-2d}$ and $\Pr\left(\left|\x^{j}\right|> pn_j+d\sqrt{n\log n}\right)\le n^{-2d^2}\le n^{-2d}$. Thus, we must have
\[ pn_j-d\sqrt{n\log n}\le a_j\le b_j\le pn_j+d\sqrt{n\log n}\]
and in particular $b_j-a_j\leq 2d\sqrt{n\log n}$ for each $j=1,\dots, d$. By the choice of $a_j$ and $b_j$ we have
\begin{equation}\label{eq:halasz2}
\Pr\left(\left|\x^{j}\right|< a_j\text{ or }\left|\x^{j}\right|>b_j\right)=\sum_{t=0}^{a_j-1} \Pr\left(\left|\x^{j}\right|=t\right)+\sum_{t=b_j+1}^{n^j} \Pr\left(\left|\x^{j}\right|=t\right)\le n\cdot n^{-2d}= n^{-2d+1}
\end{equation}
for every $j=1,\dots, d$.

Now, for each $j=1,\dots, d$, let $\sigma_{j}:[n_{j}] \to I_{j}$ be a uniformly random bijection (independently chosen for each $j$). Also, independently sample $\chi_i\sim\Ber\left(p\right)$ for each $i\in [n]\sm(I_1\cup\dots\cup I_d)$. For any integers $t_1,\dots,t_d\in [a_1,b_1]\times \dots \times [a_d,b_d]$, let the vector $\boldsymbol{\chi}(t_1,\dots,t_d)\in \lbrace 0,1\rbrace^n$ be defined as follows. For $i\in [n]\sm(I_1\cup\dots\cup I_d)$, we already chose $\chi_i$, the $i$-th entry of $\boldsymbol{\chi}(t_1,\dots,t_d)$. If $i\in I_j$, then set $\chi_i=1$ if and only if $i\in \sigma_{j}([t_i])$. In other words, among the entries $\chi_i$ for $i\in I_j$ there are precisely $t_j$ ones and those are in positions $\sigma_{j}\left(1\right),\dots,\sigma_{j}\left(t_{j}\right)$.

For any given $(t_1,\dots,t_d)\in [a_1,b_1]\times \dots \times [a_d,b_d]$, the random vector $\boldsymbol{\chi}(t_1,\dots,t_d)$ depends on the choices of the bijections $\sigma_j$ for $j=1,\dots, d$ and the random entries $\chi_i\sim\Ber\left(p\right)$ for $i\in [n]\sm(I_1\cup\dots\cup I_d)$. Very importantly, the distribution of $\boldsymbol{\chi}(t_1,\dots,t_d)$ is the same as the distribution of the random vector $\x\sim\Ber\left(p\right)^{n}$ in the theorem statement conditioned on having $\left|\x^{j}\right|=t_j$ for $j=1,\dots,d$. In particular, fixing any $\boldsymbol x\in \RR^d$, we have
\begin{multline*}\Pr\left(\left\Vert \boldsymbol{f}(\x)-\boldsymbol{x}\right\Vert _{\infty}<r\sqrt{n\log n}\,\,\bigg\vert\,\, \left|\x^{j}\right|=t_j\text{ for }j=1,\dots,d\right)\\
=\Pr\left(\left\Vert \boldsymbol{f}\left(\boldsymbol{\chi}(t_1,\dots,t_d)\right)-\boldsymbol{x}\right\Vert _{\infty}<r\sqrt{n\log n}\right).
\end{multline*}
Hence, using the independence of the random variables $|\x^{1}|, \dots, |\x^{d}|$ as well as \cref{eq:halasz1}, we obtain
\begin{multline*}\Pr\left(\left\Vert \boldsymbol{f}(\x)-\boldsymbol{x}\right\Vert _{\infty}<r\sqrt{n\log n}\text{ and }\left|\x^{j}\right|=t_j\text{ for }j=1,\dots,d\right)\\
=\Pr\left(\left\Vert \boldsymbol{f}\left(\boldsymbol{\chi}(t_1,\dots,t_d)\right)-\boldsymbol{x}\right\Vert _{\infty}<r\sqrt{n\log n}\right)\cdot \prod_{j=1}^{d}\Pr\left(\left|\x^{j}\right|=t_j\right)\\
\leq (c_p/\sqrt{\varepsilon})^d\cdot n^{-d/2}\cdot \Pr\left(\left\Vert \boldsymbol{f}\left(\boldsymbol{\chi}(t_1,\dots,t_d)\right)-\boldsymbol{x}\right\Vert _{\infty}<r\sqrt{n\log n}\right)
\end{multline*}
for each $(t_1,\dots,t_d)\in [a_1,b_1]\times \dots \times [a_d,b_d]$.
On the other hand, \cref{eq:halasz2} implies
\[\Pr\left(\left|\x^{j}\right|< a_j\text{ or }\left|\x^{j}\right|>b_j\text { for some }1\leq j\leq d\right)\leq d\cdot n^{-2d+1}\leq n^{-d}.\]
(Note that to have $\vert I_j\vert\geq \varepsilon n>0$ for $j=1,\dots,d$, we must have $d\le n$). Thus, we obtain
\begin{multline*}
\Pr\left(\left\Vert \boldsymbol{f}(\x)-\boldsymbol{x}\right\Vert _{\infty}<r\sqrt{n\log n}\right)\\
\leq n^{-d}+\sum_{(t_1,\dots,t_d)}\Pr\left(\left\Vert \boldsymbol{f}(\x)-\boldsymbol{x}\right\Vert _{\infty}<r\sqrt{n\log n}\text{ and }\left|\x^{j}\right|=t_j\text{ for }j=1,\dots,d\right)\\
\leq n^{-d}+ (c_p/\sqrt{\varepsilon})^d\cdot n^{-d/2}\cdot\sum_{(t_1,\dots,t_d)}\Pr\left(\left\Vert \boldsymbol{f}\left(\boldsymbol{\chi}(t_1,\dots,t_d)\right)-\boldsymbol{x}\right\Vert _{\infty}<r\sqrt{n\log n}\right),
\end{multline*}
where the sum is taken over all $(t_1,\dots,t_d)\in [a_1,b_1]\times \dots \times [a_d,b_d]$. In other words,
\begin{equation}\label{eq:halasz3}
\Pr\left(\left\Vert \boldsymbol{f}(\x)-\boldsymbol{x}\right\Vert _{\infty}<r\sqrt{n\log n}\right)\leq  n^{-d}+ (c_p/\sqrt{\varepsilon})^d\cdot n^{-d/2}\cdot\E Y,
\end{equation}
where $Y$ is the random variable counting the number of $d$-tuples $(t_1,\dots,t_d)\in [a_1,b_1]\times \dots \times [a_d,b_d]$ with $\left\Vert \boldsymbol{f}\left(\boldsymbol{\chi}(t_1,\dots,t_d)\right)-\boldsymbol{x}\right\Vert _{\infty}<r\sqrt{n\log n}$. This random variable $Y$ depends on the random choices of $\sigma_j$ for $j=1,\dots, d$ and on $\chi_i\sim\Ber\left(p\right)$ for $i\in [n]\sm(I_1\cup\dots\cup I_d)$.

Note that we always have $Y\leq (n_1+1)\dotsm (n_d+1)\leq (2n)^d$. Furthermore, recall that the distribution of $\boldsymbol{\chi}(t_1,\dots,t_d)$ is the same as the distribution of $\x\sim\Ber\left(p\right)^{n}$ conditioned on having $\left|\x^{j}\right|=t_j$ for $j=1,\dots,d$. This implies that for any $1\leq j^*\leq d$, any $i\in I_{j^*}$ and any $(t_1,\dots,t_d)\in [a_1,b_1]\times \dots \times [a_d,b_d]$, we have
\begin{multline*}
\Pr\left(\left\Vert \Delta_{i}\boldsymbol{f}\left(\boldsymbol{\chi}(t_1,\dots,t_d)\right)-s\boldsymbol{v}_{j^*}\right\Vert _{\infty}\ge r\right)=\Pr\left(\left\Vert \Delta_{i}\boldsymbol{f}\left(\boldsymbol{\x}\right)-s\boldsymbol{v}_{j^*}\right\Vert _{\infty}\ge r\,\,\bigg\vert\,\, \left|\x^{j}\right|=t_j\text{ for }j=1,\dots,d\right)\\
\leq \Pr\left(\left\Vert \Delta_{i}\boldsymbol{f}\left(\boldsymbol{\x}\right)-s\boldsymbol{v}_{j^*}\right\Vert _{\infty}\ge r\right)\cdot \prod_{j=1}^{d}\Pr\left(\left|\x^{j}\right|=t_j\right)^{-1}\leq n^{-6d^2}\cdot \prod_{j=1}^{d} n^{2d}\leq n^{-4d^2}\le n^{-2d-2},
\end{multline*}
where we used \cref{eq:halasz0} and the assumption in the theorem that $\Pr\left(\left\Vert \Delta_{i}\boldsymbol{f}\left(\boldsymbol{\x}\right)-s\boldsymbol{v}_{j^*}\right\Vert _{\infty}\ge r\right)\leq n^{-6d^2}$. Thus, by a union bound, the probability that there exist $1\leq j^*\leq d$, $i\in I_{j^*}$ and $(t_1,\dots,t_d)\in [a_1,b_1]\times \dots \times [a_d,b_d]$ with $\left\Vert \Delta_{i}\boldsymbol{f}\left(\boldsymbol{\chi}(t_1,\dots,t_d)\right)-s\boldsymbol{v}_{j^*}\right\Vert _{\infty}\ge r$, is at most $n^{-2d-2}\cdot d\cdot n\cdot n^{d}\le n^{-d}$.

Now, fix $c'>0$, only depending on $\boldsymbol{v}_{1},\dots,\boldsymbol{v}_{d}\in\RR^{d}$, as in \cref{claim:lattice}.

\begin{claim}\label{claim:Y-upper-bound}
If $\left\Vert \Delta_{i}\boldsymbol{f}\left(\boldsymbol{\chi}(t_1,\dots,t_d)\right)-s\boldsymbol{v}_{j^*}\right\Vert _{\infty}\le r$ for all $1\leq j^*\leq d$, all $i\in I_{j^*}$ and all $(t_1,\dots,t_d)\in [a_1,b_1]\times \dots \times [a_d,b_d]$, then we have $Y\leq c'\cdot (2d^2+1)^d\cdot (r\sqrt{n\log n}/s)^d$
\end{claim}
\begin{proof}
Note that for any $1\leq j^*\leq d$ and any $(t_1,\dots,t_d)\in [a_1,b_1]\times \dots \times [a_d,b_d]$ with $t_{j^*}<b_{j^*}$, the vectors $\boldsymbol{\chi}(t_1,\dots,t_{j^*-1}, t_{j^*}+1, t_{j^*+1},\dots,t_d)$ and $\boldsymbol{\chi}(t_1,\dots,t_d)$ only differ in that the first of these vectors has a one in position $\sigma_{j^*}(t_{j^*}+1)$, while the second has a zero in that position. Hence
\[\boldsymbol{f}(\boldsymbol{\chi}(t_1,\dots,t_{j^*-1}, t_{j^*}+1, t_{j^*+1},\dots,t_d))-\boldsymbol{f}(\boldsymbol{\chi}(t_1,\dots,t_d))=\Delta_{\sigma_{j^*}(t_{j^*}+1)}\boldsymbol{f}(\boldsymbol{\chi}(t_1,\dots,t_d)).\]
As $\sigma_{j^*}(t_{j^*}+1)\in I_{j^*}$, under the assumptions of the claim this implies
\[\left\Vert \boldsymbol{f}(\boldsymbol{\chi}(t_1,\dots,t_{j^*-1}, t_{j^*}+1, t_{j^*+1},\dots,t_d))-\boldsymbol{f}(\boldsymbol{\chi}(t_1,\dots,t_d))-s\boldsymbol{v}_{j^*}\right\Vert _{\infty}\le r\]
for all $(t_1,\dots,t_d)\in [a_1,b_1]\times \dots \times [a_d,b_d]$ and all $1\leq j^*\leq d$ with $t_{j^*}<b_{j^*}$. Adding this up for different values of $(t_1,\dots,t_d)\in [a_1,b_1]\times \dots \times [a_d,b_d]$ and using the triangle inequality, this implies that
\begin{multline*}
\left\Vert \boldsymbol{f}(\boldsymbol{\chi}(t_1,\dots,,t_d))-\boldsymbol{f}(\boldsymbol{\chi}(a_1,\dots,a_d))-(t_1-a_1)s\boldsymbol{v}_{1}-\dots-(t_d-a_d)s\boldsymbol{v}_{d}\right\Vert _{\infty}\\
\le ((t_1-a_1)+\dots+(t_d-a_d))\cdot r\le 2d^2\sqrt{n\log n}\cdot r,
\end{multline*}
for all $(t_1,\dots,t_d)\in [a_1,b_1]\times \dots \times [a_d,b_d]$, where for the second inequality we used that $t_j-a_j\leq b_j-a_j\leq 2d\sqrt{n\log n}$ for each $1\leq j\leq d$.

Recall that $Y$ is the number of (integer) $d$-tuples $(t_1,\dots,t_d)\in [a_1,b_1]\times \dots \times [a_d,b_d]$ which satisfy $\left\Vert \boldsymbol{f}\left(\boldsymbol{\chi}(t_1,\dots,t_d)\right)-\boldsymbol{x}\right\Vert _{\infty}<r\sqrt{n\log n}$. For each such $d$-tuple we then have (by the triangle inequality)
\[\left\Vert \boldsymbol{x}-\boldsymbol{f}(\boldsymbol{\chi}(a_1,\dots,a_d))-(t_1-a_1)s\boldsymbol{v}_{1}-\dots-(t_d-a_d)s\boldsymbol{v}_{d}\right\Vert _{\infty}\\
< (2d^2+1)\sqrt{n\log n}\cdot r,\]
and therefore
\[\left\Vert t_1\boldsymbol{v}_{1}+\dots+t_d\boldsymbol{v}_{d}-a_1\boldsymbol{v}_{1}-\dots-a_d\boldsymbol{v}_{d}+\frac{1}{s}\boldsymbol{f}(\boldsymbol{\chi}(a_1,\dots,a_d)))-
\frac{1}{s}\boldsymbol{x}\right\Vert _{\infty}\\
< (2d^2+1)\cdot \frac{\sqrt{n\log n}\cdot r}{s}.\]
Thus by \cref{claim:lattice} applied with $\boldsymbol{x'}=a_1\boldsymbol{v}_{1}+\dots+a_d\boldsymbol{v}_{d}-\frac{1}{s}\boldsymbol{f}(\boldsymbol{\chi}(a_1,\dots,a_d)))+\frac{1}{s}\boldsymbol{x}$ as well as $z=(2d^2+1)\sqrt{n\log n}\cdot r/s$ (note that $z\geq 1$ as $r\sqrt{n\log n}\ge s$ by the assumptions of the theorem), we obtain that
\[Y\leq c'\cdot (2d^2+1)^d\cdot \left(\frac{\sqrt{n\log n}\cdot r}{s}\right)^d,\]
as desired.
\end{proof}

Just before \cref{claim:Y-upper-bound}, we proved that its assumptions are satisfied with probability at least $1-n^{-d}$. Thus, we obtain
\[\E Y\leq n^{-d}\cdot (2n)^d+c'\cdot (2d^2+1)^d\cdot \left(\frac{\sqrt{n\log n}\cdot r}{s}\right)^d\leq (c'+1)\cdot (2d^2+1)^d\cdot \left(\frac{\sqrt{n\log n}\cdot r}{s}\right)^d,\]
using that $r\sqrt{n\log n}\ge s$. Plugging this into \cref{eq:halasz3}, we can conclude
\begin{multline*}
\Pr\left(\left\Vert \boldsymbol{f}(\x)-\boldsymbol{x}\right\Vert _{\infty}<r\sqrt{n\log n}\right)\leq  n^{-d}+ (c_p/\sqrt{\varepsilon})^d\cdot n^{-d/2}\cdot (c'+1)\cdot (2d^2+1)^d\cdot \left(\frac{\sqrt{n\log n}\cdot r}{s}\right)^d\\
=n^{-d}+ (c_p/\sqrt{\varepsilon})^d\cdot (c'+1)\cdot (2d^2+1)^d\cdot \left(\frac{r\sqrt{\log n}}{s}\right)^d\leq ((c_p/\sqrt{\varepsilon})^d\cdot (c'+1)\cdot (2d^2+1)^d+1)\cdot \left(\frac{r\sqrt{\log n}}{s}\right)^d,
\end{multline*}
where in the last inequality we used that $r\sqrt{\log n}/ s\geq n^{-1/2}\geq n^{-1}$. This implies the statement of \cref{thm:rough-halasz} with $c=(c_p/\sqrt{\varepsilon})^d\cdot (c'+1)\cdot (2d^2+1)^d+1$.
\end{proof}

\section{Colour systems and the induction hypothesis}
\label{sec:induction-hypothesis}

As outlined in \cref{sec:outline}, our proof of \cref{thm:subgraph-counts} proceeds via induction over a class of random variables generalising subgraph counts. These random variables are defined via \emph{colour systems}.

\begin{defn}\label{def:colour-system}
For integers $g\geq 0$ and $a_1,\dots,a_g, t_1,\dots,t_g\geq 1$, a \emph{colour system} $\mathcal{G}$ with parameters $(g, a_1,\dots,a_g, t_1,\dots,t_g)$ is a multigraph without loops which is coloured according to the following rules.
\begin{itemize}
\item Each vertex has at most one colour and for each $1\leq i\leq g$, there are exactly $a_i$ vertices of colour $i$.
\item Each edge is incident to at least one coloured vertex.
\item Each edge has exactly one colour. If an edge is incident to exactly one coloured vertex, it receives the colour of that vertex. If an edge is incident to two coloured vertices, and these vertices have colours $i_1$ and $i_2$, then the edge has colour $\min(i_1,i_2)$.
\item Each edge of colour $i$ is additionally labelled with an integer in $\{1,\dots,t_i\}$ (we say that there are $t_i$ different \emph{shades}) of colour $i$. We do not assign shades to vertices, only edges.
\item Between any two vertices, there is at most one edge of each shade of each colour (but there can be multiple edges of different shades of the same colour).
\end{itemize}
The \emph{order} of a colour system $\mathcal{G}$ is its total number of vertices (both coloured and uncoloured).
\end{defn}

For a colour system $\mathcal{G}$, let $\operatorname{U}(\mathcal{G})$ be the set of uncoloured vertices of $\mathcal{G}$.
In most of the statements throughout the paper, we will consider colour systems whose order $n$ is large with respect to the parameters $g,a_1,\dots,a_g, t_1,\dots,t_g$ (in which case almost all vertices of $\mathcal{G}$ are uncoloured). In all of our statements involving asymptotic notation, the colour system parameters $g,a_1,\dots,a_g, t_1,\dots,t_g$ will be treated as fixed constants for the asymptotic notation, while the order $n$ tends to infinity.

To make some sense of the definition of a colour system, recall from the outline in \cref{sec:outline} that our proof is inductive, and at each step of the induction we consider multiple possibilities for the neighbourhoods of certain vertices. The $g$ colours in a colour system correspond to the vertices chosen at the $g$ different steps of the induction, and the $t_i$ different shades of colour $i$ correspond to $t_i$ different choices of neighbourhoods for the vertices of colour $i$.

Now, we will mostly want to consider colour systems which have ``typical'' structure, meaning that the sizes of intersections between neighbourhoods of vertices are about what one should expect if the neighbourhoods were chosen randomly. For this, we introduce a notion of ``general position'' for families of sets.

\begin{defn}\label{def:general-position}
Consider subsets $S_{1},\dots,S_{m}$ of some ground set $R$. For any subset $I\subseteq\left\{ 1,\dots,m\right\}$, we write $S_{I}=R\cap \bigcap_{i\in I}S_{i}\cap \bigcap_{i\notin I}(R\sm S_{i})$. For an integer $K\geq 1$ and some $0<p<1$, we say that $S_{1},\dots,S_{m}\subseteq R $
are in \emph{$(p, K)$-general position} if for each
$I\subseteq\left\{ 1,\dots,m\right\} $, we have
\[ \left|\left|S_{I}\right|-p^{\left|I\right|}\left(1-p\right)^{m-\left|I\right|}|R|\right|\le  K\cdot |R|^{1/2}\log |R|.\]
\end{defn}

Note that if $m=0$, then $S_\emptyset=R$ and therefore in this case the empty collection of sets is in $(p, K)$-general position for every integer $K\geq 1$.

\begin{defn}
Say that a colour system $\mathcal{G}$ is \emph{$p$-general} if the following holds. If we define $S_1,\dots,S_m\su \operatorname{U}(\mathcal{G})$ to be the $\operatorname{U}(\mathcal{G})$-neighbourhoods of each of the coloured vertices of $\mathcal{G}$ in each of the shades of the respective colour (so we have $m=a_1\cdot t_1+\dots+a_g\cdot t_g$ if $\mathcal{G}$ has parameters $(g, a_1,\dots,a_g, t_1,\dots,t_g)$), then the sets $S_1,\dots,S_m\su \operatorname{U}(\mathcal{G})$ are in $(p, 3^g)$-general position.

Furthermore, say that the colour system $\mathcal{G}$ is \emph{weakly $p$-general} if the sets $S_1,\dots,S_m\su \operatorname{U}(\mathcal{G})$ are in $(p, 2\cdot 3^{g})$-general position.
\end{defn}

Note that every $p$-general colour system is in particular weakly $p$-general (the reason for having both these definitions is that when we make small changes to a collection of sets in $(p,K)$-general position, the parameter $K$ changes slightly, and it is convenient to not have to explicitly keep track of this change). Also, note that for $g=0$, every colour system with parameters $(g, a_1,\dots,a_g, t_1,\dots,t_g)$ is $p$-general (since $m=0$ and the empty collection of sets is always in in $(p, K)$-general position for all $K\geq 1$).

Recall from \cref{sec:outline} that the whole point of introducing multiple vertices at each step of the induction is to allow for a richer range of possibilities for the effect of changing the status of an edge. In order to ensure the richest possible range of possibilities, we consider colour systems which are \emph{complete} in the sense that we see essentially all the possible adjacencies between the coloured vertices, as follows.

\begin{defn}\label{defn:complete}
Call a colour system $\mathcal{G}$ with parameters $(g, a_1,\dots,a_g, t_1,\dots,t_g)$ complete, if for any $1\leq i\leq g$ the following holds. Suppose for each $1\leq j\leq i-1$ and each vertex $v$ in $\mathcal{G}$ of colour $j$ we are given a subset $I_v\su [t_j]$. Then there exists a vertex $w$ of colour $i$ such that for every $1\leq j\leq i-1$ and every vertex $v$ of colour $j$ the vertices $w$ and $v$ are connected by edges of exactly those shades of colour $j$ that belong to the set $I_{v}$.
\end{defn}

Informally speaking, \cref{defn:complete} demands that for every colour $i$ we can find a vertex with prescribed edges to all the vertices of the previous (smaller) colours. For $g=0$ every colour system with parameters $(g, a_1,\dots,a_g, t_1,\dots,t_g)$ is complete, as the condition in \cref{defn:complete} is vacuous.

Note that whether a colour system $\mathcal{G}$ with parameters $(g, a_1,\dots,a_g, t_1,\dots,t_g)$ is complete only depends on the edges between the coloured vertices, and it does not at all depend on the edges in colour $g$. In contrast, whether $\mathcal{G}$ is $p$-general for given $0<p<1$ only depends on the edges between the coloured and the uncoloured vertices.

Now, to obtain a graph from a colour system, we choose a shade of each colour, and we choose a graph on the uncoloured vertices, as follows.

\begin{defn}\label{defn:psi}
Let $\mathcal{G}$ be a colour system with parameters $(g, a_1,\dots,a_g, t_1,\dots,t_g)$. Then, given a $g$-tuple $(j_1,\dots,j_g)\in [t_1]\times \dots\times [t_g]$, and a graph $G_0$ on the vertex set $\operatorname{U}(\mathcal{G})$, define a graph $\mathcal{G}(G_0, j_1,\dots,j_g)$ by taking all vertices of $\mathcal{G}$ together with all edges of $G_0$ and all edges of shade $j_i$ of colour $i$ for all $1\leq i\leq g$. Furthermore, for a graph $H$ let $\psi_H(\mathcal{G}, G_0, j_1,\dots,j_g)$ be the number of labelled copies of $H$ in the graph $\mathcal{G}(G_0, j_1,\dots,j_g)$ which use at least one vertex of each of the $g$ colours.
\end{defn}

We will use notation such as $\psi_H(\mathcal{G}, G_0,\cdot)$ to denote the function $[t_1]\times \dots\times [t_g]\to \ZZ$ that maps $(j_1,\dots,j_g)$ to $\psi_H(\mathcal{G}, G_0, j_1,\dots,j_g)$.
Our goal for the rest of this paper will be to prove the following strengthening of \cref{thm:subgraph-counts}.

\begin{thm}
\label{thm:general-subgraph-counts}Fix some $0<p<1$, an $h$-vertex graph $H$, and integers $0\leq g\leq h-1$ and $a_1,\dots,a_g,\allowbreak t_1,\dots,t_g\geq 1$. Let $T=t_1\dotsm t_g$. Then for any $p$-general complete colour system $\mathcal{G}$ of order $n$ with parameters $(g, a_1,\dots,a_g, t_1,\dots,t_g)$ and for any function $\lambda: [t_1]\times \dots\times [t_g]\to \ZZ$ the following holds. If $G_0\sim\GG(\operatorname{U}(\mathcal{G}), p)$ is a random graph on the vertex set $\operatorname{U}(\mathcal{G})$, then 
\[
\Pr\left(\vphantom{\sum}\psi_H(\mathcal{G}, G_0, \cdot)=\lambda\right)\le n^{(g-h+1)\cdot T+o(1)}.
\]
\end{thm}

The $o(1)$-term in \cref{thm:general-subgraph-counts} goes to zero as $n$ tends to infinity, but it may depend on the choices for $p$, $h$, $H$, and $g, a_1,\dots,a_g,t_1,\dots,t_g$ fixed in the beginning of the theorem statement (in other words, these values are are treated as constants in the asymptotics).

Note that for $g=0$, in \cref{thm:general-subgraph-counts} we have $T=1$ (using the convention that the empty product is equal to 1), and the colour system $\mathcal{G}$ has no coloured vertices (so it consists of $n$ uncoloured vertices and no edges). We already observed that such a colour system is always $p$-general and complete. Furthermore, $\psi_H(\mathcal{G}, G_0)$ is simply the number of labelled copies of $H$ in $G_0\sim\GG(n, p)$. This quantity is precisely the random variable $X_H$ in \cref{thm:subgraph-counts}. Thus, \cref{thm:general-subgraph-counts} for $g=0$ states that for all $\ell\in \ZZ$ we have $X_H=\ell$ with probability at most $n^{(1-h)\cdot 1+o(1)}=n^{1-h+o(1)}$. This is precisely the statement of \cref{thm:subgraph-counts}.

So, \cref{thm:subgraph-counts} corresponds to the case $g=0$ in \cref{thm:general-subgraph-counts}, and it therefore suffices to prove \cref{thm:general-subgraph-counts}. We will use backwards induction starting from $g=h-1$ and going down to $g=0$. Note that the case $g=h-1$ is trival.

For the rest of the paper, we fix a particular graph $H$ with $h$ vertices and some $0<p<1$. Before concluding this section, we make a few more definitions and state an important intermediate result for the induction step (\cref{cor:dispersed}, to follow). Basically, at each step of the induction, we have a colour system $\mathcal G$ with $g-1$ colours, and we add vertices of a new colour with random neighbourhoods, obtaining a random colour system $\mathcal{G_S}$. We will use the ``$g$'' case of \cref{thm:general-subgraph-counts} and a moment argument to show that typically $G_0$ has the property that $\psi(\mathcal {G_S},G_0,\cdot)$ is anticoncentrated, subject only to the randomness in $\mathcal{G_S}$. This will be the content of \cref{cor:dispersed}.

\begin{defn}\label{defn:restricted-colour-system}
For integers $g\geq 1$ and $a_1,\dots,a_g, t_1,\dots,t_{g-1}\geq 1$, define a \emph{restricted colour system} with parameters $(g, a_1,\dots,a_g, t_1,\dots,t_{g-1})$ to be a colour system with parameters $(g, a_1,\dots,a_g, t_1,\dots,t_{g-1}, 1)$ in which there are no edges in colour $g$ between the coloured and the uncoloured vertices (so all edges of colour $g$ are between the vertices of colour $g$, recalling that the colour of an edge is the minimum of the colours of its endpoints).
\end{defn}

Call a restricted colour system $\mathcal{G}$ with parameters $(g, a_1,\dots,a_g, t_1,\dots,t_{g-1})$ complete, if it is complete when viewed as a colour system with parameters $(g, a_1,\dots,a_g, t_1,\dots,t_{g-1}, 1)$. Call a restricted colour system $\mathcal{G}$ with parameters $(g, a_1,\dots,a_g, t_1,\dots,t_{g-1})$ \emph{essentially $p$-general}, if the colour system with parameters $(g-1, a_1,\dots,a_{g-1}, t_1,\dots,t_{g-1})$ obtained by ignoring colour $g$ is $p$-general. Similarly, call $\mathcal{G}$ \emph{essentially weakly $p$-general}, if the colour system obtained by ignoring colour $g$ is weakly $p$-general.

\begin{defn}\label{defn:extension-restricted-colour-system}
Consider a restricted colour system $\mathcal{G}$ with parameters $(g, a_1,\dots,a_g, t_1,\dots,t_{g-1})$. For each vertex $v$ of colour $g$ in $\mathcal{G}$, choose a random subset $S_v \su \operatorname{U}(\mathcal{G})$ by taking each element of $\operatorname{U}(\mathcal{G})$ independently with probability $p$ (and choose the different sets $S_v$ all independent from each other). Let $\mathcal{S}=(S_v)_{v}$ be the collection of random sets chosen this way. Then we can obtain a random colour system $\mathcal{G_S}$ with parameters $(g, a_1,\dots,a_g, t_1,\dots,t_{g-1}, 1)$ by connecting each vertex $v$ of colour $g$ to all vertices in $S_v$ and colouring all these new edges in the unique shade of colour $g$.
\end{defn}

\begin{defn}\label{defn:dispersedness}
Consider $0<q<1$ and an essentially $p$-general complete restricted colour system $\mathcal{G}$ with parameters $(g, a_1,\dots,a_g, t_1,\dots,t_{g-1})$. We say that a graph $G_0$ on the vertex set $\operatorname{U}(\mathcal{G})$ is \emph{$(p,q,\mathcal{G})$-dispersed} if for all functions $\lambda: [t_1]\times \dots\times [t_{g-1}]\times [1]\to \ZZ$ the following holds. When choosing $\mathcal{S}$ randomly as in \cref{defn:extension-restricted-colour-system}, we have $\Pr\left(\psi_H(\mathcal{G_S}, G_0,\cdot)=\lambda\right)\le q$.
\end{defn}

Now we are ready to state \cref{cor:dispersed}, as announced above.

\begin{prop}
\label{cor:dispersed}Fix integers $1\leq g\leq h-1$ and $a_1,\dots,a_g, t_1,\dots,t_{g-1}\geq 1$ and let $T=t_1\dotsm t_{g-1}$. Then there exists a function $\delta:\NN\to \RR_{\geq 0}$ with $\lim_{n\to\infty}\delta(n)=0$, such that for any essentially $p$-general complete restricted colour system $\mathcal{G}$ of order $n$ with parameters $(g, a_1,\dots,a_g, t_1,\dots,t_{g-1})$ the following holds. If $G_0\sim\GG(\operatorname{U}(\mathcal{G}), p)$ is a random graph on the vertex set $\operatorname{U}(\mathcal{G})$, then with probability $1-n^{-\omega(1)}$ the graph $G_0$ is $(p,q,\mathcal{G})$-dispersed, where $q=n^{(g-h+(1/2))\cdot T+\delta(n)}$.
\end{prop}

We remark that in the case $g=1$, we have $T=1$ in \cref{cor:dispersed} (again by the convention that the empty product is equal to $1$).

The rest of the paper is organised as follows. In \cref{sec:general-position} we give some very straightforward lemmas about sets in general position (in particular, random collections of sets are very likely to be in general position). In \cref{sect:thm-implies-cor} we use a moment argument to prove that the ``$g$'' case of \cref{thm:general-subgraph-counts} implies the corresponding case of \cref{cor:dispersed}. In \cref{sect:cor-implies-thm} we show how this case of \cref{cor:dispersed} implies the ``$g-1$'' case of \cref{thm:general-subgraph-counts}, completing the induction step. The contents of both of these sections consist mostly of calculations and putting various pieces together. However, in both these sections we omit the proofs of important anticoncentration lemmas (\cref{lem:medium-scale,lem:rough-scale}). The rest of the paper is spent proving these lemmas via our new multivariate anticoncentration inequality in \cref{thm:rough-halasz}. In \cref{sec:cores} we introduce the formalism of a ``core'', which will be used for the proofs of \cref{lem:medium-scale,lem:rough-scale}. To be specific, we prove a linear independence lemma for certain vectors defined in terms of cores, which we will use to check the linear independence condition in \cref{thm:rough-halasz}. Finally, in \cref{sec:halasz-prep} we put everything together to prove \cref{lem:medium-scale,lem:rough-scale}.

\section{Sets in general position}
\label{sec:general-position}

In this section we record some simple lemmas regarding sets in general position. Recall that in the last section we fixed some $p\in (0,1)$.

\begin{lem}\label{lem:sets-general-position-intersection}
Suppose that $S_{1},\dots,S_{m}\su R$ are subsets of some ground set $R$ which are in $(p, K)$-general position, for some integer $K\geq 1$. Then for every subset $I\subseteq\left\{ 1,\dots,m\right\} $, we have
\[
\left|\left|\bigcap_{i\in I}S_{i}\right|-p^{\left|I\right|}|R|\right|\le 2^m\cdot  K\cdot |R|^{1/2}\log |R|.
\]
\end{lem}

\begin{proof}
Fix some subset $I\subseteq\left\{ 1,\dots,m\right\} $. For every subset $J\subseteq\left\{ 1,\dots,m\right\}$ with $I\su J$, let $S_{J}=R\cap \bigcap_{i\in J}S_{i}\cap \bigcap_{i\notin J}(R\sm S_{i})$ be as in \cref{def:general-position}. Note that there are at most $2^m$ such subsets $J$. Also, note that the set $\bigcap_{i\in I}S_{i}$ is the union of all the set $S_J$ for all $J\subseteq\left\{ 1,\dots,m\right\}$ with $I\su J$, and all these sets $S_J$ are disjoint. Hence
\[\left|\bigcap_{i\in I}S_{i}\right|=\sum_{J}\left|S_J\right|.\]
Here, and in the rest of this proof, the sum is over all subsets $J\subseteq\left\{ 1,\dots,m\right\}$ with $I\su J$. Noting that
\[p^{\left|I\right|}|R|=p^{\left|I\right|}\left(p+(1-p)\right)^{m-\left|I\right|}|R|=\sum_{J}p^{\left|J\right|}\left(1-p\right)^{m-\left|J\right|}|R|,\]
we obtain that
\begin{multline*}
\left|\left|\bigcap_{i\in I}S_{i}\right|-p^{\left|I\right|}|R|\right|=\left|\sum_{J}\left|S_J\right|-\sum_{J}p^{\left|J\right|}\left(1-p\right)^{m-\left|J\right|}|R|\right|\leq \sum_{J}\left|\left|S_J\right|-p^{\left|J\right|}\left(1-p\right)^{m-\left|J\right|}|R|\right|\\
\leq \sum_{J} K\cdot |R|^{1/2}\log |R|\leq 2^m\cdot  K\cdot |R|^{1/2}\log |R|,
\end{multline*}
as desired.
\end{proof}

\begin{lem}\label{lem:sets-general-position-addition}
Fix some $m\in \NN$. Let $R$ be an $n$-element set and let $S_{1},\dots,S_{m}\su R$ be in $(p,K)$-general position for some integer $K\geq 1$. Let $S_{m+1}$ be a random set chosen by taking each element of $R$ independently with probability $p$. Then with probability $1-n^{-\omega(1)}$ the sets $S_{1},\dots,S_m,S_{m+1}\su R$ are in $(p,K)$-general position.
\end{lem}

\begin{proof}
For any subset $I\subseteq\left\{ 1,\dots,m\right\}$, let $S_{I}=R\cap \bigcap_{i\in I}S_{i}\cap \bigcap_{i\notin I}(R\sm S_{i})$ be as in \cref{def:general-position}. We need to show that with probability $1-n^{-\omega(1)}$ we have
\begin{equation}\label{eq-S-I-addition-to-show-1}
 \left|\left|S_{I}\cap S_{m+1}\right|-p^{\left|I\right|+1}\left(1-p\right)^{m-\left|I\right|}n\right|\le  K\cdot n^{1/2}\log n.
\end{equation}
and
\begin{equation}\label{eq-S-I-addition-to-show-2}
\left|\left|S_{I}\cap (R\sm S_{m+1})\right|-p^{\left|I\right|}\left(1-p\right)^{m-\left|I\right|+1}n\right|\le  K\cdot n^{1/2}\log n.
\end{equation}
for each $I\subseteq\left\{ 1,\dots,m\right\}$. By a union bound over all $2^m$ choices for $I$, it suffices that for each individual set $I\subseteq\left\{ 1,\dots,m\right\}$ each of these properties holds with probability $1-n^{-\omega(1)}$.

Fix some $I\subseteq\left\{ 1,\dots,m\right\}$. By assumption, the set $S_I\su R$ satisfies 
\begin{equation}\label{eq-S-I-assumption-addition}
 \left|\left|S_{I}\right|-p^{\left|I\right|}\left(1-p\right)^{m-\left|I\right|}n\right|\le  K\cdot n^{1/2}\log n.
\end{equation}
Note that $\left|S_{I}\cap S_{m+1}\right|$ is binomially distributed with parameters $(\left|S_{I}\right|, p)$. Thus, the probability to have 
\begin{equation}\label{eq-S-I-close-addition}
\left|\left|S_{I}\cap S_{m+1}\right|-p\left|S_{I}\right|\right|\le  (1-p)\cdot n^{1/2}\log n
\end{equation}
is by the Chernoff bound at least
\[1-2\exp\left(-\frac{\left((1-p)\cdot n^{1/2}\log n\right)^2}{\left|S_{I}\right|}\right)\geq 1-\exp\left(-\Omega\left((\log n)^2\right)\right)=1-n^{-\omega(1)}, \]
where we used that $\vert S_I\vert\leq n$ as $S_I\su R$. Whenever \cref{eq-S-I-close-addition} is satisfied, then together with \cref{eq-S-I-assumption-addition} we obtain by the triangle inequality
\[\left|\left|S_{I}\cap S_{m+1}\right|-p^{\left|I\right|+1}\left(1-p\right)^{m-\left|I\right|}n\right|\leq  (1-p)\cdot n^{1/2}\log n+p\cdot K\cdot n^{1/2}\log n\le  K\cdot n^{1/2}\log n,\]
as desired. Thus \cref{eq-S-I-addition-to-show-1} indeed holds with probability $1-n^{-\omega(1)}$. Similarly, we can show that \cref{eq-S-I-addition-to-show-2} holds with probability $1-n^{-\omega(1)}$ by considering the probability that
\[\left|\left|S_{I}\cap (R\sm S_{m+1})\right|-(1-p)\left|S_{I}\right|\right|\le  p\cdot n^{1/2}\log n\]
holds, using that $\left|S_{I}\cap (R\sm S_{m+1})\right|$ is binomially distributed with parameters $(\left|S_{I}\right|, 1-p)$.
\end{proof}

\begin{lem}\label{lem:sets-general-position-deletion}
Fix some positive integer $\ell$. Then the following holds for sufficiently large $n$. Let $R$ be an $n$-element set and let $S_{1},\dots,S_{m}\su R$ be in $(p,K)$-general position for some integer $K\geq 1$. Let $K'$ be an integer with $K'>K$. Then for any subset $R'\su R$ obtained by deleting $\ell$ elements from $R$, the sets $S_1\cap R',\dots,S_m\cap R'\su R'$ are in $(p,K')$-general position.
\end{lem}

\begin{proof}
We need to show that for every subset $I\subseteq\left\{ 1,\dots,m\right\}$ we have
\[ \left|\left|S_{I}\cap R'\right|-p^{\left|I\right|}\left(1-p\right)^{m-\left|I\right|}|R'|\right|\le  K'\cdot |R'|^{1/2}\log |R'|.\]
So fix some $I\subseteq\left\{ 1,\dots,m\right\}$. By assumption, the set $S_I\su R$ satisfies 
\[ \left|\left|S_{I}\right|-p^{\left|I\right|}\left(1-p\right)^{m-\left|I\right|}|R|\right|\le  K\cdot |R|^{1/2}\log |R|.\]
Now, $S_{I}\cap R'$ is obtained from $S_{I}$ by deleting at most $\ell$ elements, and therefore
\[\left|\left|S_{I}\cap R'\right|-\left|S_{I}\right|\right|\le \ell.\]
Furthermore, we have
\[ \left|p^{\left|I\right|}\left(1-p\right)^{m-\left|I\right|}|R|-p^{\left|I\right|}\left(1-p\right)^{m-\left|I\right|}|R'|\right|=p^{\left|I\right|}\left(1-p\right)^{m-\left|I\right|}\left|\left|R\right|-|R'|\right|= p^{\left|I\right|}\left(1-p\right)^{m-\left|I\right|}\ell\le \ell.\]
Thus, the triangle inequality yields (using $\vert R'\vert=\vert R\vert-\ell$)
\[ \left|\left|S_{I}\cap R'\right|-p^{\left|I\right|}\left(1-p\right)^{m-\left|I\right|}|R'|\right|\le  K\cdot |R|^{1/2}\log |R|+2\ell\leq K'\cdot |R'|^{1/2}\log |R'|.\]
as long as $n=\vert R\vert$ is sufficiently large with respect to $\ell$.
\end{proof}

\begin{cor}
\label{lemma-general-deletion-weakly}
Fix $g,a_1,\dots,a_{g-1}, t_1,\dots,t_{g-1}$, and suppose $n$ is sufficiently large with respect to these values. Let $\mathcal{G}$ be a $p$-general colour system of order $n$ with parameters $(g-1, a_1,\dots,a_{g-1}, t_1,\dots,t_{g-1})$. Then, whenever we delete $2^{a_1t_1+\dots+a_{g-1}t_{g-1}}$ vertices in $\operatorname{U}(\mathcal{G})$, the resulting colour system is weakly $p$-general.
\end{cor}

\begin{proof}
Let $S_1,\dots,S_m\su \operatorname{U}(\mathcal{G})$ be the $\operatorname{U}(\mathcal{G})$-neighbourhoods of each of the coloured vertices of $\mathcal{G}$ in each shade of their respective colours. Then $S_1,\dots,S_m\su \operatorname{U}(\mathcal{G})$ are in $(p, 3^{g-1})$-general position. Thus, by \cref{lem:sets-general-position-deletion}, whenever we delete $2^{a_1t_1+\dots+a_{g-1}t_{g-1}}$ vertices in $\operatorname{U}(\mathcal{G})$, the resulting configuration of sets is in $(p, 2\cdot 3^{g-1})$-general position.
\end{proof}

\begin{lem}\label{lem:extension-p-general}
Fix $g,a_1,\dots,a_{g}, t_1,\dots,t_{g-1}$. If $\mathcal{G}$ is a an essentially weakly $p$-general restricted colour system of order $n$ with parameters $(g, a_1,\dots,a_g, t_1,\dots,t_{g-1})$, then the colour system $\mathcal{G_S}$ in \cref{defn:extension-restricted-colour-system} is $p$-general with probability  $1-n^{-\omega(1)}$.
\end{lem}
\begin{proof}
Let $\mathcal{G}'$ be the colour system of order $n$ with parameters $(g-1, a_1,\dots,a_{g-1}, t_1,\dots,t_{g-1})$ obtained from the restricted colour system $\mathcal{G}$ by ignoring colour $g$ (so $\operatorname{U}(\mathcal{G}')$ consists of the set $\operatorname{U}(\mathcal{G})$, together with the $a_g$ vertices of colour $g$ in $\mathcal{G}$). By the assumption on $\mathcal G$, the colour system $\mathcal{G}'$ is weakly $p$-general. Let  $S_1,\dots,S_m\su \operatorname{U}(\mathcal{G}')$ be the $\operatorname{U}(\mathcal{G}')$-neighbourhoods of each of the coloured vertices of $\mathcal{G}'$ in each shade of their respective colours (so we have $m=a_1\cdot t_1+\dots+a_{g-1}\cdot t_{g-1}$). Then the sets $S_1,\dots,S_m\su \operatorname{U}(\mathcal{G}')$ are in $(p, 2\cdot 3^{g-1})$-general position. Note that $\vert \operatorname{U}(\mathcal{G})\vert=n-a_1+\dots+a_g=n-O(1)$ and $\vert \operatorname{U}(\mathcal{G}')\vert- |\operatorname{U}(\mathcal{G})|=a_g=O(1)$, and that $3^g>2\cdot 3^{g-1}$, so by \cref{lem:sets-general-position-deletion} (assuming $n$ is large), the sets $S_1\cap \operatorname{U}(\mathcal{G}),\dots,S_m\cap\operatorname{U}(\mathcal{G})\subseteq \operatorname{U}(\mathcal{G})$ are in $(p,3^g)$-general position. Applying \cref{lem:sets-general-position-addition} $a_g$ times, we see that with probability at least $1-a_g\cdot \vert \operatorname{U}(\mathcal{G})\vert^{-\omega(1)}=1-n^{-\omega(1)}$ these sets together with the random sets in $\mathcal{S}$ are in $(p,3^g)$-general position.
This proves \cref{lem:extension-p-general}.
\end{proof}

\section{Random neighbourhoods: \texorpdfstring{\cref{thm:general-subgraph-counts} implies \cref{cor:dispersed}}{Theorem~\ref{thm:general-subgraph-counts} implies Proposition~\ref{cor:dispersed}}}
\label{sect:thm-implies-cor}

In this section we will prove that if \cref{thm:general-subgraph-counts}  holds for some $1\leq g\leq h-1$, then \cref{cor:dispersed} also holds for this value of $g$.

Fix some $1\leq g\leq h-1$ and assume that \cref{thm:general-subgraph-counts} holds for this value of $g$. Let $a_1,\dots,a_g, t_1,\dots,t_{g-1}$ be arbitrary positive integers. Our goal in this section is to prove \cref{cor:dispersed} for these values of $g$, $a_1,\dots,a_g, t_1,\dots,t_{g-1}$.

For the entirety of this section, we fix these values of $g,a_1,\dots,a_g, t_1,\dots,t_{g-1}$, and define $T=t_1\dotsm t_{g-1}$. In all our asymptotics, these values will be treated as constants, while $n\to \infty$.

In \cref{subsec:dispersedness-prep} we will start with some preparations. First, we use a martingale concentration inequality to prove that $\psi_H(\mathcal{G_S}, G_0,\cdot)$ is very likely to be close to its conditional expectation $\mu_{\mathcal G, \mathcal S}$ given $\mathcal S$ (by symmetry, this conditional expectation actually only depends on the sizes of the intersections between the sets in $\mathcal S$ and the neighbourhoods of the various coloured vertices). Second, we state (but do not yet prove) an anticoncentration lemma for $\mu_{\mathcal G, \mathcal S}$, subject to the randomness in $\mathcal S$. 

In \cref{subsec:thm-implies-cor-key-lemma} we will use our preparatory lemmas and \cref{thm:general-subgraph-counts} to prove an anticoncentration bound for certain joint probabilities concerning the values of $\psi_H(\mathcal{G_S}, G_0, \cdot)$ for different choices of $\mathcal{S}$. This will be the key input for a moment argument (as outlined in \cref{sec:outline}) with which we will deduce that $G_0$ is very likely to be dispersed, proving the desired case of \cref{cor:dispersed}. The details of this deduction will be presented in \cref{subsec:cor-completion}.

\subsection{Preparations}
\label{subsec:dispersedness-prep}

\begin{defn}\label{defn:mu}
For a restricted colour system $\mathcal{G}$ with parameters $(g, a_1,\dots,a_g, t_1,\dots,t_{g-1})$ and an outcome of the random collection of sets $\mathcal{S}=(S_v)_{v}$  in \cref{defn:extension-restricted-colour-system}, let $\mu_{\mathcal{G}, \mathcal{S}}: [t_1]\times \dots\times [t_{g-1}]\times [1]\to \RR$ be the function given by
\[\mu_{\mathcal{G}, \mathcal{S}}(j_1,\dots,j_g)=\E_{G_0}[\psi_H(\mathcal{G_S}, G_0, j_1,\dots,j_g)]\]
for all $(j_1,\dots,j_g)\in [t_1]\times \dots\times [t_{g-1}]\times [1]$. Here, $G_0\sim\GG(\operatorname{U}(\mathcal{G}), p)$ is a random graph on the vertex set $\operatorname{U}(\mathcal{G})$.
\end{defn}

We remark that $\mu_{\mathcal{G}, \mathcal{S}}$ only depends on $\mathcal{G}$ and $\mathcal{S}$, and not $G_0$.

\begin{lem}
\label{lemma-graph-good} For any essentially $p$-general complete restricted colour system $\mathcal{G}$ which has parameters $(g, a_1,\dots,a_g, t_1,\dots,t_{g-1})$, the following holds. If we choose a random graph $G_0\sim\GG(\operatorname{U}(\mathcal{G}), p)$ and independently choose $\mathcal{S}$ randomly as in \cref{defn:extension-restricted-colour-system}, then
\[\Pr\left(\left\|\psi_H(\mathcal{G_S}, G_0,\cdot)-\mu_{\mathcal{G}, \mathcal{S}}\right\|_\infty\leq n^{h-g-1}\cdot \log n\right)=1-n^{-\omega(1)}.\]
\end{lem}
\begin{proof}
Condition on an arbitrary outcome of $\mathcal{S}$. By a union bound, it suffices to prove that for each $(j_1,\dots,j_g)\in [t_1]\times \dots\times [t_{g-1}]\times [1]$ the probability that
\begin{equation}\label{eqn:lemma-graph-good-to-show}
\left|\psi_H(\mathcal{G_S}, G_0, j_1,\dots,j_g)-\mu_{\mathcal{G}, \mathcal{S}}(j_1,\dots,j_g)\right|>n^{h-g-1}\cdot \log n
\end{equation}
is $n^{-\omega(1)}$. So fix some $(j_1,\dots,j_g)\in [t_1]\times \dots\times [t_{g-1}]\times [1]$. The expectation of $\psi_H(\mathcal{G_S}, G_0, j_1,\dots,j_g)$ is precisely $\mu_{\mathcal{G}, \mathcal{S}}(j_1,\dots,j_g)$ (recall that we are conditioning on an outcome of $\mathcal S$).

Note that changing the status of an edge of $G_0$ changes $\psi_H(\mathcal{G_S}, G_0, j_1,\dots,j_g)$ by at most $O(n^{h-g-2})$. This is because are at most $h^{g+2} \cdot a_1\dotsm a_{g}\cdot n^{h-g-2}=O(n^{h-g-2})$ different labelled copies of $H$ in the $n$-vertex graph $\mathcal{G_S}(G_0, j_1,\dots,j_g)$ which contain any particular pair of vertices of $\operatorname{U}(\mathcal G)$, and contain at least one vertex of each of the $g$ colours. Thus, by the Azuma--Hoeffding inequality (see for example \cite[Theorem~7.2.1]{alon-spencer}) with the edge-exposure martingale, the probability that \cref{eqn:lemma-graph-good-to-show} occurs is at most
\[\exp\left(-\Omega\left(\frac{n^{2h-2g-2}\cdot (\log n)^{2}}{n^2\cdot n^{2(h-g-2)}}\right)\right)= \exp\left(-\Omega\left((\log n\right)^{2})\right)= n^{-\omega(1)}.\]
This finishes the proof of \cref{lemma-graph-good}.
\end{proof}

Recall that we fixed $g,h,a_1,\dots,a_g$ and $t_1,\dots,t_{g-1}$ throughout this section, and that we defined $T=t_1\dotsm t_{g-1}$.

\begin{lem}\label{lem:medium-scale} For any essentially $p$-general complete restricted colour system $\mathcal{G}$ of order $n$ with parameters $(g, a_1,\dots,a_g, t_1,\dots,t_{g-1})$ and any $\lambda: [t_1]\times \dots\times [t_{g-1}]\times [1]\to \ZZ$ the following holds. If we choose $\mathcal{S}$ randomly as in \cref{defn:extension-restricted-colour-system} then we have
\[\Pr\left(\left\Vert\mu_{\mathcal{G}, \mathcal{S}}-\lambda\right\Vert_\infty\leq n^{h-g-1}\cdot \log n\right)\leq n^{-T/2+o(1)}.\]
\end{lem}

We defer the proof of \cref{lem:medium-scale} to \cref{sec:halasz-prep}.

\subsection{Joint anticoncentration}
\label{subsec:thm-implies-cor-key-lemma}

The following statement is the key lemma for the moment argument we will use to finish the proof of \cref{cor:dispersed}.

\begin{lem}\label{lem:key-lemma-proof-cor}
Fix $t\in \NN$. Then for any essentially $p$-general complete restricted colour system $\mathcal{G}$ of order $n$ with parameters $(g, a_1,\dots,a_g, t_1,\dots,t_{g-1})$, and any function $\lambda: [t_1]\times \dots\times [t_{g-1}]\times [1]\to \ZZ$, the following holds. If we choose, all independently, a random graph $G_0\sim\GG(\operatorname{U}(\mathcal{G}), p)$ on the vertex set $\operatorname{U}(\mathcal{G})$ as well as $t$ random collections $\mathcal{S}_1,\dots, \mathcal{S}_t$ chosen as in \cref{defn:extension-restricted-colour-system}, then
\[
\Pr\left(\psi_H(\mathcal{G}_{\mathcal{S}_i}, G_0, \cdot)=\lambda\text{ for all }i=1,\dots,t\right)\le n^{(g-h+(1/2))\cdot T\cdot t+o(1)}.
\]
\end{lem}

\begin{proof}
Fix an essentially $p$-general complete restricted colour system $\mathcal{G}$ of order $n$ with parameters $(g, a_1,\dots,a_g, t_1,\dots,t_{g-1})$, and a function $\lambda: [t_1]\times \dots\times [t_{g-1}]\times [1]\to \ZZ$.

For an outcome of the random collections $\mathcal{S}_1,\dots, \mathcal{S}_t$, we obtain a colour system $\mathcal{G}^*$ with parameters $(g, a_1,\dots,a_g, t_1,\dots,t_{g-1}, t)$ as follows. Recall that by \cref{defn:restricted-colour-system}, $\mathcal{G}$ has only one shade of colour $g$ and all edges of colour $g$ are between the vertices of colour $g$. Replace each of these edges by $t$ parallel edges with all shades $\{1,\dots,t\}$ of colour $g$. Also, for each $i=1,\dots,t$ and each vertex $v$ with colour $g$, add edges, in shade $i$ of colour $g$, between $v$ and all vertices in the set $S_{i,v}$ (where $S_{i,v}$ is the set corresponding to the vertex $v$ in the collection $\mathcal{S}_i$). We emphasise that $\mathcal{G}^*$ only depends on the random choices of $\mathcal{S}_1,\dots, \mathcal{S}_t$, and not on the random choice of $G_0\sim\GG(\operatorname{U}(\mathcal{G}), p)$.

Note that $\operatorname{U}(\mathcal{G}^*)=\operatorname{U}(\mathcal{G})$ and that for any $i=1,\dots, t$, deleting all edges of colour $g$ except those of shade $i$ yields precisely the colour system $\mathcal{G}_{\mathcal{S}_i}$ with parameters $(g, a_1,\dots,a_g, t_1,\dots,t_{g-1}, 1)$. Thus, for any outcome of $G_0\sim\GG(\operatorname{U}(\mathcal{G}), p)$, we have
\begin{equation}\label{eq:key-lemma-0}
\psi_H(\mathcal{G}_{\mathcal{S}_i}, G_0, j_1,\dots,j_{g-1},1)=\psi_H(\mathcal{G}^*, G_0, j_1,\dots,j_{g-1},i)
\end{equation}
for all $(j_1,\dots,j_{g-1})\in [t_1]\times\dots\times [t_{g-1}]$ and all $i=1,\dots,t$. Also note that $\mathcal{G}^*$ is always complete, because $\mathcal{G}$ is complete and being complete does not depend on the edges of colour $g$.

Say that an outcome of $(G_0,\mathcal{S}_1,\dots, \mathcal{S}_t)$ is \emph{common} if both of the following conditions hold.
\begin{itemize}
\item[(i)] we have
\[\left\|\psi_H(\mathcal{G}_{\mathcal{S}_i}, G_0, \cdot)-\mu_{\mathcal{G}, \mathcal{S}_i}\right\|_\infty\leq n^{h-g-1}\cdot \log n\]
for all $i=1,\dots,t$;
\item[(ii)] the colour system $\mathcal{G}^*$ given by $\mathcal{G}$ and $\mathcal{S}_1,\dots, \mathcal{S}_t$ is $p$-general.
\end{itemize}

\begin{claim}
\label{claim:common-outcome}
$(G_0,\mathcal{S}_1,\dots, \mathcal{S}_t)$ is common with probability $1-n^{-\omega(1)}$.
\end{claim}

\begin{proof}
By \cref{lemma-graph-good}, for each $i=1,\dots,t$ the probability that (i) fails is at most $n^{-\omega(1)}$. Thus, (i) holds with probability at least $1-t\cdot n^{-\omega(1)}=1-n^{-\omega(1)}$.

For (ii), note that by \cref{lem:extension-p-general} the colour system $\mathcal{G}_{\mathcal{S}_1}$ is $p$-general with probability $1-n^{-\omega(1)}$. We can then apply \cref{lem:sets-general-position-addition} $(t-1)\cdot a_g$ times for all the additional sets in $\mathcal S_2,\dots,\mathcal S_t$.
\end{proof}

Now, in order to prove \cref{lem:key-lemma-proof-cor}, we need to bound the probability that
\begin{equation}\label{eq:key-lemma-1}
\psi_H(\mathcal{G}_{\mathcal{S}_i}, G_0,\cdot)=\lambda\text{ for all $i\in [t]$.}
\end{equation}
We remark that if we condition on an outcome for $\mathcal{S}_1,\dots, \mathcal{S}_t$ satisfying (ii), then one can apply \cref{thm:general-subgraph-counts} to get an upper bound on this probability (using only the randomness of $G_0$). However, this bound is weaker than the one claimed in \cref{lem:key-lemma-proof-cor}. We will be able to obtain a stronger bound by using the randomness of both $G_0$ and of $\mathcal{S}_1,\dots, \mathcal{S}_t$. For the argument, both of the conditions (i) and (ii) will be relevant.

Whenever the outcome of $(G_0,\mathcal{S}_1,\dots, \mathcal{S}_t)$ is common (specifically, when (i) is satisfied), in order to satisfy \cref{eq:key-lemma-1} we must have
\begin{equation}\label{eq:key-lemma-new}
\left\Vert\mu_{\mathcal{G}, \mathcal{S}_i}-\lambda\right\Vert_\infty\leq n^{h-g-1}\cdot \log n\text{ for }i=1,\dots,t.
\end{equation}
Note that \cref{eq:key-lemma-new} does not depend on $G_0$, only on the random sets in $\mathcal S_1,\dots,\mathcal S_t$. 
By \cref{lem:medium-scale}, for each $i=1,\dots,t$ the probability of having $\left\Vert\mu_{\mathcal{G}, \mathcal{S}_i}-\lambda\right\Vert_\infty\leq n^{h-g-1}\cdot \log n$ is at most $n^{-T/2+o(1)}$. So, by the independence of the $\mathcal S_i$, the probability that \cref{eq:key-lemma-new} holds is bounded as follows.
\begin{equation}\label{eq:key-lemma-2}
\Pr\left(\left\Vert\mu_{\mathcal{G}, \mathcal{S}_i}-\lambda\right\Vert_\infty\leq n^{h-g-1}\cdot \log n\text{ for }i=1,\dots,t\right)\leq \left( n^{-T/2+o(1)}\right)^t=n^{-T\cdot t/2+o(1)}.
\end{equation}
Now, fix any outcomes of $\mathcal{S}_1,\dots, \mathcal{S}_t$ 
satisfying (ii) and \cref{eq:key-lemma-new}. By condition (ii), the colour system $\mathcal{G}^*$ is $p$-general, and furthermore recall that $\mathcal{G}^*$ is complete. Hence, applying \cref{thm:general-subgraph-counts} to the colour system $\mathcal{G}^*$ (which has parameters $(g, a_1,\dots,a_g, t_1,\dots,t_{g-1}, t)$), conditioned on our outcomes of $\mathcal{S}_1,\dots, \mathcal{S}_t$, we have
\[\Pr\left(\psi_H(\mathcal{G^*}, G_0,\cdot,i)=\lambda(\cdot,1)\text{ for }i=1,\dots,t\right)\le n^{(g-h+1)\cdot t_1\dotsm t_{g-1}\cdot t+o(1)}=n^{(g-h+1)\cdot T\cdot t+o(1)}.\]
(Recall that in this section we are assuming that \cref{thm:general-subgraph-counts} holds for $g$).
In other words, recalling \cref{eq:key-lemma-0}, if we condition on any outcomes of $\mathcal{S}_1,\dots, \mathcal{S}_t$ such that 
(ii) and \cref{eq:key-lemma-new} hold,
then the random choice $G_0\sim\GG(\operatorname{U}(\mathcal{G}), p)$ satisfies \cref{eq:key-lemma-1} with probability at most $n^{(g-h+1)\cdot T\cdot t+o(1)}$.

Combining this with \cref{eq:key-lemma-2}, we see that the probability that $(G_0,\mathcal{S}_1,\dots, \mathcal{S}_t)$ simultaneously satisfies \cref{eq:key-lemma-new}, (ii) and \cref{eq:key-lemma-1} is at most
\[n^{-T\cdot t/2+o(1)}\cdot n^{(g-h+1)\cdot T\cdot t+o(1)}=n^{(g-h+(1/2))\cdot T\cdot t+o(1)}.\]
Recall that any common outcome satisfying condition \cref{eq:key-lemma-1} also needs to satisfy \cref{eq:key-lemma-new}, and recall from \cref{claim:common-outcome} that $(G_0,\mathcal{S}_1,\dots, \mathcal{S}_t)$ is common with probability at least $1-n^{-\omega(1)}$. Thus, the total probability that \cref{eq:key-lemma-1}  holds is at most $n^{(g-h+(1/2))\cdot T\cdot t+o(1)}+n^{-\omega(1)}=n^{(g-h+(1/2))\cdot T\cdot t+o(1)}$. This finishes the proof of \cref{lem:key-lemma-proof-cor}.
\end{proof}

In the proof of \cref{cor:dispersed}, we will use the following corollary of \cref{lem:key-lemma-proof-cor}. In contrast to the setting of \cref{lem:key-lemma-proof-cor}, in \cref{lem:key-lemma-proof-cor-cor} we will not require $t\in \NN$ to be fixed, but just that $n$ is sufficiently large with respect to $t$ (and with respect to the values for $g,h,a_1,\dots,a_g$ and $t_1,\dots,t_{g-1}$ that we fixed for this entire section). Also note that the statement of \cref{lem:key-lemma-proof-cor-cor} does not contain any asymptotic notation.

\begin{cor}
\label{lem:key-lemma-proof-cor-cor}
As before, let $g,h,a_1,\dots,a_g$ and $t_1,\dots,t_{g-1}$ be fixed. Then for every $t\in \NN$, there exists $N(t)\in \NN$ such that the following holds for any essentially $p$-general complete restricted colour system $G$ of order $n\geq N(t)$ with parameters $(g,a_1,\dots,a_g,t_1,\dots,t_{g-1})$, and any function $\lambda: [t_1]\times \dots\times [t_{g-1}]\times [1]\to \ZZ$. If we choose, all independently, a random graph $G_0\sim\GG(\operatorname{U}(\mathcal{G}), p)$ on the vertex set $\operatorname{U}(\mathcal{G})$ as well as $t$ random collections $\mathcal{S}_1,\dots, \mathcal{S}_t$ chosen as in \cref{defn:extension-restricted-colour-system}, then
\[
\Pr\left(\psi_H(\mathcal{G}_{\mathcal{S}_i}, G_0, \cdot)=\lambda\text{ for all }i=1,\dots,t\right)\le n^{(g-h+(1/2))\cdot T\cdot t+1}.
\]
\end{cor}

\begin{proof}
For each $t\in \NN$, the $o(1)$-term in the statement of \cref{lem:key-lemma-proof-cor} with this particular value of $t$ converges to zero (as $n$ goes to infinity). Hence there is some $N(t)\in \NN$ such that this $o(1)$-term is at most $1$ for all $n\geq N(t)$. In other words, this means that whenever $n\geq N(t)$, the probability appearing in the statements of \cref{lem:key-lemma-proof-cor} and \cref{lem:key-lemma-proof-cor-cor} is indeed at most $n^{(g-h+(1/2))\cdot T\cdot t+1}$, as desired.
\end{proof}

\subsection{Completing the proof of \texorpdfstring{\cref{cor:dispersed}}{Proposition~\ref{cor:dispersed}}}
\label{subsec:cor-completion}

In this subsection, we will finally deduce \cref{cor:dispersed} for our given value of $g$. In order to do so, we will apply \cref{lem:key-lemma-proof-cor-cor} with a carefully chosen value of $t$. Recall that $t$ is not required to be fixed in \cref{lem:key-lemma-proof-cor-cor}, it is only required that $n$ is sufficiently large with respect to $t$ (and the values of $g,a_1,\dots,a_g$ and $t_1,\dots,t_{g-1}$ that we fixed throughout this section). Let $N:\NN\to\NN$ be a function mapping each $t\in \NN$ to some $N(t)\in \NN$ such that \cref{lem:key-lemma-proof-cor-cor} holds.

In order to prove \cref{cor:dispersed}, we can assume that $n\geq N(1)$. For any integer $n\geq N(1)$, let us define $t(n)$ to be the maximum $t\in \NN$ with $N(t)\leq n$. Note that then we have $N(t(n))\leq n$ for all $n\geq N(1)$, and that $t(n)=\omega(1)$.

Define the function $\delta$ by $\delta(n)=t(n)^{-1/2}$ for all $n\geq N(1)$, and observe that $\lim_{n\to\infty}\delta(n)=0$. Let $\mathcal{G}$ be any essentially $p$-general complete restricted colour system $\mathcal{G}$ of order $n\geq N(1)$ with parameters $(g, a_1,\dots,a_g, t_1,\dots,t_{g-1})$. We need to show that a randomly chosen graph $G_0\sim\GG(\operatorname{U}(\mathcal{G}), p)$ on the vertex set $\operatorname{U}(\mathcal{G})$ is $(p,q,\mathcal{G})$-dispersed with probability $1-n^{-\omega(1)}$, where
\[q=n^{(g-h+(1/2))\cdot T+\delta(n)}.\]
Recall from \cref{defn:dispersedness} that a graph $G_0$ on the vertex set $\operatorname{U}(\mathcal{G})$ being $(p,q,\mathcal{G})$-dispersed means that for all functions $\lambda: [t_1]\times \dots\times [t_{g-1}]\times [1]\to \ZZ$ the following condition holds: when choosing $\mathcal{S}$ randomly as in \cref{defn:extension-restricted-colour-system}, $\Pr\left(\psi_H(\mathcal{G_S}, G_0,\cdot)=\lambda\right)\le q$. Let $A_{\lambda}$ be the event that $G_0$ fails to satisfy this condition (i.e. that the probability is strictly larger than $q$) for some particular $\lambda$. So, we need to show that with probability $1-n^{-\omega(1)}$, no $A_{\lambda}$ occurs. Note that we always have $0\le \psi_H(\mathcal{G_S}, G_0, j_1,\dots,j_g)\leq n^h$, so we only need to consider $(n^h+1)^T$ possibilities for $\lambda$. It therefore suffices to show that $\Pr(A_\lambda)=n^{-\omega(1)}$ for each $\lambda$: we then can take a union bound over all possibilities for $\lambda$. So, fix a function $\lambda: [t_1]\times \dots\times [t_{g-1}]\times [1]\to \lbrace 0,\dots,n^h\rbrace$.

Let $t=t(n)$. All independently, choose a random graph $G_0\sim\GG(\operatorname{U}(\mathcal{G}), p)$, and choose $t$ collections $\mathcal{S}_1,\dots, \mathcal{S}_t$ as in \cref{defn:extension-restricted-colour-system}. By \cref{lem:key-lemma-proof-cor-cor} (using that $N(t)=N(t(n))\leq n$), we have
\begin{equation}
\Pr\left(\psi_H(\mathcal{G}_{\mathcal{S}_i}, G_0, \cdot)=\lambda\text{ for all }i=1,\dots,t\right)\le n^{(g-h+(1/2))\cdot T\cdot t+1}.\label{eq:joint-probability}
\end{equation}
We remark that the probability on the left-hand side of \cref{eq:joint-probability} can be interpreted as the $t$-th moment of the random variable (depending on $G_0\sim\GG(\operatorname{U}(\mathcal{G}), p)$) measuring the conditional probability that $\psi_H(\mathcal{G}_{\mathcal{S}}, G_0,\cdot)=\lambda$, for a random choice of $\mathcal S$ as in \cref{defn:extension-restricted-colour-system}. The rest of the proof, to follow, can basically be interpreted as applying Markov's inequality with this $t$-th moment.

By the definition of $A_\lambda$, if we condition on an outcome of $G_0$ for which $A_\lambda$ holds, then for each $i$, with probability at least $q$ we have $\psi_H(\mathcal{G}_{\mathcal{S}_i}, G_0, \cdot)=\lambda$. Since the $\mathcal{S}_i$ are independent, we deduce that \[\Pr\left(\psi_H(\mathcal{G}_{\mathcal{S}_i}, G_0, \cdot)=\lambda\text{ for }i=1,\dots,t\right)\ge \Pr(A_{\lambda})\cdot \Pr\left(\psi_H(\mathcal{G}_{\mathcal{S}_i}, G_0, \cdot)=\lambda\text{ for } i=1,\dots,t\mid A_{\lambda}\right)\ge\Pr(A_{\lambda})q^t.\]

Together with \cref{eq:joint-probability}, this implies
\[\Pr\left(A_{\lambda}\right)\cdot q^t\leq n^{(g-h+(1/2))\cdot T\cdot t+1},\]
and therefore, recalling $q=n^{(g-h+(1/2))\cdot T+\delta(n)}$,
\[\Pr\left(A_{\lambda}\right)\leq \frac{n^{(g-h+(1/2))\cdot T\cdot t+1}}{n^{(g-h+(1/2))\cdot T\cdot t+\delta(n)\cdot t}}=n^{1-\delta(n)\cdot t}=n^{1-\sqrt t}=n^{-\omega(1)}.\]
(Recall that $\delta(n)=t(n)^{-1/2}=t^{-1/2}$ and that $t=t(n)=\omega(1)$). This concludes the proof.

\section{Completing the induction step: \texorpdfstring{\cref{cor:dispersed}}{Proposition~\ref{cor:dispersed}} for $g$ implies \texorpdfstring{\cref{thm:general-subgraph-counts}}{Theorem~\ref{thm:general-subgraph-counts}} for $g-1$}
\label{sect:cor-implies-thm}

For the duration of this section, we fix some $1\leq g\leq h-1$ and assume that \cref{cor:dispersed} holds for this value of $g$. Our goal for this section is to prove \cref{thm:general-subgraph-counts} for $g-1$. Our approach is as outlined in \cref{sec:outline}: we decompose $\psi_H(\mathcal{G}, G_0, \cdot)$ into two parts, use a ``coarse-scale'' anticoncentration lemma to handle the larger part, and use our assumed case of \cref{cor:dispersed} to handle the smaller part.

Let $a_1,\dots,a_{g-1}, t_1,\dots,t_{g-1}$ be arbitrary positive integers. Our goal is to prove \cref{thm:general-subgraph-counts} for $g-1$ and these values of $a_1,\dots,a_{g-1}, t_1,\dots,t_{g-1}$.

For this entire section, let us fix the values of $a_1,\dots,a_{g-1}, t_1,\dots,t_{g-1}$, and define $T=t_1\dotsm t_{g-1}$ and $a_g=2^{a_1t_1+\dots+a_{g-1}t_{g-1}}$. In all our asymptotics, these values will be treated as constants, while $n\to \infty$. In \cref{subsec:completion-prep} we will prove a concentration lemma which we will use to control the fluctuation of the smaller of our two parts. We also state (but do not yet prove) an anticoncentration lemma which we will use for the larger of our two parts. In \cref{subsec:completion-proof} we combine these lemmas with our assumed case of \cref{cor:dispersed} to prove our desired case of \cref{thm:general-subgraph-counts}. 

\subsection{Preparations}
\label{subsec:completion-prep}

\begin{defn}\label{defn:mu-G}
Consider an essentially $p$-general complete restricted colour system $\mathcal{G}$ with parameters $(g, a_1,\dots,a_g, t_1,\dots,t_{g-1})$. Let $\mu_{\mathcal{G}}: [t_1]\times \dots\times [t_{g-1}]\times [1]\to \RR$ be the function given by
\[\mu_{\mathcal{G}}(j_1,\dots,j_g)=\E_{G_0, \mathcal{S}}[\psi_H(\mathcal{G_S}, G_0, j_1,\dots,j_g)]\]
for all $(j_1,\dots,j_g)\in [t_1]\times \dots\times [t_{g-1}]\times [1]$. Here, $G_0\sim\GG(\operatorname{U}(\mathcal{G}), p)$ is a random graph on the vertex set $\operatorname{U}(\mathcal{G})$ and $\mathcal{S}=(S_v)_{v}$ is a randomly chosen collection of sets as in \cref{defn:extension-restricted-colour-system} (and $G_0$ and $\mathcal{S}$ are chosen independently).
\end{defn}

We remark that $\mu_{\mathcal{G}}$ only depends on $\mathcal{G}$.

\begin{lem}
\label{lemma-graph-good-2} The following holds for any essentially $p$-general complete restricted colour system $\mathcal{G}$ with parameters $(g, a_1,\dots,a_g, t_1,\dots,t_{g-1})$. If we choose a random graph $G_0\sim\GG(\operatorname{U}(\mathcal{G}), p)$ and independently choose $\mathcal{S}$ randomly as in \cref{defn:extension-restricted-colour-system}, then
\[\Pr\left(\left\|\psi_H(\mathcal{G_S}, G_0,\cdot)-\mu_{\mathcal{G}}\right\|_\infty\leq n^{h-g-(1/2)}\cdot \log n\right)=1-n^{-\omega(1)}.\]
\end{lem}
\begin{proof}
By a union bound, it suffices to prove that for each $g$-tuple $(j_1,\dots,j_g)\in [t_1]\times \dots\times [t_{g-1}]\times [1]$ the probability to have
\begin{equation}\label{eqn:lemma-graph-good-2-to-show}
\left|\psi_H(\mathcal{G_S}, G_0, j_1,\dots,j_g)-\mu_{\mathcal{G}}(j_1,\dots,j_g)\right|>n^{h-g-(1/2)}\cdot \log n
\end{equation}
is $n^{-\omega(1)}$. So fix some $g$-tuple $(j_1,\dots,j_g)\in [t_1]\times \dots\times [t_{g-1}]\times [1]$. Note that the expectation of $\psi_H(\mathcal{G_S}, G_0, j_1,\dots,j_g)$ is precisely $\mu_{\mathcal{G}}(j_1,\dots,j_g)$.

Note that the random collection $\mathcal{S}=(S_v)_v$ as in \cref{defn:extension-restricted-colour-system} consists of $a_g$ different sets $S_v\su \operatorname{U}(\mathcal{G})$, one for each vertex of colour $g$. For any vertex $u\in \operatorname{U}(\mathcal{G})$, changing which of the sets $S_v$ the vertex $u$ belongs to, changes the edges in the graph $\mathcal{G_S}(G_0, j_1,\dots,j_g)$ between $u$ and the vertices of colour $g$. Hence it changes the value of $\psi_H(\mathcal{G_S}, G_0, j_1,\dots,j_g)$ by at most $h^{g+1} \cdot a_1\dotsm a_{g}\cdot n^{h-g-1}=O(n^{h-g-1})$ (since there are at most that many labelled copies of $H$ in the $n$-vertex graph $\mathcal{G_S}(G_0, j_1,\dots,j_g)$ which contain $u$ as well as at least one vertex of each of the $g$ colours).

Consider the Doob martingale with respect to $\psi_H(\mathcal{G_S}, G_0, j_1,\dots,j_g)$ obtained by first exposing the graph $G_0\sim\GG(\operatorname{U}(\mathcal{G}), p)$ one edge at a time, and then exposing the random collection $\mathcal{S}$ in the following way. In each step, for one vertex $u\in \operatorname{U}(\mathcal{G})$, we expose which of the sets $S_v$ of the collection $\mathcal{S}$ contain the vertex $u$. As in the proof of \cref{lemma-graph-good}, changing the status of an edge of $G_0$ changes $\psi_H(\mathcal{G_S}, G_0, j_1,\dots,j_g)$ by at most $h^{g+2} \cdot a_1\dotsm a_{g}\cdot n^{h-g-2}=O(n^{h-g-2})$. So, by the Azuma--Hoeffding inequality the probability that \cref{eqn:lemma-graph-good-2-to-show} occurs is at most
\[\exp\left(-\Omega\left(\frac{n^{2(h-g-1/2)}\cdot (\log n)^{2}}{n^2\cdot n^{2(h-g-2)}+n\cdot n^{2(h-g-1)}}\right)\right)\leq \exp\left(-\Omega\left((\log n)^{2}\right)\right)= n^{-\omega(1)}.\]
This finishes the proof of \cref{lemma-graph-good-2}.
\end{proof}

\begin{lem}
\label{lem:rough-scale} The following holds for any weakly $p$-general complete colour system $\mathcal{G}$ of order $n$ with parameters $(g-1, a_1,\dots,a_{g-1}, t_1,\dots,t_{g-1})$ and any $\lambda: [t_1]\times \dots\times [t_{g-1}]\to \ZZ$. If we choose a random graph $G_0\sim\GG(\operatorname{U}(\mathcal{G}), p)$ on the vertex set $\operatorname{U}(\mathcal{G})$ then
\[\Pr\left(\left\|\psi_H(\mathcal{G}, G_0, \cdot)-\lambda\right\|_\infty\leq n^{h-g-(1/2)}\cdot \log n\right)\le n^{-T/2+o(1)}.\]
\end{lem}

We defer the proof of \cref{lem:rough-scale} to \cref{sec:halasz-prep}.

\subsection{Proof of \texorpdfstring{\cref{thm:general-subgraph-counts}}{Theorem~\ref{thm:general-subgraph-counts}} for $g-1$}
\label{subsec:completion-proof}

In this subsection, we deduce \cref{thm:general-subgraph-counts} for $g-1$ from \cref{cor:dispersed} for $g$. Consider any $p$-general complete colour system $\mathcal{G}$ of order $n$ with parameters $(g-1, a_1,\dots,a_{g-1}, t_1,\dots,t_{g-1})$ and any function $\lambda: [t_1]\times \dots\times [t_{g-1}]\to \ZZ$. We may assume that $n$ is sufficiently large with respect to the parameters $g-1, a_1,\dots,a_{g-1}, t_1,\dots,t_{g-1}$ that were fixed throughout this section.

Since $\mathcal{G}$ is a $p$-general colour system, the $\operatorname{U}(\mathcal{G})$-neighbourhoods of each of the coloured vertices of $\mathcal{G}$ in each of the shades of the respective colour are in $p$-general position. This implies that there exist vertices in $\operatorname{U}(\mathcal G)$ representing all possible ways to be adjacent to the coloured vertices. To be specific, for each $1\leq j\leq g-1$ and each vertex $v$ in $\mathcal{G}$ of colour $j$, consider any subset $I_v\su [t_j]$. There is a vertex $w\in \operatorname{U}(\mathcal{G})$ such that for every $1\leq j\leq g-1$ and every vertex $v$ of colour $j$, between $w$ and $v$ there are edges of exactly those shades of colour $j$ that belong to the set $I_{v}$. For each choice of the subsets $I_v$, fix a particular such vertex $w$, and let $W$ be the set of all these fixed vertices $w$. Then, $W$ has size
\[(2^{t_1})^{a_1}\dotsm (2^{t_{g-1}})^{a_{g-1}}=2^{a_1t_1+\dots+a_{g-1}t_{g-1}}=a_g.\]
To prove \cref{thm:general-subgraph-counts}, we need to show that for a random graph $G_0\sim\GG(\operatorname{U}(\mathcal{G}), p)$, we have
\[
\Pr\left(\psi_H(\mathcal{G}, G_0,\cdot)=\lambda\right)\le n^{((g-1)-h+1)\cdot T+o(1)}=n^{(g-h)\cdot T+o(1)}.
\]
Fix any possible outcome of the edges of the graph $G_0[W]$ between the vertices in $W$. We will prove that the event $\psi_H(\mathcal{G}, G_0,\cdot)=\lambda$ occurs with probability at most $n^{(g-h)\cdot T+o(1)}$ conditioned on this outcome.

We can obtain a colour system $\mathcal{G}^-$ of order $n-a_g$ with parameters $(g-1, a_1,\dots,a_{g-1},t_1,\dots,t_{g-1})$ from the colour system $\mathcal{G}$ by deleting the $a_g$ vertices in $W\su \operatorname{U}(\mathcal{G})$. By \cref{lemma-general-deletion-weakly}, since $\mathcal{G}$ is $p$-general, $\mathcal{G}^-$ is weakly $p$-general. Also, $\mathcal{G}^-$ is complete, because  $\mathcal{G}$ is complete and being complete does not depend on any of the uncoloured vertices.

Furthermore, from $\mathcal{G}$ and our fixed outcome of $G_0[W]$, we can obtain a restricted colour system $\mathcal{G}'$ of order $n$ with parameters $(g, a_1,\dots,a_{g-1},a_g, t_1,\dots,t_{g-1})$ in the following way. First, colour all the $a_g$ vertices in $W\su \operatorname{U}(\mathcal{G})$ with colour $g$. Then, take a single shade of colour $g$ and colour all edges in $G_0[W]$ with this single shade of colour $g$. The restricted colour system $\mathcal{G}'$ so obtained is complete and essentially $p$-general, by our choice of $W$ and the assumptions that $\mathcal{G}$ is complete and $p$-general.

Given our fixed outcome of $G_0[W]$, we can choose the rest of the random graph $G_0\sim\GG(\operatorname{U}(\mathcal{G}), p)$ in two steps as follows. First, we choose a random graph $G_0^{-}\sim\GG(\operatorname{U}(\mathcal{G})\sm W, p)$ on the vertex set $\operatorname{U}(\mathcal{G})\sm W$. Then, choose random subsets $S_v \su \operatorname{U}(\mathcal{G})\sm W$ for each vertex $v\in W$ by taking each element of $\operatorname{U}(\mathcal{G})\sm W$ independently with probability $p$ (and choose all these sets $S_v$ independently of each other and independent of $G_0^-$). Now, take $G_0$ to be the graph on the vertex set $\operatorname{U}(\mathcal{G})$ obtained by starting with the union of the edges of the graphs $G_0^-$ and $G_0[W]$, and adding edges between each vertex $v\in W$ and all the vertices in $S_v\su \operatorname{U}(\mathcal{G})\sm W$.

Let $\mathcal{S}=(S_v)_{v\in W}$ be the collection of random sets chosen above. Note that the random choice of $\mathcal{S}$ is precisely what is described in \cref{defn:extension-restricted-colour-system} for the restricted colour system $\mathcal{G'}$ (recall that $W$ is the set of vertices of colour $g$ in the restricted colour system $\mathcal{G'}$).

Then, for any outcome of $\mathcal{S}$ and $G_0^{-}\sim\GG(\operatorname{U}(\mathcal{G})\sm W, p)$, and any $(j_1,\dots,j_{g-1})\in [t_1]\times\dots\times [t_{g-1}]$, the graph $\mathcal{G}(G_0, j_1,\dots,j_{g-1})$ is the same as the graph $\mathcal{G'_S}(G_0^-, j_1,\dots,j_{g-1},1)$. (Here $\mathcal{G'_S}$ is the colour system obtained from the restricted colour system $\mathcal{G'}$ as in \cref{defn:extension-restricted-colour-system}.) Recall that by definition $\psi_H(\mathcal{G}, G_0, j_1,\dots,j_{g-1})$ is the number of labelled copies of $H$ in the graph $\mathcal{G}(G_0, j_1,\dots,j_{g-1})$ which use at least one vertex of each of the colours $1,\dots,g-1$. If such a copy of $H$ contains at least one vertex in $W$, then it is one of the $\psi_H(\mathcal{G'_S}, G_0^-, j_1,\dots,j_{g-1}, 1)$ labelled copies of $H$ in the graph $\mathcal{G'_S}(G_0^-, j_1,\dots,j_{g-1},1)=\mathcal{G}(G_0, j_1,\dots,j_{g-1})$ which use at least one vertex of each of the $g$ colours of $\mathcal{G}'$. On the other hand, if it does not contain a vertex of $W$, then it is one of the $\psi_H(\mathcal{G}^-, G_0^-, j_1,\dots,j_{g-1})$ labelled copies of $H$ in the graph $\mathcal{G^-}(G_0^-, j_1,\dots,j_{g-1})=\mathcal{G}(G_0, j_1,\dots,j_{g-1})\sm W$ which use at least one vertex of each of the colours $1,\dots,g-1$. Thus, for any outcome of $\mathcal{S}$ and $G_0^{-}$, we have
\[\psi_H(\mathcal{G}, G_0, \cdot)=\psi_H(\mathcal{G'_S}, G_0^-, \cdot, 1)+\psi_H(\mathcal{G}^-, G_0^-, \cdot).\]
For random $\mathcal{S}=(S_v)_{v\in W}$ and $G_0^{-}\sim\GG(\operatorname{U}(\mathcal{G})\sm W, p)$ as above, it therefore suffices to prove that
\begin{equation}\label{eqn:proof-thms-2}
\Pr\left(\psi_H(\mathcal{G'_S}, G_0^-, \cdot, 1)+\psi_H(\mathcal{G}^-, G_0^-, \cdot)
=\lambda\right)\le n^{(g-h)\cdot T+o(1)}.
\end{equation}
Let $\delta:\NN\to \RR_{\geq 0}$ be the function obtained by applying \cref{cor:dispersed} with parameters $g, a_1,\dots,a_g$, $t_1,\dots,t_{g-1}$ (recall that we are assuming that \cref{cor:dispersed} holds for $g$). So $\delta(n)=o(1)$. Call an outcome of $(\mathcal{S},G_0^{-})$ \emph{typical} if the following two conditions hold.
\begin{itemize}
\item[(a)] For $q=n^{(g-h+(1/2))\cdot T+\delta(n)}$, the graph $G_0^-$ is $(p,q,\mathcal{G'})$-dispersed;
\item[(b)] we have $\left\|\psi_H(\mathcal{G'_S}, G_0^-,\cdot)-\mu_{\mathcal{G'}}\right\|_\infty\leq n^{h-g-(1/2)}\cdot \log n$.
\end{itemize}

\begin{claim}\label{claim:typical-outcome}
$(\mathcal{S},G_0^{-})$ is typical with probability $1-n^{-\omega(1)}$.
\end{claim}
\begin{proof}
Recall that $\mathcal{G}'$ is an essentially $p$-general complete restricted colour system of order $n$ with parameters $(g, a_1,\dots,a_{g-1},a_g, t_1,\dots,t_{g-1})$. Its set of uncoloured vertices is $\operatorname{U}(\mathcal{G})\sm W$. Thus, by \cref{cor:dispersed} (which we assumed to hold for $g$), the random graph $G_0^{-}\sim\GG(\operatorname{U}(\mathcal{G})\sm W, p)$ is $(p,q,\mathcal{G'})$-dispersed with probability $1-n^{-\omega(1)}$. This shows that (a) holds with probability $1-n^{-\omega(1)}$.

On the other hand, (b) holds with probability $1-n^{-\omega(1)}$ by \cref{lemma-graph-good-2} applied to the restricted colour system $\mathcal{G'}$.
\end{proof}

Note that whenever an outcome of $(\mathcal{S},G_0^{-})$ is typical (specifically, whenever (b) holds), we cannot have $\psi_H(\mathcal{G'_S}, G_0^-, \cdot, 1)+\psi_H(\mathcal{G}^-, G_0^-, \cdot)
=\lambda$ unless
\begin{equation}\label{eqn:proof-thms-3}
\left\|\psi_H(\mathcal{G}^-, G_0^-, \cdot)+\mu_{\mathcal{G'}}(\cdot, 1)-\lambda\right\|_\infty\\
\leq n^{h-g-(1/2)}\cdot \log n\leq  2\cdot (n-a_g)^{h-g-(1/2)}\cdot \log (n-a_g).
\end{equation}
(here we used that $n$ is sufficiently large with respect to $a_g=2^{a_1t_1+\dots+a_{g-1}t_{g-1}}$).

Consider the function $\lambda': [t_1]\times \dots\times [t_{g-1}]\to \ZZ$ defined by $\lambda'( j_1,\dots, j_{g-1})=\lambda( j_1,\dots, j_{g-1})-\mu_{\mathcal{G'}}(j_1,\dots,j_{g-1}, 1)$. By \cref{lem:rough-scale} applied with the function $\lambda'$ and the weakly $p$-general complete colour system $\mathcal{G}^-$ of order $n-a_g$ with parameters $(g-1, a_1,\dots,a_{g-1}, t_1,\dots,t_{g-1})$, we see that \cref{eqn:proof-thms-3} holds with probability at most $(n-a_g)^{-T/2+o(1)}=n^{-T/2+o(1)}$.

Note that both \cref{eqn:proof-thms-3} and (a) only depend on the random choice of $G_0^-$ and not on the random choice of $\mathcal{S}$. For any outcome of $G_0^-$ such that \cref{eqn:proof-thms-3} and (a) hold, the conditional probability of the event that 
$\psi_H(\mathcal{G'_S}, G_0^-, \cdot, 1)=\lambda-\psi_H(\mathcal{G}^-, G_0^-, \cdot)$  is at most $q=n^{(g-h+(1/2))\cdot T+\delta(n)}$. (This follows directly from the definition of being $(p,q,\mathcal{G'})$-dispersed; see \cref{defn:dispersedness}).

Thus, the total probability that $(G_0^-,\mathcal{S})$ is typical and satisfies $\psi_H(\mathcal{G'_S}, G_0^-, \cdot, 1)+\psi_H(\mathcal{G}^-, G_0^-, \cdot)
=\lambda$ is at most
\[n^{-T/2+o(1)}\cdot n^{(g-h+(1/2))\cdot T+\delta(n)}=n^{(g-h)\cdot T+o(1)},\]
recalling that $\delta(n)=o(1)$. By \cref{claim:typical-outcome}, the probability that $(G_0^-,\mathcal{S})$ is not typical is $n^{-\omega(1)}$, so the probability in \cref{eqn:proof-thms-2} is at most $n^{(g-h)\cdot T+o(1)}+n^{-\omega(1)}=n^{(g-h)\cdot T+o(1)}$. This finishes the proof of \cref{thm:general-subgraph-counts} for $g-1$.

\section{Cores and non-degeneracy}
\label{sec:cores}

It remains to prove \cref{lem:medium-scale,lem:rough-scale}. Both of these lemmas will be proved using our new multivariate anticoncentration inequality (\cref{thm:rough-halasz}), which requires us to check a non-degeneracy condition for a certain collection of vectors. In this section we introduce the formalism of \emph{cores}, which are special types of colour systems of bounded size. For a core $\mathcal C$, we will define certain functions $\Gamma_{\mathcal{C},e}$ (which may be interpreted as belonging to a vector space of functions), and we show that under certain conditions these functions span the entire space.

The point of this section is that in the settings of both \cref{lem:medium-scale} and \cref{lem:rough-scale}, the collections of vectors that we need to study to apply \cref{thm:rough-halasz} are in correspondence with collections of functions $\Gamma_{\mathcal{C},e}$, for appropriately chosen cores $\mathcal C$. In the next section (\cref{sec:halasz-prep}) we will specify how to actually choose the cores for these correspondences. Throughout this section, fix any $0\le g\le h-1$ (recall that $h$ is the number of vertices of our fixed graph $H$).

\begin{defn}\label{def:core}
For integers $0\le g\le h-1$ and $a_1,\dots,a_g, t_1,\dots,t_{g-1}\geq 1$, a \emph{core} $\mathcal{C}$ with parameters $(g, a_1,\dots,a_g, t_1,\dots,t_{g-1})$ is a colour system with parameters $(g, a_1,\dots,a_g, t_1,\dots,t_{g-1}, 1)$ which has exactly one uncoloured vertex, such that the uncoloured vertex has edges to all coloured vertices in all possible shades (meaning that for each $1\leq i\leq g$, the uncoloured vertex has edges in all $t_i$ shades of colour $i$ to all vertices of colour $i$).
\end{defn}

Note that for fixed parameters $(g, a_1,\dots,a_g, t_1,\dots,t_{g-1})$ there are only finitely many different cores $\mathcal{C}$ with these parameters.

A core $\mathcal{C}$ with  parameters $(g, a_1,\dots,a_g, t_1,\dots,t_{g-1})$ is called \emph{complete} if it is complete when viewed as a colour system with parameters $(g, a_1,\dots,a_g, t_1,\dots,t_{g-1}, 1)$.

\begin{defn}
\label{def:partial-copy} A \emph{partial copy} of $H$ in a core $\mathcal{C}$ with parameters $(g, a_1,\dots,a_g, t_1,\dots,t_{g-1})$ is a labelled copy of a $(g+1)$-vertex induced subgraph of $H$ which uses one vertex of each colour in $\mathcal{C}$ and the uncoloured vertex. More formally, a partial copy of $H$ in a core $\mathcal{C}$ is given by a graph homomorphism $\phi: H[V']\to \mathcal{C}$ for some subset $V'\su V(H)$ of size $g+1$ such that the image of $\phi$ contains one vertex of each colour and the uncoloured vertex.
\end{defn}

Note that the homomorphism $\phi: H[V']\to \mathcal{C}$ in \cref{def:partial-copy} is automatically injective on vertices (and therefore also on edges).

\begin{defn}
Consider a core $\mathcal{C}$ with parameters $(g, a_1,\dots,a_g, t_1,\dots,t_{g-1})$. For a subset $V'\su V(H)$ of size $g+1$, define the \emph{weight} of a partial copy $\phi: H[V']\to \mathcal{C}$ of $H$ to be $p^{e(H)-e_H(V')}$. Also, for a $g$-tuple $(j_1,\dots,j_g)\in [t_1]\times\dots\times[t_{g-1}]\times [1]$, say that a partial copy $\phi: H[V']\to \mathcal{C}$ of $H$ in $\mathcal{C}$ is \emph{$(j_1,\dots,j_g)$-coloured} if for each $1\leq i\leq g$ all the edges of colour $i$ in the image of $\phi$ have shade $j_i$.
\end{defn}

Note that if a partial copy $\phi: H[V']\to \mathcal{C}$ of $H$ in $\mathcal{C}$ does not contain any edges of colour $i$ in its image, then it can be $(j_1,\dots,j_g)$-coloured for different values of $j_i$. However, if $\phi$ has at least one edge of colour $i$ in its image, then it can only be $(j_1,\dots,j_g)$-coloured if $j_i$ is the shade of the edges of colour $i$ in the image of $\phi$ (if there are different edges of colour $i$ with different shades in the image of $\phi$, then $\phi$ is not $(j_1,\dots,j_g)$-coloured for any $(j_1,\dots,j_g)$).

\begin{defn}\label{defn:gamma}
For a core $\mathcal{C}$ with parameters $(g, a_1,\dots,a_g, t_1,\dots,t_{g-1})$ and a collection of edges $F\su E(\mathcal{C})$, let $\Gamma_{\mathcal{C},F}: [t_1]\times\dots\times[t_{g-1}]\times [1]\to \mathbb{R}$ be the function defined as follows. For all $(j_1,\dots,j_g)\in [t_1]\times\dots\times[t_{g-1}]\times [1]$, let $\Gamma_{\mathcal{C},F}(j_1,\dots,j_g)$ be the sum of the weights of all $(j_1,\dots,j_g)$-coloured partial copies of $H$ in $\mathcal{C}$ whose image contains all edges in $F$. If $F$ just consists of one edge $e$, we write $\Gamma_{\mathcal{C},e}$ instead of $\Gamma_{\mathcal{C},\lbrace e\rbrace}$.
\end{defn}

Note that we have $\Gamma_{\mathcal{C},F}(j_1,\dots,j_g)=0$ if for some $1\leq i\leq g$ the set $F$ contains an edge of colour $i$ with a shade distinct from $j_i$, because then there are no $(j_1,\dots,j_g)$-coloured partial copies of $H$ in $\mathcal{C}$ whose image contains this edge.

Now, the main result of this section is as follows, showing that the functions $\Gamma_{\mathcal{C},e}$ satisfy a non-degeneracy condition.

\begin{lem}\label{lem:the-Gamma-e-span}
Let $\mathcal{C}$ be a complete core with parameters $(g, a_1,\dots,a_g, t_1,\dots,t_{g-1})$. Consider the functions $\Gamma_{\mathcal{C},e}$, where $e$ ranges over all edges between the uncoloured vertex of $\mathcal{C}$ and a vertex of colour $g$. Then these functions span the real vector space of all functions $[t_1]\times\dots\times[t_{g-1}]\times [1]\to \mathbb{R}$.
\end{lem}

\subsection{Proof of \texorpdfstring{\cref{lem:the-Gamma-e-span}}{Lemma~\ref{lem:the-Gamma-e-span}}}

For this subsection, fix a core $\mathcal{C}$ with parameters $(g, a_1,\dots,a_g, t_1,\dots,t_{g-1})$. Let $L$ be the span of the functions $\Gamma_{\mathcal{C},e}$, for all edges $e$ between the uncoloured vertex of $\mathcal{C}$ and a vertex of colour $g$. Note that $L$ is a subspace of the real vector space of all functions $[t_1]\times\dots\times[t_{g-1}]\times [1]\to \mathbb{R}$. Our goal is to show that $L$ is actually the entire space of  functions $[t_1]\times\dots\times[t_{g-1}]\times [1]\to \mathbb{R}$. 

First, let us make some more definitions.

\begin{defn}
For an integer $1\leq b\leq g$, a \emph{downward tree} of size $b$ in $\mathcal{C}$ is a collection $F\su E(\mathcal{C})$ of $b$ edges of colours $g-b+1,\dots,g$ which form a tree containing the uncoloured vertex of $\mathcal{C}$ as a leaf.
\end{defn}

\begin{figure}
\begin{center}
\begin{tikzpicture}[scale=0.75]
\clip(-4.2,-0.2) rectangle (6.6,6.2);
\draw [line width=1.5pt] (0,6)-- (0,5);
\draw [line width=1.5pt] (0,5)-- (-2,4);
\draw [line width=1.5pt] (0,5)-- (0.5,3);
\draw [line width=1.5pt] (0.5,3)-- (1.5,2);
\draw [line width=1.5pt] (0.5,3)-- (-0.5,1);
\draw [fill=black] (0,0) circle (2pt);
\draw [fill=black] (-0.5,1) circle (2pt);
\draw [fill=black] (0.5,1) circle (2pt);
\draw [fill=black] (-0.5,2) circle (2pt);
\draw [fill=black] (0.5,2) circle (2pt);
\draw [fill=black] (1.5,2) circle (2pt);
\draw [fill=black] (-1.5,2) circle (2pt);
\draw [fill=black] (0.5,3) circle (2pt);
\draw [fill=black] (1.5,3) circle (2pt);
\draw [fill=black] (-0.5,3) circle (2pt);
\draw [fill=black] (-1.5,3) circle (2pt);
\draw [fill=black] (-2.5,3) circle (2pt);
\draw [fill=black] (2.5,3) circle (2pt);
\draw [fill=black] (0,4) circle (2pt);
\draw [fill=black] (1,4) circle (2pt);
\draw [fill=black] (2,4) circle (2pt);
\draw [fill=black] (3,4) circle (2pt);
\draw [fill=black] (-1,4) circle (2pt);
\draw [fill=black] (-2,4) circle (2pt);
\draw [fill=black] (-3,4) circle (2pt);
\draw [fill=black] (0,5) circle (2pt);
\draw [fill=black] (1,5) circle (2pt);
\draw [fill=black] (2,5) circle (2pt);
\draw [fill=black] (3,5) circle (2pt);
\draw [fill=black] (4,5) circle (2pt);
\draw [fill=black] (-1,5) circle (2pt);
\draw [fill=black] (-2,5) circle (2pt);
\draw [fill=black] (-3,5) circle (2pt);
\draw [fill=black] (-4,5) circle (2pt);
\draw [fill=black] (0,6) circle (2pt);
\node at (5.5,6)  {uncoloured};
\node at (5.5,5)  {colour 6};
\node at (5.5,4)  {colour 5};
\node at (5.5,3)  {colour 4};
\node at (5.5,2)  {colour 3};
\node at (5.5,1)  {colour 2};
\node at (5.5,0)  {colour 1};
\node at (0.25,5.5)  {6};
\node at (-1.3,4.65)  {5};
\node at (0.5,4.05)  {4};
\node at (1.18,2.68)  {3};
\node at (-0.1,2.4)  {2};
\end{tikzpicture}
\caption{An example of a downward tree, where $g=6$ and $b=5$. The numbers on the edges indicate the respective colours.}
\label{fig:pict-tree}
\end{center}
\end{figure}
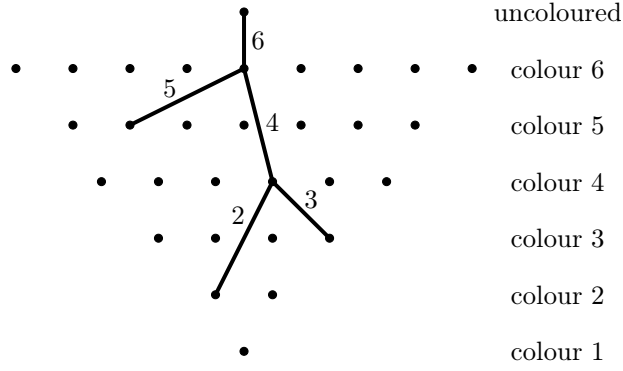

Although formally a downward tree $F$ is just a collection of edges, we say that $F$ contains a vertex of $\mathcal{C}$ if this vertex is part of the tree formed by the edges in $F$. For an example of a downward tree, see \cref{fig:pict-tree}.

\begin{lem}\label{lem:structure-downward-tree}
For any integer $1\leq b\leq g$, every downward tree $F$ of size $b$ in $\mathcal{C}$ contains the uncoloured vertex and exactly one vertex in each of the colours $g-b+1,\dots,g$ (and no vertices in the colours $1,\dots,g-b$). Furthermore, the vertex of colour $g-b+1$ is a leaf in the tree $F$. Finally, if $b\geq 2$ and $e$ is the unique edge of $F$ incident to the vertex of colour $g-b+1$, then $e$ is not incident to the uncoloured vertex and $F\setminus \lbrace e\rbrace$ is a downward tree of size $b-1$.
\end{lem}

\begin{proof}
Since $F$ has an edge in each of the colours $g-b+1,\dots,g$, it must also have a vertex in each of these colours (recall that in a colour system an edge can only have colour $i$ if it is incident to a vertex of colour $i$). Thus, $F$ contains at least one vertex in each of the colours $g-b+1,\dots,g$ and by definition it also contains the uncoloured vertex. As $F$ has only $b+1$ vertices, this establishes the first part of the lemma.

For the second part we need to show that there is only one edge of $F$ incident to the vertex of colour $g-b+1$. However, note that each edge of $F$ incident to the vertex of colour $g-b+1$ has colour $g-b+1$ (because the other vertex of each such edge is either uncoloured or has a colour with a number larger than $g-b+1$). As $F$ has only one edge of colour $g-b+1$, there is indeed only one edge of $F$ incident to the vertex of colour $g-b+1$.

Finally, for the third part, note that the edge $e$ has colour $g-b+1$ (by the argument above), so $F\setminus \lbrace e\rbrace$ is a tree with $b-1$ edges of colours $g-b+2,\dots,g$. The edge $e$ is not incident to the uncoloured vertex, as otherwise both vertices of $e$ would be leaves in $F$ which would contradict $b\geq 2$. In particular, the uncoloured vertex is still a leaf of the tree $F\setminus \lbrace e\rbrace$. Hence $F\setminus \lbrace e\rbrace$ is indeed a downward tree of size $b-1$.
\end{proof}

\begin{lem}\label{lem:downward-tree-in-L}
For every downward tree $F\su E(\mathcal{C})$, we have $\Gamma_{\mathcal{C}, F}\in L$.
\end{lem}
\begin{proof}
We prove the lemma by induction on $b$.

If $b=1$, then $F=\lbrace e\rbrace$ for a single edge $e$. By definition, the uncoloured vertex is a leaf of $F$, so $e$ is incident to the uncoloured vertex. Furthermore, $e$ has colour $g$, so it is also incident to a vertex of colour $g$. We therefore trivially have $\Gamma_{\mathcal{C}, e}\in L$, by the definition of $L$.

Now, let us assume that $b\geq 2$ and that the lemma statement holds for $b-1$. Let $F$ be a downward tree of size $b$. Recall that by \cref{lem:structure-downward-tree}, $F$ contains precisely one vertex $v^*$ of colour $g-b+1$ and this vertex is a leaf in $F$. So let $e^*$ be the unique edge in $F$ incident to $v^*$. Then by the last part of \cref{lem:structure-downward-tree}, $F\setminus \lbrace e^*\rbrace$ is a downward tree of size $b-1$. Finally, let $w$ be the other vertex of $e^*$, so that (again by \cref{lem:structure-downward-tree}) $w$ is coloured with one of the colours $g-b+2, \dots, g$.

Applying \cref{lem:structure-downward-tree} to the downward tree $F\setminus \lbrace e^*\rbrace$, we can label the vertices of $F\setminus \lbrace e^*\rbrace$ as $v_{g-b+2},\dots,v_g, u$, such that $u$ is the uncoloured vertex of $\mathcal{C}$ and such that each vertex $v_i$ has colour $i$. Note that $w$ is one of the vertices $v_{g-b+2},\dots,v_g$.

Now, since $\mathcal C$ is complete (recall \cref{defn:complete}), we can recursively define vertices $v'_{g-b+2},\dots,v'_g$ in $\mathcal{C}$, with colours $g-b+2, \dots, g$ respectively, such that the following three conditions are satisfied.
\begin{itemize}
\item For all $g-b+2\leq j<i\leq g$, the vertices $v'_{i}$ and $v'_j$ are connected by edges of exactly the same shades of colour $j$ as the vertices $v_{i}$ and $v_j$.
\item For all $g-b+2\leq i\leq g$ and all vertices $v$ of $\mathcal{C}$ of any of the colours $1,\dots,g-b+1$ such that $(v_i,v)\neq (w,v^*)$, the vertices $v'_{i}$ and $v$ are connected by edges of exactly the same shades of the colour of $v$ as the vertices $v_{i}$ and $v$.
\item For the index $g-b+2\leq i\leq g$ such that $v_i=w$, the vertices $v'_{i}$ and $v^*$ are connected by edges of exactly the same shades of the of colour $g-b+1$ as the vertices $v_{i}$ and $v^*$ except that there is no edge between $v'_{i}$ and $v^*$ with the shade of $e^*$.
\end{itemize}
Informally speaking, these conditions are saying that the shade of the edges between any two of the vertices $v'_{g-b+2},\dots,v'_g$ and the shades of the edges between the vertices $v'_{g-b+2},\dots,v'_g$ and the vertices of colours $1,\dots,g-b+1$ are the same as the corresponding shades for $v_{g-b+2},\dots,v_g$ except that between the vertices $w'$ and $v^*$ the shade of the edge $e^*$ between $w$ and $v^*$ is missing. Here, $w'$ denotes the vertex $v'_i$ for the index $g-b+2\leq i\leq g$ such that $v_i=w$.

Now, for each edge $e\in F\setminus \lbrace e^*\rbrace$ with endpoints $v_i$ and $v_j$, for $g-b+2\leq j<i\leq g$, there exists an edge $e'\in E(\mathcal{C})$ between $v'_i$ and $v'_j$ such that $e'$ has the same shade of colour $j$ as $e$. Let $F'$ be the collection of all these edges together with the unique edge between $v'_g$ and $u$ (recall that there is only one shade of colour $g$ and that the edge between $v_g$ and $u$ is the only edge in $F\setminus \lbrace e^*\rbrace$ incident to $u$). Then $F'$ forms a tree which is isomorphic to $F\setminus \lbrace e^*\rbrace$ and the corresponding edges are coloured the same way. As $F\setminus \lbrace e^*\rbrace$ is a downward tree of size $b-1$, this implies that $F'$ is also a downward tree of size $b-1$. Furthermore, between any two vertices of $F'$ there exist edges in exactly the same shades as between the corresponding vertices of $F\setminus \lbrace e^*\rbrace$. Finally, every vertex in one of the colours $1,\dots, g-b+1$ has edges in the same shades to the vertices of $F\setminus \lbrace e^*\rbrace$ as to the corresponding vertices of $F'$ except that the shade of the edge $e^*$ is missing between the vertices $v^*$ and $w'$ (recall that $e^*\in F$ is an edge between $v^*$ and $w$).

Thus, for every partial copy of $H$ in $\mathcal{C}$ containing $F'$ we can form a corresponding partial copy of $H$ in $\mathcal{C}$ containing $F\setminus\lbrace e^*\rbrace$ but not containing $e^*$, by simply replacing each of the vertices $v'_{g-b+2},\dots,v'_g$ in the image of the partial copy by $v_{g-b+2},\dots,v_g$ (and by replacing each edge in the image of the partial copy by an edge of the same shade between the corresponding vertices). This process is bijective and does not change which shades of which colours occur among the edges in the image of the partial copy. It also does not change the weight of the partial copy. Thus, for every $(j_1,\dots,j_g)\in [t_1]\times\dots\times[t_{g-1}]\times [1]$, the quantity $\Gamma_{\mathcal{C}, F'}(j_1,\dots,j_g)$ is equal to the sum of the weights of the $(j_1,\dots,j_g)$-coloured partial copies of $H$ in $\mathcal{C}$ that contain $F\setminus\lbrace e^*\rbrace$ but not $e^*$. In other words, we have $\Gamma_{\mathcal{C}, F'}(j_1,\dots,j_g)=\Gamma_{\mathcal{C}, F\setminus \lbrace e^*\rbrace}(j_1,\dots,j_g)-\Gamma_{\mathcal{C}, F}(j_1,\dots,j_g)$ for all $(j_1,\dots,j_g)\in [t_1]\times\dots\times[t_{g-1}]\times [1]$.

Since $F'$ and $F\setminus \lbrace e^*\rbrace$ are downward trees of size $b-1$, we have $\Gamma_{\mathcal{C}, F'}, \Gamma_{\mathcal{C}, F\setminus \lbrace e^*\rbrace}\in L$ by the induction assumption. Hence $\Gamma_{\mathcal{C}, F}=\Gamma_{\mathcal C,F\setminus \lbrace e^*\rbrace}-\Gamma_{\mathcal{C}, F'}\in L$, as desired.
\end{proof}

Now, note that each downward tree $F\su E(\mathcal{C})$ of size $g$ contains exactly one edge in each of the colours $1,\dots,g$. For each $1\leq i\leq g$, let $j_i\in [t_i]$ be the shade of the edge of colour $i$ in $F$. Then for any tuple $(j'_1,\dots,j'_g)\in [t_1]\times \dots\times [t_{g-1}]\times [1]$ with $(j'_1,\dots,j'_g)\neq (j_1,\dots,j_g)$ we have $\Gamma_{\mathcal{C}, F}(j'_1,\dots,j'_g)=0$, since by definition there cannot be any $(j'_1,\dots,j'_g)$-coloured partial copies of $H$ containing $F$. That is to say, $\Gamma_{\mathcal{C}, F}\in L$ is a (possibly zero) multiple of the indicator function of $(j_1,\dots,j_g)$. Since the set of indicator functions of each $(j_1,\dots,j_g)$ span the space of all functions $[t_1]\times \dots\times [t_{g-1}]\times [1]\to \RR$, it suffices to prove the following lemma in order to finish the proof of \cref{lem:the-Gamma-e-span}.

\begin{lem}\label{lem:Gamma-pos-value}
For every $(j_1,\dots,j_g)\in [t_1]\times \dots\times [t_{g-1}]\times [1]$ there exists a downward tree $F$ of size $g$ in $\mathcal{C}$ such that $\Gamma_{\mathcal{C}, F}(j_1,\dots,j_g)>0$.
\end{lem}
\begin{proof}
Fix $(j_1,\dots,j_g)\in [t_1]\times \dots\times [t_{g-1}]\times [1]$. Recall that $\Gamma_{\mathcal{C}, F}(j_1,\dots,j_g)$ is the sum of the weights of all $(j_1,\dots,j_g)$-coloured partial copies of $H$ in $\mathcal{C}$ whose image contains all edges in $F$. Since each partial copy of $H$ in $\mathcal{C}$ has positive weight, 
it suffices to show that there exists a $(j_1,\dots,j_g)$-coloured  partial copy of $H$ in $\mathcal{C}$ which contains some downward tree $F$ of size $g$.

Since $\mathcal{C}$ is complete, we can recursively choose vertices $v_1,\dots,v_g$ in $\mathcal{C}$ such that for each $1\leq i\leq n$, the vertex $v_i$ has colour $i$ and for any $1\leq j<i\leq g$ there are edges of all $t_j$ shades of colour $j$ between $v_j$ and $v_i$. Also, since $H$ is connected, and has $h\ge g+1$ vertices, we can choose a $g$-edge subtree $F_H$ of $H$. Let $V'\su V(H)$ be the set of vertices of this subtree $F_H$ and note that $|V'|=g+1$. Choose one leaf of $F_H$ and call it $u$. Now, define $\phi:V'\to V(\mathcal{C})$ by mapping $u$ to the uncoloured vertex of $\mathcal{C}$ and mapping the remaining vertices of $F_H$ to $v_1,\dots,v_g$ in order of decreasing distance to $u$ in the tree $F_H$ (this means the vertex of $F_H$ with maximum distance from $u$ in the tree $F_H$ will be mapped to $v_1$, the vertex with second-largest distance to $v_2$, and so on), where we break ties arbitrarily.

Note that the image of $\phi:V'\to V(\mathcal{C})$ contains one vertex of each colour (the vertices $v_1,\dots,v_g$) and the uncoloured vertex. By the choice of $v_1,\dots,v_g$, we can extend $\phi$ to a $(j_1,\dots,j_g)$-coloured partial copy of $H$ in $\mathcal C$ (we have already defined the way $\phi$ maps the relevant vertices of $H$, we just need to define the way it maps edges). Let $F=\phi(F_H)$ be the image of our subtree $F_H$ of $H$. We can check that $F$ is a downward tree of size $g$.
\end{proof}

\section{Proofs of \texorpdfstring{\cref{lem:medium-scale,lem:rough-scale}}{Lemmas~\ref{lem:medium-scale} and~\ref{lem:rough-scale}}}
\label{sec:halasz-prep}

In this section we finally prove \cref{lem:medium-scale,lem:rough-scale}, using our new anticoncentration inequality in \cref{sec:halasz} and the functions $\Gamma_{\mathcal C,e}$ defined in \cref{sec:cores}. In \cref{subsec:final-prep} we make some definitions and state some auxiliary lemmas. Most importantly, we explain how to define cores $\mathcal C$ in such a way that the functions $\Gamma_{\mathcal{C},e}$ represent the typical effects of changing the status of certain edges. In \cref{subsec:nu,subsec:kappa} we prove the auxiliary lemmas, and we put everything together in \cref{sec:final-pf}.

\subsection{Preparations}
\label{subsec:final-prep}

First, we define functions $\kappa_H(\mathcal G,G_0,\cdot,u,v)$, measuring the change to $\psi_H(\mathcal G,G_0,\cdot)$ that results from changing an edge $uv$ in $G_0$. In our proof of \cref{lem:rough-scale}, taking $u,v\in \operatorname U(\mathcal G)$, these functions will correspond to the $\Delta_i\boldsymbol f$ that appear in the statement of \cref{thm:rough-halasz}. Recall the definition of the graph $\mathcal{G}(G_0, j_1,\dots,j_g)$ from \cref{defn:psi}.

\begin{defn}\label{defn:kappa}
Let $\mathcal{G}$ be a colour system with parameters $(g, a_1,\dots,a_g, t_1,\dots,t_g)$. Then, given two distinct vertices $u$ and $v$ in $\mathcal{G}$, a $g$-tuple $(j_1,\dots,j_g)\in [t_1]\times \dots\times [t_g]$, and a graph $G_0$ on the vertex set $\operatorname{U}(\mathcal{G})$, let $\kappa_H(\mathcal{G}, G_0, j_1,\dots,j_g, u, v)$ be the number of labelled copies of $H$ in the graph $\mathcal{G}(G_0, j_1,\dots,j_g)+\lbrace uv\rbrace$ which use the edge $uv$ as well as at least one vertex of each of the $g$ colours. (Here, $\mathcal{G}(G_0, j_1,\dots,j_g)+\lbrace uv\rbrace$ is the graph obtained from $\mathcal{G}(G_0, j_1,\dots,j_g)$ by adding the edge $uv$ if this edge is not already present.)
\end{defn}

Next, we need a similar definition for the proof of \cref{lem:medium-scale}. Recall that \cref{lem:medium-scale} concerns functions $\mu_{\mathcal G,\mathcal S}$, which are obtained by averaging functions of the form $\psi_H(\mathcal G_{\mathcal S},G_0,\cdot)$ over $G_0$. We will be interested in the effects on $\mu_{\mathcal G,\mathcal S}$ of adding or removing vertices from the ``neighbourhood'' sets in $\mathcal S$, which is equivalent to changing the status of edges incident to one of the $a_g$ vertices of colour $g$. We define functions $\nu_{\mathcal G,\mathcal S,u,v}$ (where $u$ is an uncoloured vertex and $v$ is a vertex of colour $g$) to measure the average effects of such changes.

\begin{defn}\label{defn:nu}
Given a restricted colour system $\mathcal{G}$ with parameters $(g, a_1,\dots,a_g, t_1,\dots,t_{g-1})$ where $1\leq g\leq h-1$, an outcome of the random collection of sets $\mathcal{S}$  in \cref{defn:extension-restricted-colour-system}, a vertex $u\in \operatorname{U}(\mathcal{G})$, and a vertex $v$ of colour $g$ in $\mathcal{G}$, let $\nu_{\mathcal{G}, \mathcal{S}, u,v}: [t_1]\times \dots\times [t_{g-1}]\times [1]\to \RR$ be the function given by
\[\nu_{\mathcal{G}, \mathcal{S},u,v}(j_1,\dots,j_g)=\E_{G_0}[\kappa_H(\mathcal{G_S}, G_0, j_1,\dots,j_g, u,v)]\]
for all $(j_1,\dots,j_g)\in [t_1]\times \dots\times [t_{g-1}]\times [1]$. Here, $G_0\sim\GG(\operatorname{U}(\mathcal{G}), p)$ is a random graph on the vertex set $\operatorname{U}(\mathcal{G})$.
\end{defn}

Now, we want to show that the typical values of the $\nu_{\mathcal{G}, \mathcal{S},u,v}$ and $\kappa_H(\mathcal{G}, G_0,\cdot, u,v)$ can be expressed in terms of functions $\Gamma_{\mathcal{C},e}$. First, we consider $\nu_{\mathcal{G}, \mathcal{S},u,v}$, for the proof of \cref{lem:medium-scale}. It will suffice to restrict our attention to the cases where $u$ comes from a very ``rich'' subset of the uncoloured vertices.

\begin{defn}\label{defn:U-star}
For a restricted colour system $\mathcal{G}$ with parameters $(g, a_1,\dots,a_g, t_1,\dots,t_{g-1})$, let $\operatorname{U^*}(\mathcal{G})$ be the set of all uncoloured vertices in $\mathcal{G}$ which are connected to all vertices of the colours $1,\dots,g-1$ in all possible shades of these colours.
\end{defn}

When a general position assumption is satisfied, $\operatorname{U^*}(\mathcal{G})$ has linear size (by \cref{lem:sets-general-position-intersection}), as follows.

\begin{fact}\label{lem:U-star-big}
Fix integers $g,a_1,\dots,a_g, t_1,\dots,t_{g-1}\geq 1$. If $\mathcal{G}$ is an essentially $p$-general restricted colour system of order $n$ which has parameters $(g, a_1,\dots,a_g, t_1,\dots,t_{g-1})$ then $\vert\operatorname{U^*}(\mathcal{G})\vert=\Omega(n)$.
\end{fact}

We remind the reader that (as in the rest of this paper) the asymptotics in \cref{lem:U-star-big} are as $n\to\infty$, while $p$ and the parameters of the restricted colour system are treated as fixed constants for all asymptotic notation. Now, the relevant core for \cref{lem:medium-scale} is as follows.

\begin{defn}\label{def-core-restr-c-s}
Given a restricted colour system $\mathcal{G}$ with parameters $(g, a_1,\dots,a_g, t_1,\dots,t_{g-1})$, we can obtain a core $\mathcal{C}$ with parameters $(g, a_1,\dots,a_g, t_1,\dots,t_{g-1})$ as follows. Consider all coloured vertices of $\mathcal{G}$ together with one additional uncoloured vertex which we connect to all coloured vertices by edges in all possible shades. We call $\mathcal{C}$ the \emph{core of the restricted colour system $\mathcal{G}$}.
\end{defn}

Note that if a restricted colour system $\mathcal{G}$ is complete, then its core is also complete (recall that being complete only depends on the edges between the coloured vertices). The following lemma gives the connection between the functions $\nu_{\mathcal{G}, \mathcal{S},u,v}$ and the functions $\Gamma_{\mathcal{C},e}$. It will be proved in \cref{subsec:nu}.

\begin{lem}\label{lem-nu-close-Gamma}
Fix integers $1\leq g\leq h-1$ and $a_1,\dots,a_g, t_1,\dots,t_{g-1}\geq 1$. Let $\mathcal{G}$ be an essentially weakly $p$-general restricted colour system of order $n$ with parameters $(g, a_1,\dots,a_g, t_1,\dots,t_{g-1})$. Furthermore, let $\mathcal{S}$ be an outcome of the random collection of sets $\mathcal{S}$  in \cref{defn:extension-restricted-colour-system}, let $u\in \operatorname{U^*}(\mathcal{G})$ and let $v$ be a vertex of colour $g$ in $\mathcal{G}$. Finally, let $\mathcal{C}$ be the core of the restricted colour system $\mathcal{G}$, and let $e$ be the unique edge from $v$ to the uncoloured vertex in $\mathcal{C}$ (recall that $v$ is a vertex of colour $g$ in $\mathcal{C}$, and that the colour $g$ only has one shade). Then, if the colour system $\mathcal{G_S}$ is $p$-general, we have
\[\left\Vert\nu_{\mathcal{G}, \mathcal{S},u,v}-n^{h-g-1}\Gamma_{\mathcal{C},e}\right\Vert_\infty=O\left(n^{h-g-(3/2)}\log n\right). \]
\end{lem}

Next, we turn to the functions $\kappa_H(\mathcal G,G_0,\cdot,u,v)$, for the proof of \cref{lem:rough-scale}. For this, we consider cores of a different type.

\begin{defn}\label{def:kernel}
Given a colour system $\mathcal{G}$ with parameters $(g-1, a_1,\dots,a_{g-1}, t_1,\dots,t_{g-1})$, define its \emph{extended core} to be the core with parameters $(g, a_1,\dots,a_{g-1}, 2^{a_1t_1+\dots+a_{g-1}t_{g-1}},t_1,\dots,t_{g-1})$ obtained as follows. Start with all coloured vertices of $\mathcal{G}$. Now, for each possible choice of subsets $I_v\su [t_j]$ for each $1\leq j\leq g-1$ and each vertex $v$ of colour $j$, add a vertex of colour $g$ which is connected to all the vertices $v$ of colours $1,\dots,g-1$ with edges of exactly the shades given by the set $I_v$.  Finally, add one uncoloured vertex and connect it to all coloured vertices by edges in all possible shades (including exactly one shade of colour $g$).
\end{defn}

Note that in \cref{def:kernel}, there are precisely $(2^{t_1})^{a_1}\dotsm(2^{t_{g-1}})^{a_{g-1}}=2^{a_1t_1+\dots+a_{g-1}t_{g-1}}$ different choices for all the subsets $I_v\su [t_j]$. Thus, $2^{a_1t_1+\dots+a_{g-1}t_{g-1}}$ vertices of colour $g$ get added and the resulting core indeed has parameters $(g, a_1,\dots,a_{g-1}, 2^{a_1t_1+\dots+a_{g-1}t_{g-1}},t_1,\dots,t_{g-1})$. Furthermore, note that if the colour system $\mathcal{G}$ is complete, then its extended core is a complete core.

In a similar way to \cref{lem-nu-close-Gamma}, for the proof of \cref{lem:rough-scale} it will suffice to restrict our attention to those $\kappa_H(\mathcal G,G_0,\cdot,u,v)$ where $u$ and $v$ belong to certain special sets of uncoloured vertices.

\begin{defn}\label{def:U-e}
Let $\mathcal{G}$ be a colour system with parameters $(g-1, a_1,\dots,a_{g-1}, t_1,\dots,t_{g-1})$ and let $\mathcal{C}$ be the extended core of $\mathcal{G}$. Let $E^g(\mathcal C)$ be the set of edges $e$ connecting the uncoloured vertex in $\mathcal{C}$ to some vertex $w$ of colour $g$ in $\mathcal{C}$ (such an edge is uniquely determined by $w$ because colour $g$ only has one shade). For each such $e\in E^g(\mathcal C)$, we define the subset $U_e\su \operatorname{U}(\mathcal{G})$ as follows. Let $U_e$ consist of all those uncoloured vertices $v$ in $\mathcal{G}$ such that $v$ is connected to all vertices of $\mathcal{G}$ of colours $1,\dots,g-1$ in precisely the same shades of the corresponding colours in which the vertex $w$ is connected to these vertices in $\mathcal{C}$. Also, let $U_*\su \operatorname{U}(\mathcal{G})$ be the set of all uncoloured vertices $v$ in $\mathcal{G}$ which are connected to all vertices of $\mathcal{G}$ of colours $1,\dots,g-1$ in all possible shades of these colours.
\end{defn}

Note that $U_*$ is a special case of $U_e$, for the edge $e\in E^g(\mathcal C)$ connecting the uncoloured vertex of $\mathcal{C}$ to the unique vertex of colour $g$ in $\mathcal{C}$ which is connected to all vertices of colours $1,\dots,g-1$ in all possible shades of these colours. Also note that the $2^{a_1t_1+\dots+a_{g-1}t_{g-1}}$ sets $U_e$, for $e\in E^g(\mathcal C)$, form a partition of $\operatorname{U}(\mathcal{G})$. We will need a counterpart of \cref{lem:U-star-big} (again a simple consequence of \cref{lem:sets-general-position-intersection}): when a general position assumption is satisfied, each $U_e(\mathcal{G})$ has linear size, as follows.

\begin{fact}\label{lem:U-e-big}
Fix integers $g,a_1,\dots,a_{g-1}, t_1,\dots,t_{g-1}\geq 1$. Let $\mathcal{G}$ be a weakly $p$-general colour system of order $n$ with parameters $(g-1, a_1,\dots,\allowbreak a_{g-1},\allowbreak t_1,\dots,t_{g-1})$, and let $\mathcal C$ be the extended core of $\mathcal G$. Then for every $e\in E^g(\mathcal C)$, we have $|U_e|=\Omega(n)$. In particular, $|U_*|=\Omega(n)$.
\end{fact}

The next lemma will be proved in \cref{subsec:kappa}.

\begin{lem}\label{lem:psi-u-v-close-Gamma}
Fix integers $1\leq g\leq h-1$ and $a_1,\dots,a_{g-1}, t_1,\dots,t_{g-1}\geq 1$. Let $\mathcal{G}$ be a weakly $p$-general colour system of order $n$ with parameters $(g-1, a_1,\dots,a_{g-1}, t_1,\dots,t_{g-1})$, and let $\mathcal{C}$ be the extended core of $\mathcal{G}$. Consider an edge $e\in E^g(\mathcal C)$, and consider the sets $U_e\su  \operatorname{U}(\mathcal{G})$ and $U_*\su  \operatorname{U}(\mathcal{G})$ as defined as in \cref{def:U-e}. Then for any distinct vertices $u\in U_*$ and $v\in U_e$ the following holds. If we choose a random graph $G_0\sim\GG(\operatorname{U}(\mathcal{G}), p)$ on the vertex set $\operatorname{U}(\mathcal{G})$, then with probability $1-n^{-\omega(1)}$ we have
\[\left\|\kappa_H(\mathcal{G}, G_0, \cdot, u,v)-n^{h-g-1}\Gamma_{\mathcal{C},e}(\cdot,1)\right\|_\infty= O\left(n^{h-g-(3/2)}\log n\right). \]
\end{lem}

\subsection{Proof of \texorpdfstring{\cref{lem-nu-close-Gamma}}{Lemma~\ref{lem-nu-close-Gamma}}}
\label{subsec:nu}
In this subsection we prove \cref{lem-nu-close-Gamma}. Throughout the subsection, fix integers $1\leq g\leq h-1$ and $a_1,\dots,a_{g}, t_1,\dots,t_{g-1}\geq 1$ (in order to prove \cref{lem-nu-close-Gamma} for these values). For all asymptotic notation, these fixed values are treated as constants.

As in the statement of \cref{lem-nu-close-Gamma}, let $\mathcal{G}$ be an essentially weakly $p$-general restricted colour system of order $n$ with parameters $(g, a_1,\dots,a_g, t_1,\dots,t_{g-1})$. Let $\mathcal{S}$ be an outcome of the random collection of sets $\mathcal{S}$  in \cref{defn:extension-restricted-colour-system} such that $\mathcal{G_S}$ is $p$-general, let $u\in \operatorname{U^*}(\mathcal{G})$ and let $v$ be a vertex of colour $g$ in $\mathcal{G}$. Finally, let $\mathcal{C}$ be the core of the restricted colour system $\mathcal{G}$, and let $e\in E^g(\mathcal C)$ be the unique edge from $v$ to the uncoloured vertex in $\mathcal{C}$. We need to prove that
\[
\left\Vert\nu_{\mathcal{G}, \mathcal{S},u,v}-n^{h-g-1}\Gamma_{\mathcal{C},e}\right\Vert_\infty=O\left(\log n\cdot n^{h-g-(3/2)}\right).
\]
Now, let us define slightly modified versions of $\kappa_H$ and the $\nu_{\mathcal{G}, \mathcal{S},u,v}$, that are easier to work with. For any outcome of the random graph $G_0\sim\GG(\operatorname{U}(\mathcal{G}), p)$ and any $(j_1,\dots,j_g)\in [t_1]\times \dots\times [t_{g-1}]\times [1]$, let $\kappa_H'(\mathcal{G}, G_0, j_1,\dots,j_g, u, v)$ be the number of labelled copies of $H$ in the graph $\mathcal{G}(G_0, j_1,\dots,j_g)+\lbrace uv\rbrace$ which use the edge $uv$ and which use \emph{exactly} one vertex of each of the $g$ colours (for $\kappa_H$ we considered the number of copies of $H$ which use \emph{at least} one vertex of each colour). Then, we define $\nu_{\mathcal{G}, \mathcal{S}, u,v}': [t_1]\times \dots\times [t_{g-1}]\times [1]\to \RR$ analoguously to $\nu_{\mathcal{G}, \mathcal{S}, u,v}$:
\[\nu_{\mathcal{G}, \mathcal{S},u,v}'(j_1,\dots,j_g)=\E_{G_0}[\kappa_H'(\mathcal{G_S}, G_0, j_1,\dots,j_g, u,v)]\]
for all $(j_1,\dots,j_g)\in [t_1]\times \dots\times [t_{g-1}]\times [1]$.

Note that for any outcome of the random graph $G_0\sim\GG(\operatorname{U}(\mathcal{G}), p)$ and any $(j_1,\dots,j_g)\in [t_1]\times \dots\times[t_{g-1}]\times [1]$ the difference $\kappa_H(\mathcal{G}, G_0, j_1,\dots,j_g, u, v)-\kappa_H'(\mathcal{G}, G_0, j_1,\dots,j_g, u, v)$ is precisely the number of labelled copies of $H$ in the graph $\mathcal{G_S}(G_0, j_1,\dots,j_g)+\lbrace uv\rbrace$ which use the edge $uv$ as well as at least one vertex of each of the $g$ colours and which use at least two vertices of the same colour. Each such labelled copy has to use at least $g+1$ of the coloured vertices and also the vertex $u\in \operatorname{U^*}(\mathcal{G})\su \operatorname{U}(\mathcal{G})$. Hence it can use at most $h-g-2$ vertices in $\operatorname{U}(\mathcal{G})\sm \lbrace u\rbrace$. Thus, the number of such labelled copies is always at most $h^{g+2}\cdot (a_1+\dots+a_g+1)^{g+2}\cdot n^{h-g-2}=O(n^{h-g-2})$. It follows that
\begin{equation*}
\left\Vert \nu_{\mathcal{G}, \mathcal{S},u,v}-\nu_{\mathcal{G}, \mathcal{S},u,v}'\right\Vert_\infty=O(n^{h-g-2}).
\end{equation*}
So, to prove \cref{lem-nu-close-Gamma} it suffices to prove the following lemma.

\begin{lem}\label{lem:nu-prime-close-Gamma} We have
\[\left\Vert \nu_{\mathcal{G}, \mathcal{S},u,v}'-n^{h-g-1}\Gamma_{\mathcal{C}, e}\right\Vert_\infty= O\left(\log n\cdot n^{h-g-(3/2)}\right).\]
\end{lem}

\begin{proof}
We need to show that for every $(j_1,\dots,j_g)\in [t_1]\times \dots\times[t_{g-1}]\times [1]$ we have
\begin{equation}
\left\vert \nu_{\mathcal{G}, \mathcal{S},u,v}'(j_1,\dots,j_g)-n^{h-g-1}\Gamma_{\mathcal{C}, e}(j_1,\dots,j_g)\right\vert= O\left(\log n\cdot n^{h-g-(3/2)}\right).\label{eq:nu-Gamma}
\end{equation}
So let us fix some $(j_1,\dots,j_g)\in [t_1]\times \dots\times[t_{g-1}]\times [1]$. For every $i=1,\dots,g$ let us refer to shade $j_i$ of colour $i$ as the \emph{desired} shade of colour $i$.

Recall that $\nu_{\mathcal{G}, \mathcal{S},u,v}'(j_1,\dots,j_g)=\E_{G_0}[\kappa_H'(\mathcal{G_S}, G_0, j_1,\dots,j_g, u,v)]$ is the expected number of labelled copies of $H$ in the graph $\mathcal{G_S}(G_0, j_1,\dots,j_g)+\lbrace uv\rbrace$ which use exactly one vertex of each of the $g$ colours and use the edge $uv$. We organise these copies by how they interact with the coloured vertices, as follows.

Consider any graph homomorphism $\phi: H[V_\phi]\to \mathcal{G_S}(G_0, j_1,\dots,j_g)+\lbrace uv\rbrace$, such that $V_\phi$ is a subset of $g+1$ vertices of $H$ and such that the image of $\phi$ contains exactly one vertex of each of the $g$ colours and the edge $uv$ (let $\Phi$ be the set of all such homomorphisms). Let $E_\phi$ be the expected number of labelled copies of $H$ in the graph $\mathcal{G_S}(G_0, j_1,\dots,j_g)+\lbrace uv\rbrace$ that extend $\phi$ by mapping the vertices in $V\sm V_\phi$ into $\operatorname{U}(\mathcal{G})\sm \lbrace u\rbrace$ (where the expectation is taken over the random choice of $G_0\sim\GG(\operatorname{U}(\mathcal{G}), p)$). Then we have
\begin{equation}\label{eq:sum-E-phi}
\nu_{\mathcal{G}, \mathcal{S},u,v}'(j_1,\dots,j_g)=\sum_{\phi\in \Phi} E_\phi.
\end{equation}
Now, since $\mathcal {G_S}$ is $p$-general, we can estimate the $E_\phi$ as follows.

\begin{claim}\label{claim:E-phi}
For each $\phi\in \Phi$ as above, we have
\[\left| E_\phi-p^{e(H)-e_H(V_\phi)}n^{h-g-1}\right|\leq O\left(n^{h-g-(3/2)}\log n\right).\]
\end{claim}

\begin{proof}
Let $V_{\text{col}}$ be the set of those vertices $x\in V_\phi$ such that $\phi(x)$ is a coloured vertex (i.e. $\phi(x)\ne u$). So, $\vert V_{\text{col}}\vert=g$. Recall that $E_\phi$ is the expected number of labelled copies of $H$ in the graph $\mathcal{G_S}(G_0, j_1,\dots,j_g)+\lbrace uv\rbrace$ that extend $\phi$ by mapping the vertices in $V\sm V_\phi$ into $\operatorname{U}(\mathcal{G})\sm \lbrace u\rbrace$. For every vertex $y\in V(H)\sm V_\phi$, let $M_\phi(y)$ be the set of possible choices for the image of $y$ that are compatible with the map $\phi$ on $H[V_\phi]$. More precisely, $M_\phi(y)$ is the set of vertices $w\in\operatorname{U}(\mathcal{G})\sm \lbrace u\rbrace$ such that for every neighbour $x\in V_{\text{col}}$ of $y$, the vertex $w$ is connected to $\phi(x)$ in the desired shade of the colour of the vertex $\phi(x)$. Let $N$ be the number of $(h-g-1)$-tuples in $\prod_{y\in V(H)\sm V_\phi}M_\phi(y)$ whose vertices are distinct (that is, the number of ways to choose a distinct vertex from each $M_\phi(y)$). Then $N=\prod_{y\in V(H)\sm V_\phi}|M_\phi(y)|+O(n^{h-g-2})$, and
\begin{equation}\label{eq:expectation-E-phi}
E_\phi=p^{e_H(V(H)\sm V_{\text{col}})}\cdot N=p^{e_H(V(H)\sm V_{\text{col}})}\cdot \prod_{y\in V(H)\sm V_\phi}|M_\phi(y)|+O(n^{h-g-2}).
\end{equation}
Indeed, if we choose possible images for all the vertices $y\in V(H)\sm V_\phi$ (there are $N$ such choices), then each of the $e_H(V(H)\sm V_{\text{col}})$ edges of $H$ inside $V(H)\sm V_{\text{col}}$ needs to be mapped to an edge of $G_0\sim\GG(\vert \operatorname{U}(\mathcal{G})\vert,p)$ and the probability for this to happen is $p^{e_H(V(H)\sm V_{\text{col}})}$.

Now, the sizes of the $M_\phi(y)$ are dictated by our assumption that $\mathcal {G_S}$ is $p$-general. Fix a vertex $y\in V(H)\sm V_\phi$ and let $x_1,\dots,x_k$ be its neighbours in $V_{\text{col}}$. Let $N_1,\dots, N_k\su \operatorname{U}(\mathcal{G})$ be the neighbourhoods of the vertices $\phi(x_1),\dots,\phi(x_k)$ in $\operatorname{U}(\mathcal{G})$ in the desired shades of the colours of $\phi(x_1),\dots,\phi(x_k)$, respectively. Then the set of possible choices for the image of $y$ is $N_1\cap\dots\cap N_k\sm \lbrace u\rbrace$. So, $|M_\phi(y)|$ differs from $\left|N_1\cap\dots\cap N_k\right|$ by at most 1. Now, if we consider all the neighbourhoods  in $\operatorname{U}(\mathcal{G})$ of all vertices of colours $1,\dots,g$ in $\mathcal{G_S}$ in all the respective shades of these colours, then these are $a_1t_1+\dots+a_{g-1}t_{g-1}+a_g$ subsets of the set $\operatorname{U}(\mathcal{G})$ in $(p,3^g)$-general position (this is because $\mathcal{G_S}$ is by assumption a $p$-general colour system). Thus, by \cref{lem:sets-general-position-intersection} we have
\[\left|\left|N_1\cap\dots\cap N_k\right|-p^k\vert \operatorname{U}(\mathcal{G})\vert\right|= O\left(\vert \operatorname{U}(\mathcal{G})\vert^{1/2}\log \vert \operatorname{U}(\mathcal{G})\vert\right).\]
Recall that $\vert \operatorname{U}(\mathcal{G})\vert=n-a_1-\dots-a_g=n-O(1)$, and that $k$ was the number of neighbours in of $y$ in $V_{\text{col}}$, so
\[\left||M_\phi(y)|-p^{e_H(y, V_{\text{col}})}n\right|=\left||M_\phi(y)|-p^{k}n\right|\leq \left|\left|N_1\cap\dots\cap N_k\right|-p^kn\right|+1=O\left(n^{1/2}\log n\right),\]
or equivalently
\[|M_\phi(y)|=\left(1+O\left(\log n/\sqrt n\right)\right)p^{e_H(y, V_{\text{col}})}n.\]
Finally, observe that $e_H(V(H)\sm V_{\text{col}})$, plus the sum of all the $e_H(y, V_{\text{col}})$, for $y\in V(H)\sm V_\phi$, is equal to $e(H)-e_H(V_\phi)$. Indeed, since $V_\phi$ and $V(H)\sm V_{\text{col}}$ intersect in only one vertex $z$, every edge of $H$ is either between two vertices of $V_\phi$, two vertices of $V(H)\sm V_{\text{col}}$, or between a vertex of $V_\phi\sm \{z\}=V_{\text{col}}$ and a vertex of $(V(H)\sm V_{\text{col}})\sm \{z\}=V(H)\sm V_{\phi}$. From \cref{eq:expectation-E-phi} we therefore conclude that
\begin{multline*}
E_{\phi}=p^{e_H(V(H)\sm V_{\text{col}})}\cdot \prod_{y\in V(H)\sm V_\phi}\left(\left(1+O\left(\log n/\sqrt n\right)\right)p^{e_H(y, V_{\text{col}})}n\right)+O(n^{h-g-2})\\
=\left(1+O\left(\log n/\sqrt n\right)\right)p^{e(H)-e_H(V_\phi)}n^{h-(g+1)}=p^{e(H)-e_H(V_\phi)}n^{h-g-1}+O(n^{h-g-(3/2)}\log n),
\end{multline*}
which is equivalent to the desired bound.
\end{proof}

Now, the sum in \cref{eq:sum-E-phi} is only over $|\Phi|\le h^{g+1}a_1\dots a_g=O(1)$ choices of $\phi$, so \cref{claim:E-phi} implies that
\begin{equation*}
\left\vert\nu_{\mathcal{G}, \mathcal{S},u,v}'(j_1,\dots,j_g)-\sum_{\phi\in \Phi} \left(p^{e(H)-e_H(V_\phi)}n^{h-g-1}\right)\right\vert=O\left( n^{h-g-(3/2)}\log n\right).
\end{equation*}
Finally, recall that $\Phi$ is the set of homomorphisms of the form $\phi: H[V_\phi]\to \mathcal{G_S}(G_0, j_1,\dots,j_g)+\lbrace uv\rbrace$, with $|V_\phi|=g+1$, such that the image of $\phi$ contains exactly one vertex of each of the $g$ colours and the edge $uv$. The core $\mathcal{C}$ of the restricted colour system $\mathcal{G}$ was defined (in \cref{def-core-restr-c-s}) in such a way that there is a bijective correspondence between $\Phi$ and the set of $(j_1,\dots, j_g)$-coloured partial copies $\phi^*$ of $H$ in $\mathcal{C}$ which contain the edge $e$. Recall that $\Gamma_{\mathcal C,e}(j_1,\dots,j_g)$ was defined (in \cref{defn:gamma}) as the sum of the weights of all $(j_1,\dots, j_g)$-coloured partial copies whose image contains $e$, and the weight of a partial copy $\phi^*:H[V']\to\mathcal C$ was defined to be $p^{e(H)-e_H(V')}$. So,
$$\Gamma_{\mathcal C,e}(j_1,\dots,j_g)=\sum_{\phi\in \Phi} p^{e(H)-e_H(V_\phi)}.$$
The desired bound \cref{eq:nu-Gamma} follows.
\end{proof}

\subsection{Proof of \texorpdfstring{\cref{lem:psi-u-v-close-Gamma}}{Lemma~\ref{lem:psi-u-v-close-Gamma}}}
\label{subsec:kappa}
In this subsection we deduce \cref{lem:psi-u-v-close-Gamma} from \cref{lem-nu-close-Gamma}. Throughout the subsection, fix integers $1\leq g\leq h-1$ and $a_1,\dots,a_{g-1}, t_1,\dots,t_{g-1}\geq 1$ (in order to prove \cref{lem:psi-u-v-close-Gamma} for these values). For all asymptotic notation, these fixed values are treated as constants.

Let $\mathcal{G}$ be a weakly $p$-general colour system of order $n$ with parameters $(g-1, a_1,\dots,a_{g-1}, t_1,\dots,t_{g-1})$, let $\mathcal{C}$ be the extended core of $\mathcal{G}$, and consider some $e\in E^g(\mathcal C)$.
Fix distinct vertices $u\in U_*$ and $v\in U_e$. We need to show that for a random graph $G_0\sim\GG(\operatorname{U}(\mathcal{G}), p)$, with probability $1-n^{-\omega(1)}$ we have
\begin{equation}\label{eqn:kappa-close-Gamma-to-prove}
\left\vert\kappa_H(\mathcal{G}, G_0, j_1,\dots,j_{g-1}, u,v)-n^{h-g-1}\Gamma_{\mathcal{C},e}(j_1,\dots,j_{g-1},1)\right\vert=O\left(\log n\cdot n^{h-g-(3/2)}\right).
\end{equation}
for all $(j_1,\dots,j_{g-1})\in [t_1]\times \dots\times [t_{g-1}]$. For the rest of the proof fix some $(j_1,\dots,j_{g-1})\in [t_1]\times \dots\times [t_{g-1}]$: we will show that \cref{eqn:kappa-close-Gamma-to-prove} holds with probability $1-n^{-\omega(1)}$ (then the desired result will follow, taking a union bound over all choices of $(j_1,\dots,j_{g-1})$).

By \cref{lem:U-e-big}, each of the $2^{a_1t_1+\dots+a_{g-1}t_{g-1}}$ disjoint sets $U_f$ has size $\Omega(n)\ge 2$. Let $Z$ be a set containing one representative from each $U_f$, taking $v\in U_e$, but taking some vertex other than $u$ in $U_*$. If we imagine that the vertices of $Z$ are coloured with colour $g$, then, by our choice of $Z$, the coloured vertices of $\mathcal{G}$ together with $Z$ form a colour system which looks the same as the extended core $\mathcal{C}$ of $\mathcal{G}$ except that the uncoloured vertex of $\mathcal{C}$ is missing.

Now, for the rest of the proof, we condition on some outcome of the induced subgraph $G_0[Z]$ on the vertices in $Z$. To apply \cref{lem-nu-close-Gamma}, we define a restricted colour system $\mathcal{G}'$ of order $n$ with parameters $(g, a_1,\dots,a_{g-1}, 2^{a_1t_1+\dots+a_{g-1}t_{g-1}},t_1,\dots,t_{g-1})$ by starting with $\mathcal G$, colouring the vertices in $Z$ with colour $g$, and including all edges of our conditioned outcome of $G_0[Z]$ in a single shade of colour $g$. By construction, the core $\mathcal{C}'$ of the restricted colour system $\mathcal{G}'$ is almost isomorphic to $\mathcal{C}$; the only difference is that $\mathcal{C}'$ has the edges of $G_0[Z]$ between the vertices of colour $g$, whereas $\mathcal{C}$ has no edges of colour $g$. (To be precise, there is a colour/shade-preserving graph isomorphism between $\mathcal{C}$ and $\mathcal{C}'-E(G_0[Z])$).

Since we chose $Z$ such that $v\in Z$ and $u\not\in Z$, we have $u\in \operatorname{U}(\mathcal{G}')$ (that is, $u$ is uncoloured in $\mathcal G'$), and $v$ has colour $g$ in $\mathcal{G}'$. Recall that $u\in U_*$, meaning that $u$ is connected to all vertices of colours $1,\dots,g-1$ in all possible shades of these colours. This implies $u\in \operatorname{U^*}(\mathcal{G}')$. Now, let $e'\in E^g(\mathcal C')$ be the unique edge in the core $\mathcal{C}'$ between $v$ and the uncoloured vertex of $\mathcal{C}'$. As $v\in U_e$, this edge $e'$ corresponds to the edge $e$ in $\mathcal C$ under the isomorphism in the previous paragraph. Note that the functions $\Gamma_{\mathcal{C},e}$ do not actually depend on the edges between the vertices of colour $g$ in $\mathcal{C}$ (since the partial copies of $H$ in $\mathcal{C}$ use exactly one vertex of colour $g$). Thus, we have $\Gamma_{\mathcal{C},e}=\Gamma_{\mathcal{C}',e'}$.

We have conditioned on an outcome of $G_0[Z]$. For each $z\in Z$, let $S_z=N_{G_0}(z)\cap\operatorname{U}(\mathcal{G}')$ be the (random) set of neighbours of $z$ in $\operatorname{U}(\mathcal{G}')=\operatorname{U}(\mathcal{G})\sm Z$, and let $G_0^-=G_0[\operatorname{U}(\mathcal{G}')]\sim \GG(\operatorname{U}(\mathcal{G}'), p)$ be the induced subgraph on $\operatorname{U}(\mathcal{G}')$. Then, with $\mathcal{S}=(S_z)_{z\in Z}$ we have $\mathcal{G}( G_0, j_1,\dots,j_{g-1})=\mathcal{G}'_{\mathcal{S}}(G_0^-, j_1,\dots,j_g,1)$ in the notation of \cref{defn:extension-restricted-colour-system}. It follows that $\kappa_H(\mathcal{G}, G_0, j_1,\dots,j_{g-1}, u,v)=\kappa_H(\mathcal{G}'_{\mathcal{S}}, G_0^-, j_1,\dots,j_{g-1}, 1,u,v)$, because every labelled copy of $H$ using the edge $uv$ as well as at least one vertex of each of the colours $1,\dots,g-1$ automatically also uses a vertex of colour $g$ (namely $v$). Thus, \cref{eqn:kappa-close-Gamma-to-prove} is equivalent to the inequality
\begin{equation}\label{eqn:kappa-close-Gamma-to-prove-2}
\left\vert\kappa_H(\mathcal{G}'_{\mathcal{S}}, G_0^-, j_1,\dots,j_{g-1}, 1,u,v)-n^{h-g-1}\Gamma_{\mathcal{C}',e'}(j_1,\dots,j_{g-1},1)\right\vert= O\left( \log n \cdot n^{h-g-(3/2)}\right),
\end{equation}
where $G_0^-\sim\GG(\operatorname{U}(\mathcal{G}'), p)$, and $\mathcal{S}=(S_z)_{z\in Z}$ is a collection of random sets with respect to the restricted colour system $\mathcal{G}'$ as in \cref{defn:extension-restricted-colour-system}.

Note that $\mathcal G'$ is essentially weakly $p$-general, due to the way it was defined in terms of $\mathcal G$. So, by \cref{lem:extension-p-general}, $\mathcal{G}'_{\mathcal{S}}$ is $p$-general with probability at least $1-n^{-\omega(1)}$, and if $\mathcal{G}'_{\mathcal{S}}$ is $p$-general then by \cref{lem-nu-close-Gamma} we have $\Gamma_{\mathcal{C}',e'}(j_1,\dots,j_{g-1},1)n^{h-g-1}=\nu_{\mathcal{G}', \mathcal{S},u,v}(j_1,\dots,j_{g-1},1)+O(\log n\cdot n^{h-g-(3/2)})$. So, to conclude the proof of \cref{lem:psi-u-v-close-Gamma}, it suffices to prove the following claim.

\begin{claim}\label{lem:kappa-close-nu}
With probability $1-n^{-\omega(1)}$ we have
\begin{equation}\label{eqn:ineq-kappa-nu-far}
\left\vert\kappa_H(\mathcal{G}'_{\mathcal{S}}, G_0^-, j_1,\dots,j_{g-1}, 1,u,v)-\nu_{\mathcal{G}', \mathcal{S},u,v}(j_1,\dots,j_{g-1},1)\right\vert\le \log n\cdot n^{h-g-(3/2)}.
\end{equation}
\end{claim}

\begin{proof}
Condition on a fixed outcome of $\mathcal S$ (so that only $G_0^-$ remains random). Recall that by \cref{defn:nu} we have $\nu_{\mathcal{G}', \mathcal{S},u,v}(j_1,\dots,j_{g-1},1)=\E_{G_0^-}[\kappa_H(\mathcal{G_S'}, G_0^-, j_1,\dots,j_{g-1},1, u,v)]$. Consider the vertex-exposure martingale for $G_0^-$ (with respect to $\kappa_H(\mathcal{G}'_{\mathcal{S}}, G_0^-, j_1,\dots,j_{g-1}, 1,u,v)$), where we fix an ordering of $\operatorname{U}(\mathcal{G}')$ (ending with $u$) and at each step we consider the next vertex in our ordering and expose all the edges of $G_0^-$ incident to that vertex which have not yet been exposed. Changing the status of edges adjacent to a single vertex in $\operatorname{U}(\mathcal{G}')\sm \lbrace u\rbrace$ changes the value of $\kappa_H(\mathcal{G}'_{\mathcal{S}}, G_0^-, j_1,\dots,j_{g-1}, 1,u,v)$ by at most $h^{g+2} \cdot a_1\dotsm a_{g-1}\cdot 2^{a_1t_1+\dots+a_{g-1}t_{g-1}}\cdot n^{h-g-2}=O(n^{h-g-2})$. This is due to the fact that there can be at most $h^{g+2} \cdot a_1\dotsm a_{g-1}\cdot 2^{a_1t_1+\dots+a_{g-1}t_{g-1}}\cdot n^{h-g-2}$ different labelled copies of $H$ in the $n$-vertex graph $\mathcal{G_S}(G_0^-, j_1,\dots,j_g)+\lbrace uv\rbrace$ which use the edge $uv$ as well as at least one vertex of each of the $g$ colours and the exposed vertex (as both $u$ and the exposed vertex are uncoloured). Thus, by the Azuma--Hoeffding inequality the probability that \cref{eqn:ineq-kappa-nu-far} fails to hold is at most
\[\exp\left(-\Omega\left(\frac{(\log n)^{2}\cdot n^{2(h-g-3/2)}}{n \cdot n^{2(h-g-2)}}\right)\right)= \exp\left(-\Omega\left( (\log n)^{2}\right)\right)= n^{-\omega(1)},\]
as desired.
\end{proof}

\subsection{Putting everything together}
\label{sec:final-pf}

\begin{proof}[Proof of \cref{lem:rough-scale}]
Recall that the statement of \cref{lem:rough-scale} is for fixed integers $1\leq g\leq h-1$ and $a_1,\dots,a_{g-1},t_1,\dots,t_{g-1}\geq 1$ (which were fixed throughout \cref{sect:cor-implies-thm}), and recall that $T=t_1\dotsm t_{g-1}$. Let $\mathcal{G}$ and $\lambda$ be as in the statement of the lemma, and let $\mathcal{C}$ be the extended core of $\mathcal{G}$ (which is a complete core). Let $N=\binom{|\operatorname{U}(\mathcal{G})|}{2}=\Theta(n^2)$, so that the random choice of $G_0\sim\GG(\operatorname{U}(\mathcal{G}), p)$ can be encoded by a Bernoulli sequence $\x\sim \Ber(p)^{N}$, with one random bit for each of the $N$ possible edges of $G_0$. Abusing notation slightly, we identify the integers $1,\dots,N$ with pairs of vertices in $\operatorname{U}(\mathcal{G})$, so that we may write $\xi_{\{u,v\}}$ to indicate the random bit that encodes the presence of the edge $\{u,v\}$.

Now, abusing notation, we index the coordinates of $\RR^T$ by tuples in $[t_1]\times\dots,\times[t_{g-1}]$ (so that we may talk about the $(j_1,\dots,j_{g-1})$-coordinate of a vector in $\RR^T$). Let $\boldsymbol{f}:\{0,1\}^N\to \RR^T$ be the vector-valued function defined such that, for $\x\in \{0,1\}^N$ corresponding to a graph $G_0$, the $(j_1,\dots,j_{g-1})$-coordinate of $\boldsymbol f(\x)$ is $\psi_H(\mathcal{G}, G_0, j_1,\dots,j_{g-1})$. With this definition, and the notation of \cref{thm:rough-halasz}, the random vector $\Delta_{\{u,v\}} \boldsymbol{f}(\boldsymbol \xi)$ corresponds to the function $\kappa_H(\mathcal G,G_0,\cdot,1,u,v)$.

The plan is to now apply \cref{thm:rough-halasz} with $d=T$ and $m=2^{a_1t_1+\dots+a_{g-1}t_{g-1}}$ and with the $\Gamma_{\mathcal C,e}$ taking the role of the vectors $\boldsymbol{v}_1,\dots,\boldsymbol{v}_m$. For each edge $e\in E^g(\mathcal C)$, let $\boldsymbol{\gamma}_e\in \RR^T$ be the vector corresponding to the function $\Gamma_{\mathcal{C},e}(\cdot,1)$, and let $\boldsymbol{x}\in \RR^T$ be the vector corresponding to the function $\lambda(\cdot,1)$. By \cref{lem:U-e-big}, there is some $\varepsilon=\Omega(1)$ such that for each edge $e\in E^g(\mathcal C)$, there are $\Omega(n^2)\ge \varepsilon N$ pairs of vertices $\{u,v\}$ with $u\in U_*,v\in U_e$. Let $I_e$ be the set of these pairs $\{u,v\}$ (observe that all the $I_e$ are disjoint). By \cref{lem:psi-u-v-close-Gamma}, for each $\{u,v\}\in I_e$ we have $\Pr\left(\left\Vert \Delta_{\{u,v\}}\boldsymbol{f}\left(\x\right)-s\boldsymbol{\gamma}_{e}\right\Vert_{\infty}\ge r\right)\le n^{-\omega(1)}$ for $s=n^{h-g-1}$ and some $r=\Theta(n^{h-g-3/2}\log n)$. Note that $r\sqrt {N\log N}\ge s$, and by \cref{lem:the-Gamma-e-span}, the vectors $\boldsymbol{\gamma}_{e}$ span $\RR^T$. We can now apply \cref{thm:rough-halasz} to obtain
\[
\Pr\left(\left\Vert \boldsymbol{f}\left(\x\right)-\boldsymbol{x}\right\Vert _{\infty}<r\sqrt{N\log N}\right) = O\left(\left(\frac{r\sqrt{\log N}}{s}\right)^{T}\right)=O\left((\log n)^{3T/2}n^{-T/2}\right)\le n^{-T/2+o(1)}.
\]
(The implied constants in the above asymptotic notation may \emph{a priori} depend on $\mathcal C$, but note that there are only finitely many possibilities for a core with parameters $(g,a_1,\dots,2^{a_1t_1+\dots+a_{g-1}t_{g-1}},t_1,\dots,t_{g-1})$). Finally, to conclude the proof we recall that $\left\|\psi_H(\mathcal{G}, G_0, \cdot)-\lambda\right\|_\infty=\left\Vert \boldsymbol{f}\left(\x\right)-\boldsymbol{x}\right\Vert _{\infty}$ and observe that $r\sqrt{N\log N}=\Theta(n^{h-g-(1/2)}(\log n)^{3/2})\ge n^{h-g-(1/2)}\log n$ for large $n$.
\end{proof}

\begin{proof}[Proof of \cref{lem:medium-scale}]
This proof is very similar to the proof of \cref{lem:rough-scale}. Recall that the statement of \cref{lem:medium-scale} is for fixed integers $1\leq g\leq h-1$ and $a_1,\dots,a_{g},t_1,\dots,t_{g-1}\geq 1$ (which were fixed throughout \cref{sect:thm-implies-cor}), and recall that $T=t_1\dotsm t_{g-1}$. Let $\mathcal{G}$ and $\lambda$ be as in the statement of the lemma, and let $\mathcal{C}$ be the core of $\mathcal{G}$. Let $N=a_g\cdot |\operatorname{U}(\mathcal{G})|=\Theta(n)$, so that a random choice of $\mathcal{S}$ as in \cref{defn:extension-restricted-colour-system} can be encoded by a Bernoulli sequence $\x\in \Ber(p)^{N}$, with one random bit for each potential element in each $S_v\in \mathcal S$. Abusing notation slightly, we identify $1,\dots,N$ with ordered pairs of vertices: for $u\in \operatorname U(\mathcal G)$ and a vertex $v$ of colour $g$ we write $\xi_{(u,v)}$ for the random bit that encodes the presence of $u$ in $S_v$.

Let $\boldsymbol{f}:\{0,1\}^N\to \RR^T$ be the vector-valued function (with coordinates indexed by $[t_1]\times \dots\times [t_{g-1}]$) defined such that, for $\x\in \{0,1\}^N$ corresponding to an outcome of $\mathcal{S}$, the $(j_1,\dots,j_{g-1})$-coordinate of $\boldsymbol f(\x)$ is $\mu_{\mathcal G,\mathcal S}(j_1,\dots,j_{g-1},1)$. Then, $\Delta_{(u,v)} \boldsymbol{f}(\boldsymbol \xi)$ corresponds to the function $\nu_{
\mathcal G,\mathcal S,u,v}(\cdot,1)$. For each $e\in E^g(\mathcal C)$, let $\boldsymbol{\gamma}_e\in \RR^T$ be the vector corresponding to $\Gamma_{\mathcal{C},e}(\cdot,1)$, and let $\boldsymbol{x}\in \RR^T$ be the vector corresponding to $\lambda(\cdot,1)$.

By \cref{lem:U-star-big}, there is some $\varepsilon=\Omega(1)$ such that $|\operatorname{U^*}(\mathcal{G})|\ge \varepsilon n$. For each edge $e\in E^g(\mathcal C)$ between the uncoloured vertex of $\mathcal C$ and some vertex $v$ of colour $g$, let $I_e=\operatorname{U^*}(\mathcal{G})\times \{v\}$. By \cref{lem:extension-p-general} the colour system $\mathcal{G_S}$ is $p$-general with probability $1-n^{-\omega(1)}$, in which case, by \cref{lem-nu-close-Gamma}, for each $(u,v)\in I_e$ we have $\left\Vert \Delta_{(u,v)}\boldsymbol{f}\left(\x\right)-s\boldsymbol{\gamma}_{e}\right\Vert_{\infty}\le r$ for $s=n^{h-g-1}$ and some $r=\Theta(n^{h-g-(3/2)}\log n)$. Also, by \cref{lem:the-Gamma-e-span}, the vectors $\boldsymbol{\gamma}_{e}$ span $\RR^T$.

We can now apply \cref{thm:rough-halasz} to obtain $\Pr\left(\left\Vert \boldsymbol{f}\left(\x\right)-\boldsymbol{x}\right\Vert _{\infty}<r\sqrt{N\log N}\right)\le n^{-T/2+o(1)}.$ Finally, to conclude the proof we observe that $r\sqrt{N\log N}=\Theta(n^{h-g-1}(\log n)^{3/2})\ge n^{h-g-1}\log n$ for large $n$, and recall that $\left\Vert\mu_{\mathcal{G}, \mathcal{S}}-\lambda\right\Vert_\infty=\left\Vert \boldsymbol{f}\left(\x\right)-\boldsymbol{x}\right\Vert_{\infty}$.
\end{proof}

\section{Concluding remarks}
We have proved that for connected $H$ and constant $p\in (0,1)$, we have $\max_x\Pr(X_H=x)\le n^{1-v(H)+o(1)}$. There are several interesting directions for future research. Most obviously, \cref{conj:general} remains open: for connected $H$ we are still a factor of $n^{o(1)}$ away from an optimal bound, and for disconnected $H$ we do not even have a bound that improves as $H$ grows (the best general bound is $\Pr\left(X_{H}=x\right)\le\Pr\left(\left|X_{H}-x\right|\le n^{v\left(H\right)-2}\right)=O\left(1/n\right)$, as we mentioned in the introduction). It seems that the ideas in this paper are robust enough to give certain nontrivial bounds (in terms of the size of the largest component of $H$) even in the disconnected case, but we have not explored this further.

For certain graphs $H$, a possible route to a proof of \cref{conj:general} might be via a local central limit theorem, which one might hope to prove by extending the methods of Gilmer and Kopparty~\cite{GK16}, and Berkowitz~\cite{Ber16,Ber18}. Basically, this involves carefully estimating the characteristic function $\varphi(t)=\E e^{itX_H}$, using different arguments for different ranges of $t$. We remark that $\varphi(1/k)$ is small if the distribution of $X_H$ is not too biased mod $k$, which seems comparable to anticoncentration of $X_H$ at ``scale'' $k$. So, we wonder whether the ideas in this paper might be helpful for estimating $\varphi$: recall that our argument proceeds by breaking up $X_H$ into a sum of many random variables that fluctuate at different scales. However, we emphasise that local central limit theorems do not seem to be the right path to a proof of \cref{conj:general} in its full generality: for example, if
$H$ is a disjoint union of an edge and a 2-edge path, then the probability
that $X_{H}$ is odd is substantially different from the probability
that it is even (see \cite{DR16}), meaning that $X_H$ does not obey a local central limit theorem.

Also, let $X_H^{\mathrm{hom}}$ be the number of (possibly non-injective) homomorphisms from $H$ into $G\sim \GG(n,p)$. This random variable is very closely related to $X_H$, and we remark that with very minimal changes, one can modify our proof of \cref{thm:subgraph-counts} to prove the corresponding theorem for $X_H^{\mathrm{hom}}$, when $H$ is connected. Interestingly, the homomorphism-counting analogue of \cref{conj:general} fails dramatically in general: if $H$ is the disjoint union of two copies of a graph $H'$, then $X_H^{\mathrm{hom}}=(X_{H'}^{\mathrm{hom}})^2$, meaning that $X_H$ has the same point probabilities as $X_{H'}$. This means that any proof of \cref{conj:general} must be sensitive to the difference between subgraph counts and homomorphism counts.

It would also be interesting to consider the ``sparse'' regime where $p$ is allowed to decay with $n$. For example, it is known that if $H$ is \emph{strictly balanced} (see for example \cite[Section~3.2]{JLR00}) and $p=p(n)$ is such that $\Var X_H=o((\E X_H)^2)$, then $\Pr(X_H>0)=1-o(1)$. Could it be that under these conditions we also have that $\max_{x\in \ZZ}\Pr(X_H=x)=O(1/\sqrt{\Var X_H})$?

As mentioned in \cite{FKS1} it may also be interesting to study anticoncentration of the number of \emph{induced} copies
$X_{H}'$ of a subgraph $H$ in a random graph $\GG\left(n,p\right)$. (This question was also raised by Meka, Nguyen and Vu~\cite{MNV16}). The natural analogue of \cref{conj:general} is that for a fixed graph $H$ and fixed $p\in\left(0,1\right)$, we have
\[
\max_{x\in \NN}\Pr\left(X_{H}'=x\right)=O\left(\frac{1}{\sqrt{\Var\left(X_{H}'\right)}}\right).
\]
We remark that the behaviour of $\sqrt{\Var\left(X_{H}'\right)}$ is not entirely trivial: for most values of $p$ it has order $\Theta(n^{h-1})$, but when $p$ is exactly equal to the edge-density of $H$ it may have order $\Theta(n^{h-3/2})$ or $\Theta(n^{h-2})$ (see \cite[Theorem~6.42]{JLR00}).

Finally, it would be interesting to prove similar anticoncentration results in other combinatorial settings. One important example is random subsets of the integers (or other groups): for instance, what can we say about anticoncentration of the number of $k$-term arithmetic progressions in a random subset of $\{1,\dots,n\}$? Arithmetic configuration counts have been an interesting analogue to subgraph counts in a number of other settings, for example in the study of large deviations (both fall in the framework of \emph{nonlinear large deviations} initiated by Chatterjee and Dembo~\cite{CD16}; see for example \cite{HMS19} and the references therein). Another interesting direction of research would be to consider subgraph counts in random $k$-uniform hypergraphs, or for other random graph models (for example, the uniform distribution $\GG(n,m)$ on graphs with a fixed set of $n$ vertices and exactly $m$ edges).

\textit{Remark added in proof.} While this paper was under review, Sah and Sahwney~\cite{SS20} proved \cref{conj:general} for connected $H$ (via a local limit theorem) and disproved \cref{conj:general} in general.

\textit{Acknowledgements.} We thank the referees for their careful reading, and for many helpful comments.


\end{document}